\documentclass[twoside,11pt]{article}
\usepackage{jmlr2e}
\usepackage{mathrsfs,amsmath,amsfonts,amssymb,bm,bbm,eufrak,dsfont,pifont,amscd,
stmaryrd,euscript,appendix,color,epsfig,url}
\usepackage[breaklinks=true]{hyperref}
\usepackage{breakcites}
\definecolor{darkblue}{rgb}{0,0.08,0.45}
\usepackage{soul}
\hypersetup{  
  colorlinks = true, 
  linkcolor  = darkblue, 
  citecolor  = darkblue,
  filecolor  = darkblue, 
  urlcolor   = darkblue
} 
\usepackage{graphicx}
\allowdisplaybreaks
\newtheorem{appxthm}{Theorem}[section]
\newtheorem{appxlem}[appxthm]{Lemma}

\newtheorem{appxpro}[appxthm]{Proposition}

\newcommand{\QEDA}{\hfill\ensuremath{\blacksquare}}

\newcommand{\bb}{\mathbb}
\newcommand{\eu}{\EuScript}
\newcommand{\Scr}{\mathscr}
\newcommand{\Cal}{\mathcal}

  \newcommand{\lv}{\left\Vert}
\newcommand{\rv}{\right\Vert}

\newcommand{\R}{{\mathbb{R}}}

\newcommand{\hxi}{\hat{\xi}}

\firstpageno{1}

\begin{document}
\title{Density Estimation in Infinite Dimensional
Exponential Families}
 
\author{\name Bharath Sriperumbudur \email bks18@psu.edu \\ 
       \addr Department of Statistics, Pennsylvania State University \\
       University Park, PA 16802, USA.
       \AND
       \name Kenji Fukumizu \email fukumizu@ism.ac.jp \\ 
       \addr The Institute of Statistical Mathematics \\ 
       10-3 Midoricho, Tachikawa, Tokyo 190-8562 Japan.              
       \AND 
       \name Arthur Gretton \email arthur.gretton@gmail.com \\
       \addr Gatsby Computational Neuroscience Unit, University College London \\
       Sainsbury Wellcome Centre, 25 Howland Street, London W1T 4JG, UK
       \AND
       \name Aapo Hyv\"{a}rinen \email aapo.hyvarinen@helsinki.fi \\
       \addr Department of Computer Science, University of Helsinki \\
       P.O.~Box 68, FIN-00014, Finland.
       \AND
       \name Revant Kumar \email rkumar74@gatech.edu\\
       \addr College of Computing, Georgia Institute of Technology\\
       801 Atlantic Drive, Atlanta, GA 30332, USA.
       }


\maketitle

\begin{abstract}
In this paper, we consider an infinite dimensional 
exponential family $\Cal{P}$ of probability densities, which 
are parametrized by functions in a reproducing
kernel Hilbert space $\eu{H}$, and show it to be quite rich in
the sense that a broad class of densities on 
$\bb{R}^d$ 
can be approximated arbitrarily well in Kullback-Leibler
(KL) divergence 
by
elements in $\Cal{P}$. Motivated by this approximation property, the paper addresses the 
question of estimating an unknown density $p_0$ through an element in $\Cal{P}$.
Standard techniques like maximum likelihood estimation
(MLE) or pseudo MLE (based on the method of sieves), which are based on
minimizing the KL divergence between $p_0$ and $\Cal{P}$, do not yield
practically
useful estimators because of their inability to efficiently handle the
log-partition
function. We propose an estimator $\hat{p}_n$ based on minimizing
the \emph{Fisher divergence}, $J(p_0\Vert p)$ between $p_0$ and $p\in \Cal{P}$,
which
involves solving a simple finite-dimensional linear
system. When $p_0\in\Cal{P}$, we show that the proposed estimator is consistent,
and provide a convergence rate of
$n^{-\min\left\{\frac{2}{3},\frac{2\beta+1}{2\beta+2}\right\}}$ in Fisher
divergence
under the smoothness assumption that $\log p_0\in\Cal{R}(C^\beta)$
for some $\beta\ge 0$, where $C$ is a certain Hilbert-Schmidt
operator on $\eu{H}$ and
$\Cal{R}(C^\beta)$ denotes the image of $C^\beta$. We also
investigate the
misspecified case of $p_0\notin\Cal{P}$ and show that
$J(p_0\Vert\hat{p}_n)\rightarrow \inf_{p\in\Cal{P}}J(p_0\Vert p)$ as
$n\rightarrow
\infty$, and provide a rate for this
convergence under a similar smoothness condition as above.
Through numerical simulations we demonstrate that
the proposed estimator outperforms the non-parametric kernel density
estimator, and that the advantage of the proposed estimator grows as $d$ increases.
\end{abstract}

\setlength{\parskip}{4pt}

\section{Introduction}\label{Sec:Introduction}\vspace{0mm}
Exponential families are among the most important classes of parametric models
studied in statistics, and include many common distributions such as the normal, exponential, gamma, and Poisson. 
In its ``natural
form'', the family generated by a probability density $q_0$ (defined
over $\Omega\subseteq\bb{R}^d$) and \emph{sufficient
statistic}, $T:\Omega\rightarrow\bb{R}^m$ is defined as
\begin{equation}\Scr{P}_{\text{fin}}:=\left\{p_\theta(x)=q_0(x)e^{\theta^T
T(x)-A(\theta)},\,x\in\Omega\,:\,\theta\in\Theta\subset\bb{R}^m\right\}
\label{Eq:finite-exp}\end{equation} where $A(\theta):=\log \int_\Omega
e^{\theta^T
T(x)}q_0(x)\,dx$ is the cumulant generating function (also called the
log-partition function), $\Theta\subset \{\theta\in\bb{R}^m : A(\theta)<\infty\}$ is the \emph{natural parameter space} and $\theta$ is a finite-dimensional vector called the
\emph{natural parameter}. 
Exponential families have a  number of properties that make them extremely useful for statistical analysis (see \citealp*{Brown-86} for more details).

In this paper, we consider an infinite dimensional generalization
\citep{Canu-05,Fukumizu-09a} of (\ref{Eq:finite-exp}), 
\begin{equation}
 \Cal{P}=\left\{
p_f(x)=e^{f(x)-A(f)}q_0(x),\,x\in\Omega:f\in\Cal{F}\right\},\nonumber
\end{equation}
where the function space
$\Cal{F}$ is defined as $$\Cal{F}=\left\{f\in\eu{H}:e^{A(f)}<\infty\right\},\,\,\,\text{with}\,\,\,A(f):=\log\int_\Omega
e^{f(x)}q_0(x)\,dx$$ 
being the cumulant generating function, and $(\eu{H},\langle
\cdot,\cdot\rangle_\eu{H})$ a reproducing
kernel
Hilbert space (RKHS) \citep{Aronszajn-50} with $k$ as its reproducing
kernel. 
While various generalizations are possible for different choices of
$\Cal{F}$ (e.g., an Orlicz space as in \citealp*{Pistone-95}), 
the connection of $\Cal{P}$ to the
natural exponential family in (\ref{Eq:finite-exp}) is particularly
enlightening when $\eu{H}$ is an RKHS. This is due to the reproducing property
of the kernel, $f(x)=\langle f,
k(x,\cdot)\rangle_\eu{H}$, 
through which 
$k(x,\cdot)$
takes the role of the sufficient statistic. In fact, it can be shown
(see Section~\ref{Sec:approximation}
and
Example~\ref{exm:finite-dim} for more details) that
every $\Scr{P}_{\text{fin}}$ is generated by $\Cal{P}$ induced
by a finite dimensional RKHS $\eu{H}$, and therefore the family $\Cal{P}$ with
$\eu{H}$ being an infinite dimensional RKHS 
is a natural infinite dimensional
generalization of $\Scr{P}_{\text{fin}}$. Furthermore, this generalization is particularly interesting as in contrast to $\Scr{P}_{\text{fin}}$,
it can be shown that $\Cal{P}$ is a rich class of
densities (depending on the choice of $k$ and therefore
$\eu{H}$) that can approximate a broad class of probability densities
arbitrarily well (see
Propositions~\ref{Thm:approx}, \ref{Thm:approx-fd} and Corollary~\ref{cor:approx}). This
generalization is not only of theoretical interest, but also has implications for 
statistical and machine learning applications. For example, in Bayesian non-parametric density estimation, 
the densities in $\Cal{P}$ are chosen as prior distributions on a collection of probability
densities (e.g., see \citealp*{Vaart-08}). $\Cal{P}$ has also found applications in nonparametric hypothesis testing
\citep{Gretton-12,Fukumizu-08a} and dimensionality
reduction~\citep{Fukumizu-04,Fukumizu-09} through the mean
and covariance operators, which are obtained as the first and second Fr\'{e}chet derivatives of $A(f)$ \cite[see][Section
1.2.3]{Fukumizu-09a}. Recently, the infinite dimensional exponential family, $\Cal{P}$ has been used to develop a gradient-free adaptive MCMC algorithm 
based on Hamiltonian Monte Carlo \citep{Heiko-15} and also has been used in the context of learning the structure of graphical models \citep{Sun-15}.
 
\par Motivated by the richness of the infinite dimensional
generalization and its statistical applications, it is of interest to model densities by $\Cal{P}$, and therefore the goal of this paper is to estimate
unknown densities by elements in $\Cal{P}$ 
when $\eu{H}$ is an infinite dimensional RKHS. 
Formally, given i.i.d.~random samples $(X_a)^n_{a=1}$ drawn from an unknown
density $p_0$, the goal is to estimate $p_0$ through $\Cal{P}$. Throughout the
paper, we refer to case of $p_0\in\Cal{P}$ as \emph{well-specified}, in contrast
to the \emph{misspecified} case where $p_0\notin\Cal{P}$. The setting is
useful because $\Cal{P}$ is a rich class of
densities that can approximate a broad class of probability densities
arbitrarily well, 
hence it may be widely used in place of non-parametric density
estimation methods (e.g., kernel density estimation (KDE)). 
In fact, through numerical simulations, we show in Section~\ref{Sec:experiments} that estimating $p_0$
through $\Cal{P}$ performs better than KDE, and that the advantage of the proposed estimator grows with increasing dimensionality.

In
the finite-dimensional case where $\theta\in\Theta\subset\bb{R}^m$,
estimating $p_\theta$ through maximum
likelihood (ML) leads to solving elegant likelihood equations \cite[Chapter
5]{Brown-86}. However,
in the infinite dimensional case (assuming $p_0\in\Cal{P}$), as in many
non-parametric estimation methods,
a straightforward extension of maximum likelihood estimation (MLE) suffers from
the problem of ill-posedness \cite[Section 1.3.1]{Fukumizu-09a}. To
address this problem, \cite{Fukumizu-09a} proposed a  method of sieves
involving pseudo-MLE by restricting the infinite
dimensional
manifold $\Cal{P}$ to a series of finite-dimensional submanifolds, which enlarge
as the
sample size increases, i.e., $p_{\hat{f}^{(l)}}$ is the density estimator with \vspace{-1mm}
\begin{equation}\hat{f}^{(l)}=\arg\max_{f\in\Cal{F}^{(l)}}\frac{1}{n}\sum^n_{a=1}f(X_a)-A(f),\label{Eq:ml}\vspace{-1mm}\end{equation}
where
$\Cal{F}^{(l)}=\{f\in\eu{H}^{(l)}:e^{A(f)}<\infty\}$ and
$(\eu{H}^{(l)})^\infty_{l=1}$ is a sequence of finite-dimensional
subspaces of $\eu{H}$ such that $\eu{H}^{(l)}\subset\eu{H}^{(l+1)}$ for all
$l\in\bb{N}$. While the consistency of $p_{\hat{f}^{(l)}}$ is proved in
Kullback-Leibler (KL) divergence \cite[Theorem 6]{Fukumizu-09a}, the method
suffers from many drawbacks that are both theoretical and computational in
nature. On the theoretical front, the consistency in \citet[Theorem
6]{Fukumizu-09a} is established by assuming a decay rate on the eigenvalues of the covariance operator (see (A-2) and the discussion in Section
1.4 of \cite{Fukumizu-09a} for details), which is usually difficult to
check in practice. Moreover, it is
not clear which classes of RKHS should be used to obtain a consistent
estimator \cite[(A-1)]{Fukumizu-09a} and the paper does not
provide any discussion about the convergence rates. On the practical side, the
estimator is not attractive as it can be quite difficult to construct the
sequence $(\eu{H}^{(l)})^\infty_{l=1}$ that satisfies the assumptions in
\citet[Theorem 6]{Fukumizu-09a}. In fact, the impracticality of the estimator,
$\hat{f}^{(l)}$ is accentuated by the difficulty in efficiently handling $A(f)$
(though it can be approximated by numerical integration). 

A related work was
carried out by \citet{Barron-91}---also see references
therein---where the goal is to estimate a density, $p_0$ by approximating its logarithm
as an expansion in terms of basis functions, 
such as polynomials, splines or
trigonometric series. 
Similar to \cite{Fukumizu-09a},
Barron and Sheu proposed the ML estimator $p_{\hat{f}_m}$, where
$$\hat{f}_m=\arg\max_{f\in \Cal{F}_m}\frac{1}{n}\sum^n_{a=1}f(X_a)-A(f)$$ and
$\Cal{F}_m$ is the linear space of dimension $m$ spanned by
the chosen basis functions.
Under the assumption that $\log p_0$ has square-integrable derivatives up to order $r$, they showed
that
$KL(p_0\Vert p_{\hat{f}_m})= O_{p_0}(n^{-2r/(2r+1)})$ with
$m=n^{1/(2r+1)}$ 
for each of the approximating families, where $KL(p\Vert q)=\int
p(x)\log(p(x)/q(x))\,dx$ is the
KL divergence between $p$ and $q$. Similar work was carried out by \citet{Gu-93}, who assumed that
$\log p_0$ lies in an RKHS, and proposed an estimator based on penalized MLE, with consistency and rates
established in Jensen-Shannon divergence. Though these
results are theoretically interesting, these estimators are obtained via a
procedure similar to that in \cite{Fukumizu-09a}, and therefore
suffers from the practical drawbacks discussed above. 

The discussion so far shows that the MLE
approach to learning $p_0\in\Cal{P}$ results in estimators that are of limited practical interest. 
To alleviate this, one can treat
the problem of
estimating $p_0\in\Cal{P}$ in a completely non-parametric fashion by using KDE,
which is well-studied \cite[Chapter 1]{Tsybakov-09} and easy to implement. This
approach ignores the structure
of $\Cal{P}$, however, and is known to perform poorly for moderate to large $d$
\cite[Section 6.5]{Wasserman-06} (see also Section~\ref{Sec:experiments} of
this paper).
\subsection{Score Matching and Fisher Divergence}\label{subsec:fisher} 
To counter the disadvantages
of KDE and pseudo/penalized-MLE, in this paper, we
propose to use the \emph{score matching method} introduced by
\citet{Hyvarinen-05, Hyvarinen-07}. 
While MLE is based on minimizing the KL divergence, the
score matching method involves minimizing the \emph{Fisher
divergence} (also called the Fisher information distance; see Definition
1.13 in \cite{Johnson-04}) between two continuously differentiable densities,
$p$ and $q$ on
an open set $\Omega\subseteq\bb{R}^d$, given as
\begin{equation}J(p\Vert q)=\frac{1}{2}\int_\Omega p(x) \left\Vert \nabla\log p(x)-\nabla \log q(x)
\right\Vert^2_2\,dx,\label{Eq:fisher}\end{equation}
where $\nabla\log p(x)=
\left(\partial_1\log p(x),\ldots,\partial_d\log p(x)\right)$ with $\partial_i\log p(x):=\frac{\partial}{\partial x_i}\log
p(x)$. 
Fisher
divergence
 is
closely related to the KL divergence through de
Bruijn's identity \cite[Appendix C]{Johnson-04} and it can be shown that 
$KL(p\Vert q)=\int^\infty_0
J(p_t\Vert q_t)\,dt,$ 
where $p_t=p\ast N(0,t I_d)$, $q_t=q\ast N(0,t I_d)$, $\ast$ denotes the
convolution, and $N(0,t I_d)$ denotes a normal distribution on $\bb{R}^d$ with
mean zero and diagonal covariance with $t>0$ (see
Proposition~\ref{pro:supp-kl-J} 
for a precise statement; also
see Theorem 1 in \citealp*{Lyu-09}). Moreover, convergence in Fisher divergence is a 
stronger form of convergence than that in KL, total variation and Hellinger distances 
(see Lemmas E.2 \& E.3 in \citealp*{Johnson-04} and Corollary
5.1 in \citealp*{Ley-13}).

To understand
the advantages associated with the score matching method, let us consider the
problem of density estimation where the data
generating distribution (say $p_0$) belongs to $\Scr{P}_{\text{fin}}$ in (\ref{Eq:finite-exp}). 
In other words, given
random samples $(X_a)^n_{a=1}$ drawn i.i.d.~from $p_0:=p_{\theta_0}$,
the goal is to estimate $\theta_0$ as $\hat{\theta}_n$, and use
$p_{\hat{\theta}_n}$ as an estimator of $p_0$. While the MLE approach is
well-studied and enjoys nice statistical properties in asymptopia (i.e.,
asymptotically unbiased, efficient, and normally distributed), the computation
of $\hat{\theta}_n$ can be intractable in many situations as discussed
above. In
particular,
this is the case for $p_\theta(x)=\frac{r_\theta(x)}{A(\theta)}$ where
$r_\theta\ge 0$ for all
$\theta\in\Theta$, $A(\theta)=\int_\Omega r_\theta(x)\,dx$, and the functional
form of $r$ is known (as a function of $\theta$ and $x$); yet we do not know how
to easily compute $A$, which is often analytically intractable. In this
setting (which is exactly the setting of this paper), assuming
$p_\theta$ to be differentiable (w.r.t.~$x$), and $\int_\Omega
p_0(x)\Vert\nabla \log p_\theta(x)\Vert^2_2\,dx<\infty,\,\forall\,\theta\in\Theta$,
$J(p_0\Vert p_\theta)=:J(\theta)$ in (\ref{Eq:fisher}) reduces to
\begin{equation}
J(\theta)=\sum^d_{i=1}\int_\Omega p_0(x)
\left(\frac{1}{2}\left(\partial_i\log
p_\theta(x)\right)^2+\partial^2_i\log
p_\theta(x)\right)\,dx+\frac{1}{2}\int_\Omega p_0(x)\left\Vert \nabla\log p_0(x)\right\Vert^2_2\,dx,\label{Eq:score-1}
\end{equation}
through integration by parts \cite[see][Theorem 1]{Hyvarinen-05}, under appropriate regularity conditions on $p_0$
and $p_\theta$ for all $\theta\in\Theta$. 
Here $\partial^2_i\log p_\theta(x):=\frac{\partial^2}{\partial x^2_i}\log p_\theta(x)$. 
The main advantage of the objective in (\ref{Eq:fisher})
(and also (\ref{Eq:score-1})) is that when it is applied to the situation
discussed above where $p_\theta(x)=\frac{r_\theta(x)}{A(\theta)}$, $J(\theta)$
is independent of $A(\theta)$, and an estimate of
$\theta_0$ can be obtained by simply minimizing the empirical counterpart of
$J(\theta)$, given by
$$J_n(\theta):=\frac{1}{n}\sum^n_{a=1}\sum^d_{i=1}\left(\frac{1}{2}
\left(\partial_i\log
p_\theta(X_a)\right)^2+\partial^2_i\log
p_\theta(X_a)\right)+\frac{1}{2}\int_\Omega p_0(x)\left\Vert \nabla\log p_0(x)\right\Vert^2_2\,dx.$$
Since $J_n(\theta)$ is also independent of
$A(\theta)$, $\hat{\theta}_n=\arg\min_{\theta\in\Theta}J_n(\theta)$ may 
be easily computable, unlike the MLE. We would like to highlight that
while the score matching approach may have computational advantages over MLE, it
only estimates $p_\theta$ up to the scaling factor $A(\theta)$, and
therefore requires the approximation or computation of $A(\theta)$ through
numerical integration to estimate $p_\theta$. Note that this issue (of computing $A(\theta)$ 
through numerical integration) exists even with MLE, but not with KDE. In score matching, however, numerical
integration is needed only once, while MLE would typically require a functional form of the log-partition function which is approximated through 
numerical integration at every step of an iterative
optimization algorithm (for example, see (\ref{Eq:ml})), thus leading to major computational savings. An important
application that does not require the computation of $A(\theta)$ is in finding modes of the distribution, which has
recently become very popular in image processing \citep{Comaniciu-02}, and has already been investigated in the score matching framework \citep{Sasaki-14}. Similarly, 
in sampling methods such as sequential Monte Carlo \citep{Doucet-01}, it is often
the case that the evaluation of unnormalized densities is sufficient to calculate required importance weights.
\subsection{Contributions}\label{subsec:contributions}
\emph{(i)} We present an estimate of 
$p_0\in\Cal{P}$ in the well-specified case
through the minimization of Fisher divergence, 
in Section~\ref{Sec:Theory}. First, we show that estimating
$p_0:=p_{f_0}$ using
the score matching method reduces to estimating $f_0$ by solving a simple finite-dimensional linear system
(Theorems~\ref{Thm:score} and \ref{Thm:representer}).
\citet{Hyvarinen-07} obtained a similar result for $\Scr{P}_{\text{fin}}$ where the estimator is obtained by solving
a
linear system, which in the case of Gaussian family matches the
MLE \citep{Hyvarinen-05}. The estimator obtained
in the infinite dimensional case is not a simple extension of its
finite-dimensional counterpart, however, as the former requires an
appropriate
regularizer (we use $\Vert\cdot\Vert^2_\eu{H}$) to make the
problem well-posed. 
 We would like to highlight that to the best of our knowledge, the proposed estimator is the 
 first practically computable estimator of $p_0$ with consistency guarantees (see below).
\vspace{1mm}\\
\emph{(ii)} In contrast to \cite{Hyvarinen-07} where no guarantees on
consistency or
convergence rates are provided for the density estimator in $\Scr{P}_{\text{fin}}$, we establish
in Theorem~\ref{Thm:rates} the
consistency and rates of convergence for the proposed
estimator of $f_0$, and use these to prove consistency and rates of
convergence for the corresponding plug-in estimator of $p_0$
(Theorems~\ref{Thm:density} and \ref{Thm:density-2}), 
even when $\eu{H}$ is
infinite dimensional. 
Furthermore, while the estimator of $f_0$ (and therefore $p_0$) is
obtained by minimizing the Fisher divergence, the resultant density
estimator is also shown to be consistent in KL divergence (and
therefore in Hellinger and total-variation distances) and we provide
convergence rates in all these distances. 

Formally, we
show that
the proposed estimator  $\hat{f}_n$ is converges as
$$\Vert
f_0-\hat{f}_n\Vert_\eu{H}=O_{p_0}(n^{-\alpha}),\,\,KL(p_0\Vert p_{\hat{f}_n})=O_{p_0}(n^{-2\alpha})\,\,\text{and}\,\,J(p_0\Vert p_{\hat{f}_n})=O_{p_0}
\left(n^{-\min\left\{\frac{2}{3},\frac{2\beta+1}{
2\beta+2}\right\}}\right)$$
if $f_0\in \Cal{R}(C^\beta)$ for some $\beta>0$, 
where 
$\Cal{R}(A)$ denotes the range or
image of an operator $A$,
$\alpha=\min\{\frac{1}{4},\frac{\beta}
{2\beta+2}\}$,
and $C:=\sum^d_{i=1}\int_\Omega
\partial_i k(x,\cdot)\otimes\partial_i
k(x,\cdot)\,p_0(x)\,dx$ is a Hilbert-Schmidt 
operator on
$\eu{H}$ (see
Theorem~\ref{Thm:score}) with $k$ being the reproducing kernel and $\otimes$
denoting the tensor product. When $\eu{H}$ is a finite-dimensional RKHS, we show
that the estimator enjoys parametric rates of convergence, i.e.,
$$\Vert
f_0-\hat{f}_n\Vert_\eu{H}=O_{p_0}(n^{-1/2}),\,\,KL(p_0\Vert
p_{\hat{f}_n})=O_{p_0 } (n^
{-1})\,\,\text{and}\,\,J(p_0\Vert p_{\hat{f}_n})=O_{p_0}(n^
{-1}).$$ Note that the convergence rates are obtained under a
non-classical smoothness assumption on $f_0$, namely that it lies in the image 
of certain fractional power of $C$, which reduces to a more classical
assumption if we choose $k$ to be a Mat\'{e}rn kernel (see
Section~\ref{Sec:notation} for its definition), as it induces a Sobolev space.
In Section~\ref{subsec:range}, we discuss in detail the smoothness assumption on
$f_0$ for the Gaussian (Example~\ref{exm:gaussian}) and
Mat\'{e}rn (Example~\ref{exm:matern}) kernels.
Another interesting point to observe is that unlike
in the classical function estimation methods (e.g., kernel density estimation and regression), the rates presented above for the
proposed estimator tend to saturate for $\beta>1$
($\beta>\frac{1}{2}$ w.r.t.~$J$), with the best rate attained
at $\beta=1$ ($\beta=\frac{1}{2}$ w.r.t.~$J$), which means the smoothness of
$f_0$ is not fully captured
by the estimator. Such a saturation behavior is well-studied in the inverse
problem
literature \citep{Engl-96} where it has been attributed to the choice
of regularizer. In Section~\ref{subsec:regularizer}, we discuss alternative
regularization strategies
using ideas from \citet{Bauer-07},
which covers non-parametric least squares regression: we show that for
appropriately chosen regularizers, the above mentioned rates hold for any
$\beta>0$, and do not saturate for the aforementioned ranges of $\beta$ (see
Theorem~\ref{Thm:rates-new}).\vspace{1mm}\\
\emph{(iii)} In Section~\ref{Sec:misspecified}, we study the
problem of density
estimation in the misspecified setting, i.e., $p_0\notin\Cal{P}$, which is not
addressed in \citet{Hyvarinen-07} and \citet{Fukumizu-09a}. Using a more sophisticated analysis than in
the well-specified case, we show in Theorem~\ref{Thm:misspecified-1} that
$J(p_0\Vert p_{\hat{f}_n})\rightarrow \inf_{p\in\Cal{P}}J(p_0\Vert p)$ as
$n\rightarrow \infty$. Under an appropriate smoothness assumption on
$\log\frac{p_0}{q_0}$ (see
the statement of Theorem~\ref{Thm:misspecified-1} for details), we show that
$J(p_0\Vert p_{\hat{f}_n})\rightarrow 0$ as $n\rightarrow\infty$ along with a
rate for this convergence, even though $p_0\notin\Cal{P}$. However, unlike in
the
well-specified case, where the
consistency is obtained not
only in $J$ but also in other distances, we obtain 
convergence only in $J$ for the misspecified case. Note that while
\cite{Barron-91} considered the estimation of $p_0$ in the
misspecified setting, the results are restricted to the
approximating families consisting of polynomials, splines, or
trigonometric series. Our results are more
general, as they hold for abstract RKHSs.\vspace{1mm}\\
\emph{(iv)} 
In
Section~\ref{Sec:experiments}, we present preliminary numerical results comparing the
proposed estimator with KDE in estimating a Gaussian and mixture of Gaussians,
with the goal of empirically evaluating  performance as $d$ gets large for a fixed sample
size. In these two estimation problems, we show that the proposed estimator
outperforms KDE, and the advantage 
grows as $d$ increases. Inspired by this preliminary empirical investigation, our proposed estimator
(or computationally efficient approximations)
 has been used  by
\citet{Heiko-15} in a  gradient-free adaptive MCMC sampler, and by \citet{Sun-15}  
 for graphical model structure learning. These applications demonstrate the practicality and performance of the proposed
estimator.

Finally, we would like to make clear that 
our principal goal is not to construct density estimators
that improve uniformly upon KDE, but to provide a novel flexible modeling technique for approximating an unknown density by a rich parametric family of densities,
with the parameter being infinite dimensional, in contrast to the classical approach of finite dimensional approximation.
\vspace{1mm}

Various notations and definitions that are used throughout the paper are
collected in Section~\ref{Sec:notation}. The proofs of
the results are
provided in Section~\ref{Sec:proofs}, along with some supplementary results in
an appendix.
\section{Definitions \& Notation}\label{Sec:notation}
We introduce the notation used throughout the paper. Define
$[d]:=\{1,\ldots,d\}$. For
$a:=(a_1,\ldots,a_d)\in\bb{R}^d$ and $b:=(b_1,\ldots,b_d)\in\bb{R}^d$, $\Vert
a\Vert_2:=\sqrt{\sum^d_{i=1}a^2_i}$ and $\langle a,b\rangle:=\sum^d_{i=1}a_ib_i$. For $a,b>0$, we write $a\lesssim b$ if $a\le \gamma b$ for
some positive universal constant $\gamma$. For a topological
space $\Cal{X}$, $C(\Cal{X})$ (\emph{resp.} $C_b(\Cal{X})$) denotes the space of
all continuous (\emph{resp.} bounded continuous) functions on $\Cal{X}$. For a
locally compact Hausdorff space $\Cal{X}$, $f\in C(\Cal{X})$ is said to
\emph{vanish at infinity} if for every $\epsilon > 0$ the set $\{x :
|f(x)|\ge\epsilon\}$ is compact. The class of all continuous $f$ on $\Cal{X}$
which vanish at infinity is denoted as $C_0(\Cal{X})$. For
open $\Cal{X}\subset\bb{R}^d$, $C^1(\Cal{X})$ denotes the space of
continuously differentiable functions on $\Cal{X}$. For $f\in
C_b(\Cal{X})$, $\Vert f\Vert_\infty:=\sup_{x\in\Cal{X}}|f(x)|$
denotes the supremum norm of $f$. $M_b(\Cal{X})$ denotes the set of all finite
Borel measures on $\Cal{X}$. For $\mu\in M_b(\Cal{X})$, $L^r(\Cal{X},\mu)$
denotes the Banach space of $r$-power ($r\ge
1$) $\mu$-integrable functions. For $\Cal{X}\subset\bb{R}^d$, we will use
$L^r(\Cal{X})$ for $L^r(\Cal{X},\mu)$ if $\mu$ is a Lebesgue measure on
$\Cal{X}$. For $f\in L^p(\Cal{X},\mu)$, $\Vert
f\Vert_{L^r(\Cal{X},\mu)}:=\left(\int_\Cal{X}|f|^r\,d\mu\right)^{1/r}$ denotes
the $L^r$-norm of $f$ for $1\le r<\infty$ and we denote it as
$\Vert\cdot\Vert_{L^r(\Cal{X})}$ if $\Cal{X}\subset\bb{R}^d$ and $\mu$ is the
Lebesgue measure. The convolution $f\ast g$
of two measurable functions $f$ and $g$ on $\bb{R}^d$ is defined as $$(f\ast
g)(x):=\int_{\bb{R}^d}f(y)g(x-y)\,dy,$$ provided the integral exists for all
$x\in\bb{R}^d$. The Fourier transform of $f\in
L^1(\bb{R}^d)$ is defined
as $$f^\wedge(y)=(2\pi)^{-d/2}\int_{\bb{R}^d}f(x)\,e^{-i\langle
y,x\rangle}\,dx$$
where $i$ denotes the imaginary unit $\sqrt{-1}$. 

In the following, for the sake of completeness and simplicity, we present
definitions restricted to Hilbert spaces. Let $H_1$ and $H_2$ be abstract Hilbert spaces. A map $S:H_1\rightarrow H_2$ is called a \emph{linear operator}
if it satisfies $S(\alpha x)=\alpha Sx$ and $S(x+x')=Sx+Sx'$ for all
$\alpha\in\bb{R}$ and $x,x'\in H_1$, where $Sx:=S(x)$. A linear operator $S$ is
said to be \emph{bounded}, i.e., the image $SB_{H_1}$ of $B_{H_1}$ under $S$ is
bounded if and only if there exists a constant $c\in [0,\infty)$ such that for
all $x\in H_1$ we have $\Vert Sx\Vert_{H_2}\le c\Vert x\Vert_{H_1}$, where
$B_{H_1}:=\{x\in H_1:\Vert x\Vert_{H_1}\le 1\}$. In this case, the
\emph{operator norm} of $S$ is defined as $\Vert S\Vert:=\sup\{\Vert
Sx\Vert_{H_2}:x\in B_{H_1}\}$. Define $\Cal{L}(H_1,H_2)$ be the space of bounded
linear operators from $H_1$ to $H_2$. $S\in\Cal{L}(H_1,H_2)$ is said to be
\emph{compact} if $\overline{SB_{H_1}}$ is a compact subset in $H_2$. The
\emph{adjoint operator} $S^*:H_2\rightarrow H_1$ of $S\in \Cal{L}(H_1,H_2)$ is
defined by $\langle x,S^*y\rangle_{H_1}=\langle Sx,y \rangle_{H_2},\,x\in
H_1,\,y\in H_2$. $S\in\Cal{L}(H):=\Cal{L}(H,H)$ is called \emph{self-adjoint} if
$S^*=S$ and is called \emph{positive} if $\langle Sx,x\rangle_H\ge 0$ for all
$x\in H$. $\alpha\in\bb{R}$ is called an \emph{eigenvalue} of $S\in\Cal{L}(H)$
if there exists an $x\ne 0$ such that $Sx=\alpha x$ and such an $x$ is called
the \emph{eigenvector} of $S$ and $\alpha$. For compact, positive, self-adjoint
$S\in \Cal{L}(H)$, $S^r: H\rightarrow H$, $r\ge 0$ is called a \emph{fractional
power} of $S$ and $S^{1/2}$ is the \emph{square root} of $S$, which we write as
$\sqrt{S}:=S^{1/2}$. An operator $S\in \Cal{L}(H_1,H_2)$ is
\emph{Hilbert-Schmidt} if $\Vert
S\Vert_{HS}:=(\sum_{j\in J}\Vert Se_j\Vert^2_{H_2})^{1/2}<\infty$
where $(e_j)_{j\in J}$ is an arbitrary orthonormal basis of separable
Hilbert space $H_1$. $S\in\Cal{L}(H_1,H_2)$ is said to be of \emph{trace class}
if $\sum_{j\in J}\langle (S^* S)^{1/2}e_j,e_j\rangle_{H_1}<\infty$. For $x\in
H_1$ and $y\in H_2$, $x\otimes y$ is an element of the tensor product space
$H_1\otimes H_2$ which can also be seen as an operator from $H_2$ to $H_1$ as
$(x\otimes y)z=x\langle y,z\rangle_{H_2}$ for any $z\in H_2$. $\Cal{R}(S)$
denotes the \emph{range space (or image)} of $S$.

A real-valued symmetric function $k:\Cal{X}\times \Cal{X}\rightarrow\bb{R}$ is
called a
positive definite (pd) kernel if, for all $n\in\bb{N}$,
$\alpha_1,\ldots,\alpha_n\in\bb{R}$ and all $x_1,\ldots,x_n\in \Cal{X}$, we
have $\sum^n_{i,j=1}\alpha_i\alpha_jk(x_i,x_j)\ge 0$. 
A function $k:\Cal{X}\times\Cal{X}\rightarrow\bb{R},\, (x,y)\mapsto k(x,y)$ is
a \emph{reproducing kernel} of the Hilbert space $(\Cal{H}_k,\langle
\cdot,\cdot\rangle_{\Cal{H}_k})$ of functions if and only if (i)
$\forall\,y\in\Cal{X}$,
$k(y,\cdot)\in\Cal{H}_k$ and (ii) $\forall\,y\in\Cal{X}$,
$\forall\,f\in\Cal{H}_k$, $\langle
f,k(y,\cdot)\rangle_{\Cal{H}_k}=f(y)$ hold. If such a $k$ exists,
then $\Cal{H}_k$ is called a
\emph{reproducing kernel Hilbert space}. Since $\langle
k(x,\cdot),k(y,\cdot)\rangle_{\Cal{H}_k}=k(x,y),\,\forall\,x,y\in\Cal{X}$, it
is easy to show that every reproducing kernel (r.k.) $k$ is symmetric
and positive definite. 
Some examples of kernels that appear throughout the paper are: 
\emph{Gaussian kernel}, $k(x,y)=\exp(-\sigma\Vert
x-y\Vert^2_2),\,x,y\in\bb{R}^d,\,\sigma>0$ that induces the following
\emph{Gaussian RKHS},
$$\Cal{H}_k=\eu{H}_\sigma:=\Big\{f\in L^2(\bb{R}^d)\cap C(\bb{R}^d)\,:\,\int
|f^\wedge(\omega)|^2e^{\Vert
\omega\Vert^2_2/4\sigma}\,d\omega<\infty\Big\},$$ the \emph{inverse
multiquadric kernel}, $k(x,y)=(1+\Vert\frac{x-y}{c}
\Vert^2_2)^{-\beta},
\,x,y\in\bb{R}^d,\,\beta>0,\,c\in (0,\infty)$ and the \emph{Mat\'{e}rn
kernel},
$k(x,y)=\frac{2^{1-\beta}}{\Gamma(\beta)}\Vert
x-y\Vert^{\beta-d/2}_2\frak{K}_{d/2-\beta}(\Vert
x-y\Vert_2),\,x,y\in\bb{R}^d,\,\beta>d/2$ that induces the Sobolev space,
$H^\beta_2$,
\begin{equation}\Cal{H}_k=H^\beta_2:=\Big\{f\in L^2(\bb{R}^d)\cap
C(\bb{R}^d)\,:\,\int
(1+\Vert
\omega\Vert^2_2)^\beta|f^\wedge(\omega)|^2\,d\omega<\infty\Big\},\nonumber
\end{equation}
where $\Gamma$ is the Gamma
function, and $\frak{K}_v$ is the modified Bessel
function of the third kind of order $v$ ($v$ controls the smoothness of
$k$).

For any real-valued function $f$ defined on open $\Cal{X}\subset\bb{R}^d$, $f$ is said to be $m$-times continuously differentiable if for $\alpha\in\bb{N}^d_0$ 
with $|\alpha|:=\sum^d_{i=1}\alpha_i\le m$, $\partial^\alpha f(x)=\partial^{\alpha_1}_1\ldots\partial^{\alpha_d}_d f(x)=\frac{\partial^{|\alpha|}}{\partial x^{\alpha_1}_1\ldots\partial x^{\alpha_d}_d}f(x)$ exists. A kernel $k$ is said to be $m$-times continuously differentiable if $\partial^{\alpha,\alpha}k:\Cal{X}\times\Cal{X}\rightarrow\bb{R}$ 
exists and is continuous for all $\alpha\in\bb{N}^d_0$ with $|\alpha|\le m$ where $\partial^{\alpha,\alpha}:=\partial^{\alpha_1}_1\ldots\partial^{\alpha_d}_d\partial^{\alpha_1}_{1+d}\ldots\partial^{\alpha_d}_{2d}$. 
Corollary 4.36 in \citet{Steinwart-08} and Theorem 1 in \citet{Zhou-08} state that if $\partial^{\alpha,\alpha}k$ exists and is continuous, then 
$\partial^\alpha k(x,\cdot)=\partial^{\alpha_1}_{1}\ldots\partial^{\alpha_d}_{d}k(x,\cdot)$ $=\frac{\partial^{|\alpha|}}{\partial x^{\alpha_1}_1\ldots\partial x^{\alpha_d}_d}k((x_1,\ldots,x_d),\cdot)\in\Cal{H}_k$ with $x=(x_1,\ldots,x_d)$ and for every $f\in\Cal{H}_k$, we have $\partial^\alpha f(x)=\langle \partial^\alpha k(x,\cdot),f\rangle_{\Cal{H}_k}$ and $\partial^{\alpha,\alpha}k(x,x')=\langle \partial^\alpha k(x,\cdot), \partial^\alpha k(x',\cdot)\rangle_{\Cal{H}_k}$.

Given two probability
densities, $p$ and $q$ on $\Omega\subset\bb{R}^d$, the Kullback-Leibler
divergence (KL) and
Hellinger distance ($h$) are defined as $KL(p\Vert q)=\int
p(x)\log\frac{p(x)}{q(x)}\,dx$ and $h(p,q)=\Vert
\sqrt{p}-\sqrt{q}\Vert_{L^2(\Omega)}$ respectively. We refer to 
$\Vert p-q\Vert_{L^1(\Omega)}$ as the total
variation (TV) distance between $p$ and $q$. 

\section{Approximation of Densities by $\Cal{P}$}\label{Sec:approximation}
In this section, we first show that every finite dimensional exponential family, $\Scr{P}_{\text{fin}}$
is generated by the family $\Cal{P}$ induced by a finite dimensional RKHS, which naturally leads to the 
infinite dimensional generalization of $\Scr{P}_{\text{fin}}$ when $\eu{H}$ is an infinite dimensional RKHS. 
Next, we investigate the
approximation properties of $\Cal{P}$ in Proposition~\ref{Thm:approx} and
Corollary~\ref{cor:approx} when $\eu{H}$ is an infinite dimensional RKHS.

Let us consider a $r$-parameter
exponential family, $\Scr{P}_{\text{fin}}$ with sufficient statistic
$T(x):=(T_1(x),\ldots,T_r(x))$ 
and construct a Hilbert space, $\eu{H}=\text{span}\{T_1(x),\ldots,T_r(x)\}$. It is easy to
verify that $\Cal{P}$ induced by $\eu{H}$ is exactly the same as
$\Scr{P}_{\text{fin}}$ since any $f\in \eu{H}$ can be written as
$f(x)=\sum^r_{i=1}\theta_i T_i(x)$ for some $(\theta_i)^r_{i=1}\subset\bb{R}$. In fact, by defining 
the inner product between $f=\sum^r_{i=1}\theta_i T_i$ and $g=\sum^r_{i=1}\gamma_i T_i$ as $\langle f,g\rangle_\eu{H}:=\sum^r_{i=1}\theta_i\gamma_i$, 
it follows that $\eu{H}$ is an RKHS with the r.k.~$k(x,y)=\langle
T(x),T(y)\rangle_{\bb{R}^r}$ since $\langle f,k(x,\cdot)\rangle_\eu{H}=\sum^r_{i=1}\theta_i T_i(x)=f(x)$. Based on this equivalence between $\Scr{P}_{\text{fin}}$ and $\Cal{P}$
induced by a finite dimensional RKHS, it is therefore clear that $\Cal{P}$ induced by a infinite dimensional RKHS is a strict 
generalization to $\Scr{P}_{\text{fin}}$ with $k(\cdot,x)$ playing the role of a sufficient statistic. 
\begin{example}\label{exm:finite-dim} 
The following are some popular examples of probability distributions that belong to $\Scr{P}_{\emph{fin}}$. Here we show the corresponding RKHSs $(\eu{H},k)$
that generate these distributions. In some of these examples, we choose $q_0(x)=1$ and ignore
the fact that $q_0$ is a probability distribution as assumed in the definition
of $\Cal{P}$.\vspace{-2mm}
\begin{itemize}
 \item[] \text{Exponential:} $\Omega=\bb{R}_{++}:=\bb{R}_+\backslash\{0\}$,
$k(x,y)=xy$. \vspace{-1mm}
\item[] \text{Normal:} $\Omega=\bb{R}$, $k(x,y)=xy+x^2y^2$.  \vspace{-1mm}
\item[] \text{Beta:} $\Omega=(0,1)$, $k(x,y)=\log x\log y+\log(1-x)\log(1-y)$. \vspace{-1mm}
\item[] \text{Gamma:} $\Omega=\bb{R}_{++}$, $k(x,y)=\log x \log y+xy$. \vspace{-1mm}
 \item[] \text{Inverse Gaussian:} $\Omega=\bb{R}_{++}$,
 $k(x,y)=xy+\frac{1}{xy}$. \vspace{-1mm}
\item[] \text{Poisson:} $\Omega=\bb{N}\cup\{0\}$, $k(x,y)=xy$, $q_0(x)=(x!\,
e)^{-1}$. \vspace{-1mm}
\item[] \text{Binomial:} $\Omega=\{0,\ldots,m\}$, $k(x,y)=xy$,
$q_0(x)=2^{-m}{ m \choose c}$.
\end{itemize}
\end{example}

While Example~\ref{exm:finite-dim} shows that all popular probability distributions 
are contained in $\Cal{P}$
for an appropriate choice of finite-dimensional $\eu{H}$, it is of interest to
understand the richness of $\Cal{P}$ (i.e., what class of distributions can be approximated arbitrarily
well by $\Cal{P}$?) when $\eu{H}$ is an infinite dimensional
RKHS. This is addressed by the following result, which is proved in Section~\ref{subsec:thm-approx}.
\begin{proposition}\label{Thm:approx}
Define
$\Cal{P}_0:=\left\{\pi_f(x)=e^{f(x)- A(f)}q_0(x),\,x\in\Omega:f\in
C_0(\Omega)\right\}$ where $\Omega\subseteq\bb{R}^d$ is locally compact Hausdorff. Suppose $k(x,\cdot)\in
C_0(\Omega),\,\forall\,x\in \Omega$ and
\begin{equation}\int\int
k(x,y)\,d\mu(x)\,d\mu(y)>0,\,\,\forall\,\mu\in
M_b(\Omega)\backslash\{0\}.\label{Eq:ispd}\end{equation} Then $\Cal{P}$ is
dense
in $\Cal{P}_0$ w.r.t.~Kullback-Leibler
divergence, total variation ($L^1$ norm) and Hellinger distances. In addition,
if $q_0\in L^1(\Omega)\cap L^r(\Omega)$ for some $1<
r\le\infty$, then $\Cal{P}$ is also dense in $\Cal{P}_0$ w.r.t.~$L^r$ norm.\vspace{-2mm}
\end{proposition}
A sufficient condition for $\Omega\subseteq\bb{R}^d$ to be locally compact Hausdorff is that it is either open or closed. Condition (\ref{Eq:ispd}) is equivalent to $k$ being $c_0$-universal \citep[p. 2396]{Sriperumbudur-11}. If $k(x,y)=\psi(x-y),\,x,y\in\Omega=\bb{R}^d$ where $\psi\in C_b(\bb{R}^d)\cap
L^1(\bb{R}^d)$, then
(\ref{Eq:ispd}) can be shown to be equivalent to
$\text{supp}(\psi^\wedge)=\bb{R}^d$ \citep[Proposition
5]{Sriperumbudur-11}. Examples of kernels that satisfy the conditions in
Proposition~\ref{Thm:approx} include the Gaussian, Mat\'{e}rn and inverse
multiquadrics. In fact, any compactly supported non-zero $\psi\in C_b(\bb{R}^d)$
satisfies the assumptions in Proposition~\ref{Thm:approx} as
$\text{supp}(\psi^\wedge)=\bb{R}^d$ \cite[Corollary 10]{Sriperumbudur-10a}.
Though $\Cal{P}_0$ is still a parametric family of densities indexed by a
Banach space (here $C_0(\Omega)$), the following corollary (proved in Section~\ref{subsec:cor-approx-1}) 
to Proposition~\ref{Thm:approx} shows that a
broad class of continuous densities are contained
in $\Cal{P}_0$ and therefore can be approximated arbitrarily well in $L^r$ norm
($1\le r\le\infty$), Hellinger distance, and KL divergence by $\Cal{P}$.\vspace{-1mm} 
\begin{corollary}\label{cor:approx}
Let $q_0\in C(\Omega)$ be a
probability density 
such that $q_0(x)>0$ for all
$x\in\Omega$, where $\Omega\subseteq\bb{R}^d$ is locally compact Hausdorff. Suppose there exists a constant $\ell$ such that for any $\epsilon > 0$, $\exists\,R > 0$ that satisfies
$|\frac{p(x)}{q_0(x)} - \ell | \leq \epsilon$ for any $x$ with $\Vert x\Vert_2>R$. Define
$$\Cal{P}_{c}:=\left\{p\in
C(\Omega):\int_\Omega
p(x)\,dx=1,p(x)\ge 0,\,\forall\,x\in
\Omega\,\,\emph{and}\,\,
\frac{p}{q_0}-\ell\in C_0(\Omega)\right\}.$$ Suppose $k(x,\cdot)\in
C_0(\Omega),\,\forall\,x\in\Omega$ and (\ref{Eq:ispd}) holds. Then
$\Cal{P}$ is dense in $\Cal{P}_c$ w.r.t.~KL divergence, TV and
Hellinger distances. Moreover, if
$q_0\in L^1(\Omega)\cap L^r(\Omega)$ for some $1<
r\le\infty$, then $\Cal{P}$ is also dense in $\Cal{P}_c$ w.r.t.~$L^r$ norm.\vspace{-2mm}
\end{corollary}
By choosing $\Omega$ to be compact and $q_0$ to be a uniform distribution on
$\Omega$, Corollary~\ref{cor:approx} reduces to an easily interpretable result
that any continuous density $p_0$ on $\Omega$ can be approximated arbitrarily
well by densities in $\Cal{P}$ in KL, Hellinger and $L^r$ ($1\le r\le\infty$)
distances.

Similar to the results so far, an approximation result for
$\Cal{P}$ can also be obtained w.r.t.~Fisher divergence (see
Proposition~\ref{Thm:approx-fd}). Since this result is heavily based on the
notions and results developed in
Section~\ref{Sec:misspecified}, we defer its presentation until that section. 
Briefly, this result states that if $\eu{H}$ is
sufficiently rich (i.e., dense in an appropriate class of functions), then any
$p\in C^1(\Omega)$ with $J(p\Vert q_0)<\infty$ can
be approximated arbitrarily well by elements in $\Cal{P}$ w.r.t.~Fisher
divergence, where $q_0\in C^1(\Omega)$.


 \section{Density Estimation in $\Cal{P}$: Well-specified Case}\label{Sec:Theory}
In this section,   we present our score matching  estimator for an unknown density
$p_0:=p_{f_0}\in
\Cal{P}$ (well-specified case) from i.i.d.~random samples $(X_a)^n_{a=1}$
drawn from it. This involves
choosing the minimizer of
the (empirical) Fisher divergence between $p_0$ and $p_f\in\Cal{P}$ as the
estimator, $\hat{f}$ which we show in
Theorem~\ref{Thm:representer} to be obtained by solving a simple
finite-dimensional linear system. In contrast, we would like to remind the
reader that the MLE is  infeasible in practice due to the difficulty in handling $A(f)$. 
The consistency and convergence rates of $\hat{f}\in\Cal{F}$ and the plug-in estimator $p_{\hat{f}}$ are
provided in Section~\ref{subsec:consistency} (see
Theorems~\ref{Thm:rates} and \ref{Thm:density}). 
Before we proceed, 
we list the assumptions on $p_0$, $q_0$ and
$\eu{H}$ that we need in our analysis.\vspace{-1mm}
\begin{itemize}
\item[\textbf{(A)}] $\Omega$ is a non-empty open subset of $\bb{R}^d$ with a piecewise smooth boundary $\partial\Omega:=\overline{\Omega}\backslash\Omega$, where $\overline{\Omega}$
denotes the closure of $\Omega$.
\item[\textbf{(B)}] $p_0$ is continuously extendible to $\overline{\Omega}$. $k$ is twice continuously
differentiable on $\Omega\times\Omega$ with continuous extension of $\partial^{\alpha,\alpha}k$ to $\overline{\Omega}\times\overline{\Omega}$ for $|\alpha|\le 2$. 
\item[\textbf{(C)}] $\partial_i\partial_{i+d}k(x,x)p_0(x)=0$ for $x\in\partial\Omega$ and $\sqrt{\partial_i\partial_{i+d}k(x,x)}p_0(x)=o(\Vert x\Vert^{1-d}_2)$ as $x\in\Omega$, $\Vert x\Vert_2\rightarrow\infty$
for all $i\in[d]$.
\item[\textbf{(D)}] \emph{($\varepsilon$-Integrability)} For some $\varepsilon\ge
1$ and $\forall\,i\in[d]$, $\partial_i\partial_{i+d}k(x,x), \sqrt{\partial^2_i\partial^2_{i+d}k(x,x)}$ and $\sqrt{\partial_i\partial_{i+d}k(x,x)}\partial_i\log q_0(x)\in L^\varepsilon(\Omega,p_0),$
where $q_0\in
C^1(\Omega).$
\end{itemize}
\begin{remo}\label{rem:assumptions}
\emph{(i)} $\Omega$ being a subset of $\R^d$ along with $k$ being continuous ensures that
$\eu{H}$ is separable \cite[Lemma 4.33]{Steinwart-08}. The twice differentiability of $k$ ensures that every $f\in\eu{H}$ is twice continuously
differentiable \cite[Corollary 4.36]{Steinwart-08}. \emph{\textbf{(C)}} ensures that $J$ in (\ref{Eq:fisher}) is equivalent to the one in (\ref{Eq:score-1}) through 
integration by parts on $\Omega$ (see Corollary 7.6.2 in \citealp{Duistermaat-04} for integration by parts on bounded subsets of $\bb{R}^d$ which can be extended to unbounded $\Omega$
through a truncation and limiting argument) for densities in $\Cal{P}$. In particular, \emph{\textbf{(C)}} ensures that 
$\int_\Omega \partial_i f(x) \partial_i p_0(x)\,dx=-\int_\Omega \partial^2_i f(x) p_0(x)\,dx$ for all $f\in\eu{H}$ and $i\in[d]$, which will be critical to prove 
the representation in Theorem~\ref{Thm:score}(ii), upon which rest of the results depend. The decay condition in \emph{\textbf{(C)}} can be weakened
to $\sqrt{\partial_i\partial_{i+d}k(x,x)}p_0(x)=o(\Vert x\Vert^{1-\overline{d}}_2)$ as $x\in\Omega$, $\Vert x\Vert_2\rightarrow\infty$
for all $i\in[d]$ if $\Omega$ is a (possibly unbounded) box where $\overline{d}=\#\{i\in[d]| (a_i,b_i)\,\,\text{is unbounded}\}$.\vspace{2mm}\\
\emph{(ii)} When $\varepsilon=1$, the first condition in
\emph{\textbf{(D)}} ensures that $J(p_0\Vert p_f)<\infty$ for any $p_f\in\Cal{P}$. The
other
two conditions ensure the validity of the alternate representation
for $J(p_0\Vert p_f)$ in (\ref{Eq:score-1}) which will be useful in constructing
estimators of $p_0$ (see
Theorem~\ref{Thm:score}). 
Examples of kernels that satisfy
\emph{\textbf{(D)}} are the Gaussian, Mat\'{e}rn (with $\beta>\max\{2,d/2\}$), and
inverse multiquadric kernels, for which it is easy to show that
there
exists $q_0$ that satisfies \emph{\textbf{(D)}}.\vspace{2mm}\\
\emph{(iii)} \emph{(Identifiability)} The above list of assumptions do
not include the identifiability condition that ensures $p_{f_1}=p_{f_2}$ if and only if $f_1=f_2$. It is clear that if constant functions are included in
$\eu{H}$, i.e., $1\in\eu{H}$, then $p_f=p_{f+c}$ for any $c\in\bb{R}$. On the
other hand, it can be shown that if $1\notin\eu{H}$ and $\emph{supp}(q_0)=\Omega$, then
$p_{f_1}=p_{f_2}\Leftrightarrow
f_1=f_2$. 
A
sufficient condition for $1\notin\eu{H}$ is $k\in
C_0(\Omega\times\Omega)$. 
We do not explicitly impose the
identifiability condition as a part of our blanket assumptions because the
assumptions under which consistency and
rates are obtained in Theorem~\ref{Thm:density} 
automatically ensure identifiability.\vspace{-2mm} 
\end{remo}
Under these assumptions, the following result---proved in Section~\ref{subsec:thm-score}---shows that the problem of estimating $p_0$ through the
minimization of Fisher divergence reduces to the problem of estimating $f_0$
through a weighted least squares minimization in $\eu{H}$ (see parts (i) and
(ii)). This
motivates the minimization of the regularized empirical weighted least squares
(see part (iv)) to obtain an estimator $f_{\lambda,n}$ of $f_0$, which is then used
to construct the plug-in estimate $p_{f_{\lambda,n}}$ of
$p_0$. \vspace{-.5mm}

\begin{theorem}\label{Thm:score}
Suppose \textbf{\emph{(A)}}--\textbf{\emph{(D)}} hold with $\varepsilon=1$. 
Then 
$J(p_0\Vert p_f)<\infty$ 
for all $f\in\Cal{F}$. In addition, the following
hold.\vspace{1.5mm}\\
(i)  For all $f\in\Cal{F}$, \begin{equation}J(f):=J(p_0\Vert p_f)=\frac{1}{2}\left\langle
f-f_0,C(f-f_0)\right\rangle_\eu{H},\label{Eq:population}\end{equation} where 
$C:\eu{H}\rightarrow\eu{H}$, $C:=\int_\Omega p_0(x)\sum^d_{i=1}\partial_{i}k(x,\cdot)\otimes\partial_{i}k(x,\cdot)\,dx$
is a trace-class positive operator with $$Cf=\int_\Omega p_0(x)\sum^d_{i=1}\partial_{i}k(x,\cdot)\partial_if(x)\,dx.$$
(ii) Alternatively, $$J(f)=\frac{1}{2}\langle
f,Cf\rangle_\eu{H}+\langle
f,\xi\rangle_\eu{H}+J(p_0\Vert q_0)$$
where $$\xi:=\int_\Omega
p_0(x)\sum^d_{i=1}\left(\partial_{i}k(x,\cdot)\partial_i\log q_0(x)+\partial^2_{i}k(x,\cdot)\right)\,dx\in \eu{H}$$ and $f_0$ satisfies $Cf_0=-\xi$.\vspace{2mm}\\
%
(iii) For any $\lambda>0$, a unique minimizer $f_\lambda$ of 
$J_\lambda(f):=J(f)+\frac{\lambda}{2}\Vert f\Vert^2_\eu{H}$ over
$\eu{H}$ exists and
is given by
$$f_\lambda=-(C+\lambda I)^{-1}\xi=(C+\lambda I)^{-1}Cf_0.\vspace{2mm}$$
(iv) \textbf{(Estimator of $f_0$)} Given samples $(X_a)^n_{a=1}$ drawn
i.i.d.~from $p_0$, for any
$\lambda>0$, the unique minimizer $f_{\lambda,n}$ of
$\hat{J}_\lambda(f):=\hat{J}(f)+\frac{\lambda}{2}\Vert f\Vert^2_\eu{H}$ over $\eu{H}$ exists and is
given by
\begin{equation}f_{\lambda,n}=-(\hat{C}+\lambda
I)^{-1}\hat{\xi},\nonumber\end{equation}
where $\hat{J}(f):=\frac{1}{2}\langle
f,\hat{C}f\rangle_\eu{H}+\langle
f,\hat{\xi}\rangle_\eu{H}+J(p_0\Vert q_0)$, $\hat{C}:=\frac{1}{n}\sum^n_{a=1}\sum^d_{i=1}\partial_{i}
k(X_a,\cdot)\otimes\partial_{i} k(X_a,\cdot)$ and 
$$\hat{\xi}:=\frac{1}{n}\sum^n_{a=1}\sum^d_{i=1}\left(
\partial_{i}
k(X_a,\cdot)\partial_i\log q_0(X_a)+\partial^2_{i}k(X_a,\cdot)\right).$$
\end{theorem}
An advantage of the alternate formulation of
$J(f)$ in Theorem~\ref{Thm:score}(ii) over (\ref{Eq:population}) is that it
provides
a simple way to obtain an empirical estimate of $J(f)$---by replacing $C$ and
$\xi$ by their empirical estimators, $\hat{C}$ and $\hat{\xi}$
respectively---from
finite samples drawn i.i.d.~from $p_0$, which is then used to obtain an
estimator of $f_0$. Note that the empirical estimate of $J(f)$, i.e.,
$\hat{J}(f)$ depends only on $\hat{C}$ and $\hat{\xi}$ which in turn depend on
the known quantities, $k$ and $q_0$, and therefore $f_{\lambda,n}$ in
Theorem~\ref{Thm:score}(iv) should
in principle be computable. In practice, however, it is not easy to compute the expression for $f_{\lambda,n}=-(\hat{C}+\lambda
I)^{-1}\hat{\xi}$ as it involves solving an infinite dimensional linear
system. In Theorem~\ref{Thm:representer} (proved in Section~\ref{subsec:thm-representer}), we provide an alternative
expression for $f_{\lambda,n}$ 
as a solution of
a simple finite-dimensional linear system (see (\ref{Eq:representer-f}) and
(\ref{Eq:linearsystem})), using the general representer theorem (see Theorem~\ref{thm:representer}).
It is interesting to note that while the solution to
$J(f)$ in Theorem~\ref{Thm:score}(ii) is obtained by solving a non-linear
system,
$Cf_0=-\xi$ (the system is non-linear as $C$ depends on $p_0$ which in turn
depends on $f_0$), its estimator $f_{\lambda,n}$ proposed in
Theorem~\ref{Thm:score}, is obtained by solving a simple linear system. In
addition, we would like to highlight the fact that the proposed estimator,
$f_{\lambda,n}$ is precisely the Tikhonov regularized solution (which is
well-studied in the theory of linear inverse problems) to the ill-posed linear
system $\hat{C}f=-\hat{\xi}$. 
We further discuss the choice of
regularizer in Section~\ref{subsec:regularizer} using ideas from
the inverse problem literature.

An important remark we would like to make about Theorem~\ref{Thm:score} is that
though $J(f)$ in (\ref{Eq:population}) is valid only for $f\in\Cal{F}$, as it is
obtained from $J(p_0\Vert p_f)$ where $p_0,p_f\in\Cal{P}$, the expression
$\langle
f-f_0,C(f-f_0)\rangle_{\eu{H}}$ is valid for any $f\in\eu{H}$, as it is
finite under the assumption that \textbf{(D)} holds with $\varepsilon=1$.
Therefore, in Theorem~\ref{Thm:score}(iii, iv), $f_\lambda$ and $f_{\lambda,n}$
are obtained by minimizing $J_\lambda$ and $\hat{J}_\lambda$ over $\eu{H}$ instead of over $\Cal{F}$, as the latter does not
yield a nice expression (unlike $f_\lambda$ and $f_{\lambda,n}$, respectively). However, there is no guarantee that $f_{\lambda,n}\in\Cal{F}$, and so the
density estimator $p_{f_{\lambda,n}}$ may not be valid. While this is not an
issue when studying the convergence of $\Vert f_{\lambda,n}-f_0\Vert_\eu{H}$
(see Theorem~\ref{Thm:rates}), the convergence of $p_{f_{\lambda,n}}$ to $p_0$
(in various distances) needs to be handled slightly differently depending on
whether the kernel is bounded or not (see Theorems~\ref{Thm:density} and \ref{Thm:density-2}). 
Note that when the kernel is bounded,
we obtain $\Cal{F}=\eu{H}$, which implies $p_{f_{\lambda,n}}$ is valid.
\begin{theorem}[Computation of $f_{\lambda,n}$]\label{Thm:representer}
Let
$f_{\lambda,n}=\arg\inf_{f\in\eu{H}}\hat{J}_\lambda(f)$, where $\hat{J}_\lambda(f)$ is defined in Theorem~\ref{Thm:score}(iv) and $\lambda>0$. Then
\begin{equation}f_{\lambda,n}=-\frac{\hat{\xi}}{\lambda}+\sum^n_{a=1}\sum^d_{i=1}\beta_{(a-1)d+i}
\partial_i k(X_a,\cdot),\label{Eq:representer-f}\end{equation}
where $\hat{\xi}$ is defined in Theorem~\ref{Thm:score}(iv) and $\bm{\beta}=(\beta_{(a-1)d+i})_{a,i}$ is obtained by solving 
\begin{equation}\left(\bm{G}+n\lambda I\right)\bm{\beta}=\bm{h}/\lambda\label{Eq:linearsystem}\end{equation}
with $(\bm{G})_{(a-1)d+i,(b-1)d+j}=\partial_i\partial_{j+d}k(X_a,X_b)\,\,\,\text{and}$
$$(\bm{h})_{(a-1)d+i}=\langle \hat{\xi},\partial_i k(X_a,\cdot)\rangle_\eu{H}=\frac{1}{n}\sum^n_{b=1}\sum^d_{j=1} \partial_i\partial^2_{j+d}
k(X_a,X_b) +\partial_i\partial_{j+d}
k(X_a,X_b)\partial_j \log q_0(X_b).$$
\end{theorem}
We would like to highlight that though $f_{\lambda,n}$ requires
solving a simple linear system in (\ref{Eq:linearsystem}), it can still be
computationally intensive when $d$ and $n$ are large as $\bm{G}$ is a $nd\times
nd$ matrix. 
This is still a better scenario than that of MLE, however, 
since computationally efficient
methods exist to solve large linear systems such as
(\ref{Eq:linearsystem}), whereas MLE can be intractable due to the difficulty 
in handling the log-partition function (though it can be approximated). 
On the
other hand, MLE is statistically well-understood, with
consistency and convergence rates established in general for the problem of
density estimation \citep{vandeGeer-00} and in particular for the problem at hand
\citep{Fukumizu-09a}. In order to ensure that $f_{\lambda,n}$ and $p_{f_{\lambda,n}}$ are statistically useful, in the following section, we
investigate their consistency and convergence rates under some smoothness
conditions on $f_0$. \vspace{-1mm}
\subsection{Consistency and Rate of Convergence}\label{subsec:consistency}
In this section, we prove the consistency of $f_{\lambda,n}$ (see
Theorem~\ref{Thm:rates}(i)) and $p_{f_{\lambda,n}}$ (see
Theorems~\ref{Thm:density} and \ref{Thm:density-2}). 
Under
the smoothness assumption that $f_0\in\Cal{R}(C^\beta)$ for some $\beta>0$, we present convergence rates
for $f_{\lambda,n}$ and $p_{f_{\lambda,n}}$ in Theorems~\ref{Thm:rates}(ii), \ref{Thm:density} and \ref{Thm:density-2}. 
In reference to the
following
results, for simplicity we suppress the dependence of $\lambda$ on $n$ by
defining $\lambda:=\lambda_n$ where $(\lambda_n)_{n\in\bb{N}}\subset
(0,\infty)$.

\begin{theorem}[Consistency and convergence rates for
$f_{\lambda,n}$]\label{Thm:rates}
Suppose \emph{\textbf{(A)}}--\emph{\textbf{(D)}} with
$\varepsilon=2$ hold.
\vspace{1mm}\\
(i) 
If
$f_0\in\overline{\Cal{R}(C)}$, then 
$\left\Vert f_{\lambda,n} -
f_0\right\Vert_{\eu{H}}\stackrel{p_0}{\rightarrow}
0\,\,\text{as}\,\,\lambda\to 0,\,
\lambda\sqrt{n} \to
\infty\,\,\text{and}\,\,n\to\infty.$\vspace{1mm}\\
(ii) 
If $f_0\in \Cal{R}(C^\beta)$ for some
$\beta>0$, then for $\lambda=n^{-\max\left\{\frac{1}{4},\frac{1}{2(\beta+1)}\right\}}$,
$$\Vert
f_{\lambda,n}-f_0\Vert_\eu{H}=O_{p_0}\Big(n^{-\min\left\{\frac{1}{4},\frac{
\beta}{2(\beta+1)}\right\}}\Big)\,\,\,\text{as}\,\,\,n\rightarrow\infty.$$
(iii) 
If $\Vert
C^{-1}\Vert<\infty$, then for $\lambda=n^{-\frac{1}{2}}$,
$\Vert
f_{\lambda,n}-f_0\Vert_\eu{H}=O_{p_0}(n^{-1/2})$ as
$n\rightarrow\infty$.
\end{theorem}
\begin{rem}
\emph{(i)} While Theorem~\ref{Thm:rates} (proved in Section~\ref{subsec:ratesproof}) provides an asymptotic behavior for $\Vert
f_{\lambda,n}-f_0\Vert_\eu{H}$ under conditions that depend on $p_0$ (and are therefore not
easy to check in practice), a non-asymptotic bound on $\Vert
f_{\lambda,n}-f_0\Vert_\eu{H}$ that holds for all $n\ge 1$ can be obtained under stronger assumptions 
through an application of
Bernstein's inequality in separable Hilbert spaces. For the sake of simplicity, we provided asymptotic results which are obtained through an application
of Chebyshev's inequality.\vspace{2mm}\\
\emph{(ii)} The proof of Theorem~\ref{Thm:rates}(i) involves decomposing $\Vert f_{\lambda,n}-f_0\Vert_\eu{H}$ into an estimation error part, 
$\Cal{E}(\lambda,n):=\Vert f_{\lambda,n}-f_\lambda\Vert_\eu{H}$, and an approximation error part, $\Cal{A}_0(\lambda):=\Vert f_\lambda-f_0\Vert_\eu{H}$, 
where $f_\lambda=(C+\lambda I)^{-1}Cf_0$. While $\Cal{E}(\lambda,n)\rightarrow 0$ as $\lambda\rightarrow 0$, $\lambda\sqrt{n}\rightarrow\infty$ and $n\rightarrow\infty$ without any assumptions on $f_0$ 
(see the proof in Section~\ref{subsec:ratesproof} for details), it is not reasonable to expect $\Cal{A}_0(\lambda)\rightarrow 0$ as $\lambda\rightarrow 0$ without assuming $f_0\in\overline{\Cal{R}(C)}$. This is because, if $f_0$ lies
in the null space of $C$, then $f_\lambda$ is zero
irrespective of $\lambda$ and therefore cannot approximate $f_0$.\vspace{2mm}\\
\emph{(iii)} The condition $f_0\in\overline{\Cal{R}(C)}$ is
difficult to check in practice as it depends on $p_0$ (which
in turn depends on $f_0$). However, since the null space of $C$ is just
constant functions if the kernel is bounded and $\emph{supp}(q_0)=\Omega$ (see
Lemma~\ref{lem:support} in Section~\ref{subsec:densityproof} for details),
assuming $1\notin\eu{H}$ yields
that $\overline{\Cal{R}(C)}=\eu{H}$ and therefore consistency can be attained
under conditions that are easy to impose in practice. 
As mentioned in
Remark~\ref{rem:assumptions}(iii), the condition $1\notin\eu{H}$ ensures
identifiability and a sufficient condition for it to hold is $k\in
C_0(\Omega\times\Omega)$, which is satisfied by Gaussian, Mat\'{e}rn and inverse
multiquadric kernels. 
\vspace{2mm}\\
\emph{(iv)} It is well known that convergence rates are possible only if the
quantity
of interest (here $f_0$) satisfies some additional conditions. In function
estimation, this additional condition is classically imposed by assuming $f_0$
to be sufficiently smooth, e.g., $f_0$ lies in a Sobolev space of certain smoothness. By contrast, the smoothness condition in Theorem~\ref{Thm:rates}(ii) is imposed in an
indirect manner by assuming $f_0\in\Cal{R}(C^\beta)$ for some
$\beta>0$---so that the results hold for abstract RKHSs and not just Sobolev spaces---which then provides a rate, with the best rate being
$n^{-1/4}$ that is
attained when $\beta\ge 1$. While such a condition has already been used in
various works \citep{Caponnetto-07,Smale-07,Fukumizu-13} in the context of
non-parametric least squares regression, 
we
explore it in more detail in Proposition~\ref{pro:range}, and Examples~\ref{exm:gaussian} and \ref{exm:matern}. Note that this condition is common in the inverse problem theory (see \citealp*{Engl-96}), and it
naturally arises here through the connection of $f_{\lambda,n}$ being a Tikhonov
regularized solution to the ill-posed linear system $\hat{C}f=-\hat{\xi}$. An
interesting observation about the rate is that it does not improve with increasing $\beta$ (for $\beta>1$), in contrast to the classical results in function estimation 
(e.g., kernel density estimation and kernel regression) where the rate improves with increasing
smoothness. This issue is discussed in detail in
Section~\ref{subsec:regularizer}.
\vspace{2mm}\\
\emph{(v)} Since $\Vert C^{-1}\Vert<\infty$ only if $\eu{H}$ is
finite-dimensional, we recover the parametric rate of $n^{-1/2}$ in a
finite-dimensional situation with an automatic choice for $\lambda$ as
$n^{-1/2}$.\vspace{-1mm}
\end{rem}
While Theorem~\ref{Thm:rates} provides statistical guarantees for parameter
convergence, the question of primary interest is the
convergence of $p_{f_{\lambda,n}}$ to $p_0$. This is
guaranteed by
the following result, which is proved in Section~\ref{subsec:densityproof}.
\begin{theorem}[Consistency and rates for
$p_{f_{\lambda,n}}$]\label{Thm:density}
Suppose \emph{\textbf{(A)}}--\emph{\textbf{(D)}} with $\varepsilon=2$ hold and $\Vert
k\Vert_\infty:=\sup_{x\in\Omega}k(x,x)<\infty$. Assume $\emph{supp}(q_0)=\Omega$. Then the following hold:\vspace{2mm}\\
(i) For any $1< r\le\infty$
with $q_0\in L^1(\Omega)\cap L^r(\Omega)$, 
$$\Vert p_{f_{\lambda,n}}-p_0\Vert_{L^r(\Omega)}\rightarrow
0,\,h(p_{f_{\lambda,n}},p_0)\rightarrow
0,\,KL(p_0\Vert p_{f_{\lambda,n}})\rightarrow
0\,\,\,\text{as}\,\,\,\lambda\sqrt{n}
\rightarrow \infty,\,\lambda\rightarrow
0\,\,\text{and}\,\,n\rightarrow \infty.$$ In addition, if $f_0\in
\Cal{R}(C^\beta)$ for some $\beta>0$, then for $\lambda=n^{-\max\left\{\frac{1}{4},\frac{1}{2(\beta+1)}\right\}}$,
$$\Vert p_{
f_{\lambda,n}}-p_0\Vert_{L^r(\Omega)}=O_{p_0}(\theta_n),\,h(p_0,p_{f_
{\lambda,n}})=O_{p_0}(\theta_n),\,KL(p_0\Vert p_{f_{\lambda,n}})=O_{p_0}
(\theta^2_n)$$
as $n\rightarrow\infty$ where $\theta_n:=n^{-\min\left\{\frac{1}{4},\frac {\beta}{2(\beta+1)}\right\}}$.\vspace{2mm}
\\
(ii) 
$J(p_0\Vert p_{f_{\lambda,n}})\rightarrow
0\,\,\text{as}\,\,\lambda n
\rightarrow \infty,\,\lambda\rightarrow
0\,\,\text{and}\,\,n\rightarrow \infty.$ In addition, if $f_0\in
\Cal{R}(C^\beta)$ for some $\beta\ge 0$, then for $\lambda=n^{-\max\left\{\frac{1}{3},\frac{1}{2(\beta+1)}\right\}}$,
$$J(p_0\Vert p_{f_{\lambda,n}})=O_{p_0}\left(n^{-\min\left\{\frac{2}{3},
\frac{2\beta+1}{2(\beta+1)}\right\}}\right)\,\,\,\text{as}\,\,\,n\rightarrow\infty.$$
(iii) If $\Vert C^{-1}\Vert<\infty$,
then
$\theta_n=n^{-\frac{1}{2}}$ and 
$J(p_0\Vert p_{f_{\lambda,n}})=O_{p_0}(n^{-1})$ with $\lambda=n^{-\frac{1}{2}}$.\vspace{-1mm}
\end{theorem}
\begin{rem}\label{rem:density}
\emph{(i)} Comparing
the results of Theorem~\ref{Thm:rates}(i) and Theorem~\ref{Thm:density}(i) (for
$L^r$, Hellinger and KL divergence), we would like to
highlight that while the conditions on $\lambda$ and $n$ match in both the
cases, the latter does not require $f_0\in\overline{\Cal{R}(C)}$ to ensure
consistency. While $f_0\in\overline{\Cal{R}(C)}$ can be imposed in
Theorem~\ref{Thm:density} to attain consistency, we replaced this condition
with $\emph{supp}(q_0)=\Omega$---a simple and easy condition to work
with---which along with the boundedness of the kernel
ensures that for any $f_0\in\eu{H}$, there
exists
$\tilde{f_0}\in\overline{\Cal{R}(C)}$ such that $p_{\tilde{f_0}}=p_0$ (see
Lemma~\ref{lem:support}). \vspace{2mm}\\
\emph{(ii)} In contrast to the results in $L^r$, Hellinger and KL divergence, consistency in $J$ 
can be obtained with $\lambda$ converging to
zero at a rate faster than in these results. In addition, one can obtain rates in $J$ with
$\beta=0$, i.e., no smoothness assumption on $f_0$, while no rates are possible
in other distances (the latter might also be an artifact of the proof
technique, as these results are obtained through an application of
Theorem~\ref{Thm:rates}(ii) in Lemma~\ref{lem:distances}) which is due to the fact that the
convergence in these other distances is based on the convergence of $\Vert
f_{\lambda,n}-f_0\Vert_\eu{H}$, which in turn involves convergence of
$\Cal{A}_0(\lambda):=\Vert f_\lambda-f_0\Vert_\eu{H}$ to zero 
while the convergence in $J$ is
controlled by
$\Cal{A}_{\frac{1}{2}}(\lambda):=\Vert \sqrt{C}(f_\lambda-f_0)\Vert_\eu{H}$ which can be shown to behave as $O(\sqrt{\lambda})$ as $\lambda\rightarrow 0$, without requiring
any assumptions on $f_0$ (see
Proposition~\ref{pro:approxerror}). 
Indeed, as a further consequence, 
the rate of convergence in $J$ is faster than in other distances. \vspace{2mm}\\
\emph{(iii)} An interesting aspect in Theorem~\ref{Thm:density} is that $p_{f_{\lambda,n}}$ is consistent in various distances such as $L^r$, Hellinger
and KL, despite being obtained by minimizing a different loss
function, i.e., $J$. However, we will see in Section~\ref{Sec:misspecified}
that such nice results are difficult to obtain in the misspecified case, where consistency and rates are provided
only in $J$.\vspace{-2mm} 
\end{rem}
While Theorem~\ref{Thm:density} addresses the case of bounded kernels, the case
of unbounded kernels requires a technical modification. The
reason for this modification, as alluded to in the discussion following
Theorem~\ref{Thm:score}, is due to the fact that $f_{\lambda,n}$ may not be in
$\Cal{F}$ when $k$ is unbounded, and therefore the corresponding density
estimator, $p_{f_{\lambda,n}}$ may not be well-defined. In order to keep the main ideas intact, we discuss the unbounded case in detail in Section~\ref{subsec:unbounded-kernel} in 
Appendix~\ref{Sec:supp-results}.
\vspace{-1mm}
\subsection{Range Space Assumption}\label{subsec:range}
\par While
Theorems~\ref{Thm:rates} and \ref{Thm:density} are satisfactory from the point
of view of consistency, we believe the presented rates are possibly not minimax optimal since these rates are valid for any RKHS that satisfies
the conditions \textbf{(A)}--\textbf{(D)} and does not capture the smoothness of $k$ (and therefore the corresponding $\eu{H}$). In other words, the rates presented in 
Theorems~\ref{Thm:rates} and \ref{Thm:density} should depend on the decay rate of the eigenvalues of $C$ which in turn effectively captures the smoothness of $\eu{H}$. However, we 
are not able to obtain such a result---see the remark following the proof of Theorem~\ref{Thm:rates} for a discussion. While
these rates do not reflect the intrinsic smoothness of $\eu{H}$, they are obtained under 
the smoothness assumption, i.e., \emph{range space condition} that
$f_0\in\Cal{R}(C^\beta)$ for some $\beta>0$. This condition is quite different from the
classical smoothness conditions that appear in non-parametric function estimation. 
While the range space
assumption has been made
in various earlier works (e.g., \cite{Caponnetto-07,Smale-07,Fukumizu-13} in
the context of non-parametric least square regression), in the following, we investigate the implicit smoothness
assumptions that it makes on $f_0$ in our context. To this end, first it is easy to show (see the proof of Proposition~\ref{Thm:interpolation} 
in Section~\ref{sec:interpolate}) that
\begin{equation}\Cal{R}(C^\beta)=\left\{\sum_{i\in I}c_i\phi_i\,:\,\sum_{i\in
I}c^2_i\alpha^{-2\beta}_i<\infty\right\},\label{Eq:range-smooth}
\end{equation}
where $(\alpha_i)_{i\in I}$ are the positive eigenvalues of $C$,
$(\phi_i)_{i\in I}$ are the corresponding
eigenvectors that form an orthonormal basis for $\Cal{R}(C)$, and $I$ is an index
set which is either finite (if $\eu{H}$
is finite-dimensional) or $I=\bb{N}$ with $\lim_{i\rightarrow\infty}\alpha_i=0$
(if $\eu{H}$ is infinite dimensional). From (\ref{Eq:range-smooth}) it is
clear that larger the value of $\beta$, the faster is the decay of
the Fourier coefficients $(c_i)_{i\in I}$, which in turn implies that the
functions in $\Cal{R}(C^\beta)$ are smoother. Using (\ref{Eq:range-smooth}), an
interpretation can be provided for $\Cal{R}(C^\beta)$ ($\beta>0$ and
$\beta\notin\bb{N}$) as interpolation spaces (see Section~\ref{subsec:interpolation}
for the definition of interpolation spaces) between
$\Cal{R}(C^{\lceil\beta\rceil})$ and $\Cal{R}(C^{\lfloor\beta\rfloor})$
where $\Cal{R}(C^0):=\eu{H}$ (see Proposition~\ref{Thm:interpolation} for details). 
While it is not completely straightforward to
obtain a sufficient
condition for $f_0\in \Cal{R}(C^\beta)$, $\beta\in\bb{N}$, the following result
provides a necessary condition for $f_0\in \Cal{R}(C)$ (and therefore a
necessary condition for
$f_0\in\Cal{R}(C^\beta),\,\forall\,\beta>1$) for translation invariant kernels
on $\Omega=\bb{R}^d$, whose proof is presented in
Section~\ref{subsec:supp-rangeproof}.
\begin{proposition}[Necessary condition]\label{pro:range}
Suppose $\psi,\phi\in C_b(\bb{R}^d)\cap
L^1(\bb{R}^d)$ are positive definite functions on $\bb{R}^d$ with Fourier
transforms $\psi^\wedge$ and $\phi^\wedge$ respectively. Let $\eu{H}$ and
$\eu{G}$ be the RKHSs associated with $k(x,y)=\psi(x-y)$ and
$l(x,y)=\phi(x-y),x,y\in\bb{R}^d$ respectively. For $1\le r\le 2$, suppose the
following hold: \vspace{-2mm}
\begin{itemize}
\item[(i)] $\int_{\bb{R}^d} \Vert
\omega\Vert^2_2\psi^\wedge(\omega)\,d\omega<\infty$; (ii)
$\left\Vert\frac{\phi^\wedge}{\psi^\wedge}\right\Vert_\infty<\infty$;
(iii) $\frac{\Vert\cdot\Vert^2_2(\psi^\wedge)^2}{\phi^\wedge}\in
L^{\frac{r}{2-r}}(\bb{R}^d)$; (iv) $q_0\in L^r(\bb{R}^d)$.\vspace{-2mm}
\end{itemize}
Then $f_0\in \Cal{R}(C)$ implies $f_0\in \eu{G}\subset\eu{H}$.\vspace{-2mm}
\end{proposition}
In the following, we apply the above result in two examples involving
Gaussian and Mat\'{e}rn kernels to get insights into the range space
assumption.
\begin{example}[Gaussian kernel]\label{exm:gaussian}
Let $\psi(x)=e^{-\sigma\Vert x\Vert^2}$ with $\eu{H}_\sigma$ as its
corresponding RKHS (see Section~\ref{Sec:notation} for its definition).
 By Proposition~\ref{pro:range}, it is easy to verify
that $f_0\in\Cal{R}(C)$ implies $f_0\in\eu{H}_\alpha\subset\eu{H}_\sigma$ for
$\frac{\sigma}{2}<\alpha\le\sigma$. Since $\eu{H}_\beta\subset \eu{H}_\gamma$
for $\beta<\gamma$ (i.e., Gaussian RKHSs are nested), $f_0\in \Cal{R}(C)$
ensures that $f_0$ lies
in $\eu{H}_{\frac{\sigma}{2}+\epsilon}$ for arbitrary small $\epsilon>0$.
\vspace{-2mm}
\end{example}
\begin{example}[Mat\'{e}rn kernel]\label{exm:matern}
Let
$\psi(x)=\frac{2^{1-s}}{\Gamma(s)}\Vert x\Vert^{s-\frac{d}{2}}_2\frak{K}_{d/2-s}
(\Vert x\Vert_2),\,x\in\bb{R}^d$ with $H^s_2(\bb{R}^d)$ as its
corresponding RKHS (see Section~\ref{Sec:notation} for its definition) where
$s>\frac{d}{2}$. 
By Proposition~\ref{pro:range}, we have
that for $q_0\in L^1(\bb{R}^d)$, if 
$f_0\in\Cal{R}(C)$,
then $f_0\in H^\alpha_2(\bb{R}^d)\subset
H^s_2(\bb{R}^d)$ for $1+\frac{d}{2}<s\le\alpha<2s-1-\frac{d}{2}$. 
Since
$H^\delta_2(\bb{R}^d)\subset H^\gamma_2(\bb{R}^d)$ for $\gamma<\delta$ (i.e.,
Sobolev spaces are nested), this means
$f_0$ lies in $H^{2s-1-\frac{d}{2}-\epsilon}_2(\bb{R}^d)$ 
for
arbitrarily small $\epsilon>0$, i.e., $f_0$ has
at least $2s-1-\lceil\frac{d}{2}\rceil$ 
weak-derivatives. By the minimax theory \cite[Chapter 2]{Tsybakov-09}, it is well known that 
for any $\alpha>\delta\ge 0$,
\begin{equation}\inf_{\hat{f}_n}\sup_{f_0\in H^\alpha_2(\bb{R}^d)}\Vert
\hat{f}_n-f_0\Vert_{H^{\delta}_2(\bb{R}^d)}\asymp
n^{-\frac{\alpha-\delta}{2(\alpha-\delta)+d}},\label{Eq:minimax}\end{equation}
where the infimum is taken over
all possible estimators. Here $a_n\asymp b_n$ means that for any two sequences $a_n,b_n>0$, $a_n/b_n$ 
is bounded away from zero and infinity as $n\rightarrow\infty$. 
Suppose $f_0\notin H^\alpha_2(\bb{R}^d)$ for
$\alpha\ge 2s-1-\frac{d}{2}$, 
which means $f_0\in H^{2s-1-\frac{d}{2}-\epsilon}_2(\bb{R}^d)$
for arbitrarily small
$\epsilon>0$.
This implies that the rate of $n^{-1/4}$ obtained in Theorem~\ref{Thm:rates} is
minimax optimal if $\eu{H}$ is chosen to be
$H^{1+d+\epsilon}_2(\bb{R}^d)$ 
(i.e., choose $\alpha=2s-1-\frac{d}{2}-\epsilon$ and $\delta=s$ in (\ref{Eq:minimax}) and
 solve for $s$
 by equating the exponent in the r.h.s.~of (\ref{Eq:minimax}) to $-\frac{1}{4}$). Similarly, it can be shown that
 if $q_0\in L^2(\bb{R}^d)$, then the rate of $n^{-1/4}$ in Theorem~\ref{Thm:rates} is
minimax optimal if $\eu{H}$ is chosen to be
$H^{1+\frac{d}{2}+\epsilon}_2(\bb{R}^d)$. This example also explains away the dimension
independence of the rate provided by
Theorem~\ref{Thm:rates} by showing
that the dimension effect is captured in the relative smoothness
of $f_0$ w.r.t.~$\eu{H}$.
\vspace{-2mm}
\end{example}
While Example~\ref{exm:matern} provides some understanding about the minimax
optimality of $f_{\lambda,n}$ under additional assumptions on $f_0$, the
problem is not completely resolved. 
In the following section, however, we show that
the rate in
Theorem~\ref{Thm:rates} is not optimal for
$\beta>1$, and that improved rates can be obtained by choosing the
regularizer appropriately.
%
\subsection{Choice of Regularizer}\label{subsec:regularizer}
We understand from the characterization of $\Cal{R}(C^\beta)$ in
(\ref{Eq:range-smooth}) that
larger $\beta$ values yield smoother functions in $\eu{H}$. 
However, the smoothness
of $f_0\in \Cal{R}(C^\beta)$ for $\beta>1$ is not captured in the rates in
Theorem~\ref{Thm:rates}(ii), where the rate saturates at $\beta=1$ providing the
best possible rate of $n^{-1/4}$ (irrespective of the size of $\beta$).
This is unsatisfactory on the part of the estimator, as it does
not effectively capture the smoothness of $f_0$, i.e., the estimator is not adaptive to the smoothness of $f_0$. We remind the
reader that the estimator $f_{\lambda,n}$ is obtained by minimizing the
regularized empirical Fisher divergence (see Theorem~\ref{Thm:score}(iv))
yielding
$f_{\lambda,n}=-(\hat{C}+\lambda I)^{-1}\hat{\xi}$, which can be seen as a
heuristic to solve the (non-linear) inverse problem $Cf_0=-\xi$ (see
Theorem~\ref{Thm:score}(ii)) from finite samples, by replacing $C$ and $\xi$ with
their empirical counterparts. This heuristic, which ensures that the finite
sample inverse problem is well-posed, is popular in inverse problem literature
under the name of Tikhonov regularization \cite[Chapter 5]{Engl-96}. Note
that Tikhonov regularization helps to make the ill-posed inverse problem a
well-posed one by approximating $\alpha^{-1}$ by
$(\alpha+\lambda)^{-1}$, $\lambda>0$, where $\alpha^{-1}$ appears
as the inverse of the eigenvalues of $C$ while computing $C^{-1}$. 
In
other
words, if $\hat{C}$ is invertible, then an estimate of $f_0$ can be obtained as
$\hat{f}_n=-\hat{C}^{-1}\hat{\xi}$, i.e.,
$\hat{f}_n=-\sum_{i\in I}\frac{\langle
\hat{\xi},\hat{\phi}_i\rangle_\eu{H}}{\hat{\alpha}_i}\hat{\phi}_i,$ where
$(\hat{\alpha}_i)_{i\in I}$ and $(\hat{\phi}_i)_{i\in I}$ are the
eigenvalues and
eigenvectors of $\hat{C}$ respectively. However, $\hat{C}$ being a rank $n$
operator
defined on $\eu{H}$ (which can be
infinite dimensional) is not invertible and 
therefore the regularized
estimator is 
constructed as $f_{\lambda,n}=-g_\lambda(\hat{C})\hat{\xi}$ where
$g_\lambda(\hat{C})$ is defined through functional calculus (see \citealp*[Section 2.3]{Engl-96}) as 
$$g_\lambda(\hat{C})=\sum_{i\in I}g_\lambda(\hat{\alpha}_i)\langle
\cdot,\hat{\phi}_i\rangle_\eu{H}\hat{\phi}_i$$
with $g_\lambda:\bb{R}_+\rightarrow\bb{R}$ and 
$g_\lambda(\alpha):=(\alpha+\lambda)^{-1}$. Since the Tikhonov regularization
is well-known to saturate (as explained above)---see \citet[Sections 4.2 and
5.1]{Engl-96} for details---, better approximations to $\alpha^{-1}$ have been
used in the inverse problems literature to improve the rates by using
$g_\lambda$ other than $(\cdot+\lambda)^{-1}$ where $g_\lambda(\alpha)\rightarrow \alpha^{-1}$ as $\lambda\rightarrow
0$. In the
statistical context, \citet{Rosasco-05} and \citet{Bauer-07} have used the ideas from
\cite{Engl-96} in non-parametric regression for learning a square integrable
function from finite samples through regularization in RKHS. In the following,
we use these ideas to construct an alternate estimator for $f_0$ (and
therefore for $p_0$) that appropriately captures the smoothness of $f_0$ by
providing a better convergence rate when $\beta>1$. To this end, we need the
following assumption---quoted from \citet[Theorems 4.1--4.3 and Corollary 4.4]{Engl-96} and
\citet[Definition 1]{Bauer-07}---that is standard in the theory of inverse problems.
\begin{itemize}
\item[\textbf{(E)}] There exists finite positive constants $A_g$, $B_g$,
$C_g$, $\eta_0$ and $(\gamma_\eta)_{\eta\in (0,\eta_0]}$ (all independent of $\lambda>0$) such that
$g_\lambda:[0,\chi]\rightarrow\bb{R}$
satisfies: \vspace{1mm}\\ $(a)\, \sup_{\alpha\in \Cal{D}} |\alpha g_\lambda(\alpha)|\le
A_g$, $(b)\, \sup_{\alpha\in
\Cal{D}}|g_\lambda(\alpha)|\le\frac{B_g}{\lambda}$, $(c)\,
\sup_{\alpha\in\Cal{D}}|1-\alpha
g_\lambda(\alpha)|\le C_g$ and $(d)\,\sup_{\alpha\in\Cal{D}}|1-\alpha
g_\lambda(\alpha)|\alpha^\eta\le\gamma_\eta\lambda^\eta,\,\,\forall\,
\eta\in (0,\eta_0]$ where $\Cal{D}:=[0,\chi]$ and $\chi:=d\sup_{x\in\Omega,i\in[d]}\partial_i\partial_{i+d}k(x,x)<\infty$.
\end{itemize}
The constant $\eta_0$ is called the \emph{qualification} of
$g_\lambda$ which is what determines the point of saturation of $g_\lambda$. We
show in Theorem~\ref{Thm:rates-new} that if $g_\lambda$ has a
finite qualification, then the resultant estimator cannot fully exploit the
smoothness of $f_0$ and therefore the rate of convergence will suffer for
$\beta>\eta_0$. Given $g_\lambda$ that satisfies \textbf{(E)}, we
construct our estimator of $f_0$ as
\begin{equation}
 f_{g,\lambda,n}=-g_\lambda(\hat{C})\hat{\xi}.\nonumber
\end{equation}
Note that the above estimator can be obtained by using the data dependent
regularizer, $\frac{1}{2}\langle
f,((g_\lambda(\hat{C}))^{-1}-\hat{C})f\rangle_\eu{H}$ in the minimization of
$\hat{J}(f)$ defined in Theorem~\ref{Thm:score}(iv), i.e.,
$$f_{g,\lambda,n}=\arg\inf_{f\in\eu{H}}\hat{J}(f)+\frac{1}{2}\langle
f,((g_\lambda(\hat{C}))^{-1}-\hat{C})f\rangle_\eu{H}.$$ 
However, unlike $f_{\lambda,n}$ for which a simple form is available in
Theorem~\ref{Thm:representer} by solving a linear system, we are not able to
obtain such a nice expression for $f_{g,\lambda,n}$. The
following result (proved in Section~\ref{subsec:supp-rates-new}
) presents an analog of Theorems~\ref{Thm:rates} and
\ref{Thm:density} for the new
estimators, $f_{g,\lambda,n}$ and $p_{f_{g,\lambda,n}}$.
\begin{theorem}[Consistency and convergence rates for
$f_{g,\lambda,n}$ and $p_{f_{g,\lambda,n}}$]\label{Thm:rates-new}
Suppose \textbf{\emph{(A)}}--\textbf{\emph{(E)}} hold with $\varepsilon=2$. 
\vspace{2mm}\\
(i) If $f_0\in\Cal{R}(C^\beta)$ for some $\beta>0$, then for any $\lambda\ge n^{-1/2}$, $$\Vert
f_{g,\lambda,n}-f_0\Vert_\eu{H}=O_{p_0}\left(\theta_n \right),$$ where $\theta_n:=n^{-\min\left\{\frac{\beta}{2(\beta
+1)},\frac{\eta_0}{2(\eta_0+1)}\right\}}$ 
with
$\lambda=n^{-\max\left\{\frac{1}{2(\beta+1)},\frac{1}{
2(\eta_0+1) } \right\}}$. In addition, if $\Vert k\Vert_\infty<\infty$, then for any $1<
r\le\infty$ with $q_0\in L^1(\Omega)\cap L^r(\Omega)$, $$\Vert p_{
f_{g,\lambda,n}}-p_0\Vert_{L^r(\Omega)}=O_{p_0}(\theta_n),\, h(p_0,p_{f_
{g,\lambda,n}})=O_{p_0}(\theta_n)\,\,\,\text{and}\,\,\,KL(p_0\Vert p_{f_{g,\lambda,n}})=O_{p_0}
(\theta^2_n).$$ 
(ii) If $f_0\in\Cal{R}(C^\beta)$ for some $\beta\ge 0$, then for any $\lambda\ge n^{-1/2}$,
$$J(p_0\Vert p_{f_{g,\lambda,n}})=O_{p_0}\left(n^{-\frac{\min\{2\beta+1 ,
2\eta_0\}}{\min\{ 2\beta+2 , 2\eta_0+1\}}}\right)$$
with
$\lambda=n^{-\frac{1}{\min\{2\beta+2,2\eta_0+1\}}}$.\vspace{1mm}\\
(iii) If $\Vert C^{-1}\Vert<\infty$, 
then for any $\lambda\ge n^{-1/2}$,
$$\Vert
f_{g,\lambda,n}-f_0\Vert_\eu{H}=O_{p_0}(\theta_n)\,\,\,\text{and}\,\,\,J(p_0\Vert p_{f_{g,\lambda,n}})=O_{p_0}(\theta^2_n)$$ with $\theta_n=n^{-\frac{1}{2}}$ 
and
$\lambda=n^{-\frac{1}{\min\{2,2\eta_0\}}}$.
\end{theorem}
Theorem~\ref{Thm:rates-new} shows that if $g_\lambda$ has infinite
qualification, then smoothness of $f_0$ is fully captured in the rates and as
$\beta\rightarrow \infty$, we attain $O_{p_0}(n^{-1/2})$ rate for $\Vert
f_{g,\lambda,n}-f_0\Vert_\eu{H}$ in contrast to $n^{-1/4}$ (similar improved
rates are also obtained for $p_{f_{g,\lambda,n}}$ in various distances) in
Theorem~\ref{Thm:rates}. 
In the following example, we present two choices of $g_\lambda$ that improve on
Tikhonov regularization. We refer the reader to \citet[Section 3.1]{Rosasco-05}
for more examples of $g_\lambda$.
\begin{example}[Choices of $g_\lambda$] (i) Tikhonov regularization
involves $g_\lambda(\alpha)=(\alpha+\lambda)^{-1}$ for which it is easy to verify that $\eta_0=1$ and 
therefore the rates saturate at $\beta=1$, leading to the results in
Theorems~\ref{Thm:rates} and \ref{Thm:density}.
\vspace{1mm}\\
(ii) Showalter's method and spectral cut-off 
use $$g_\lambda(\alpha)=\frac{1-e^{-\alpha/\lambda}}{\alpha}\,\,\,\text{and}\,\,\,g_\lambda(\alpha)=\begin{cases}\frac{1}{\alpha}, &\mbox {} 
\alpha\ge\lambda\\
 0,&\mbox{} \alpha<\lambda\end{cases}$$ respectively for which it is easy to verify that $\eta_0=+\infty$ (see \citealp*[Examples 4.7 \& 4.8]{Engl-96} for
details) and therefore improved rates are obtained for $\beta>1$ in Theorem~\ref{Thm:rates-new} compared to
that of Tikhonov regularization.
\end{example}

\section{Density Estimation in $\Cal{P}$: Misspecified Case}
\label{Sec:misspecified}
In
this section, we analyze the misspecified case where $p_0\notin\Cal{P}$,
which is a more reasonable case than the well-specified one, as in
practice it is not easy to check whether $p_0\in\Cal{P}$. To this end, we
consider the same estimator $p_{f_{\lambda,n}}$ as considered in the
well-specified case where $f_{\lambda,n}$ is obtained from
Theorem~\ref{Thm:representer}. The following result shows that
$J(p_0\Vert p_{f_{\lambda,n}})\rightarrow \inf_{p\in\Cal{P}}J(p_0\Vert p)$ as
$\lambda\rightarrow 0$, $\lambda n\rightarrow\infty$ and $n\rightarrow \infty$ under the assumption that there exists $f^*\in\Cal{F}$ such
that $J(p_0\Vert p_{f^*})=\inf_{p\in\Cal{P}}J(p_0\Vert p)$. We present the result for bounded kernels although it can be easily extended to
unbounded kernels as in Theorem~\ref{Thm:density-2}. Also, the presented result for Tikhonov regularization extends easily to $p_{f_{g,\lambda,n}}$
using the ideas in the proof of Theorem~\ref{Thm:rates-new}. Note that unlike in the well-specified case where convergence in other
distances can be shown even though the estimator is constructed from $J$, it is difficult to show such a result in the
misspecified case.
\begin{theorem}\label{Thm:misspecified}
Let $p_0,\,q_0\in C^1(\Omega)$ be
probability densities such that $J(p_0\Vert q_0)<\infty$
where $\Omega$
satisfies \textbf{\emph{(A)}}. Assume that
\textbf{\emph{(B)}}, \textbf{\emph{(C)}} and \textbf{\emph{(D)}} with
$\varepsilon=2$ hold. Suppose $\Vert k\Vert_\infty<\infty$, $\emph{supp}(q_0)=\Omega$
 and
there exists $f^\ast\in \Cal{F}$
such that
$$J(p_0\Vert p_{f^*})=\inf_{p\in\Cal{P}}J(p_0\Vert p).$$ Then for an estimator
$p_{f_{\lambda,n}}$ constructed from random samples $(X_a)^n_{a=1}$ drawn
i.i.d.~from $p_0$, where $f_{\lambda,n}$ is defined in
(\ref{Eq:representer-f})---also see Theorem~\ref{Thm:score}(iv)---with
$\lambda>0$, we have
$$J(p_0\Vert p_{f_{\lambda,n}})\rightarrow
\inf_{p\in\Cal{P}}J(p_0\Vert p)\,\,\,\text{as}\,\,\lambda\rightarrow 0,\,\lambda
n\rightarrow\infty\,\,\text{and}\,\,n\rightarrow\infty.$$
In addition, if $f^*\in\Cal{R}(C^\beta)$ for some
$\beta\ge 0$, then
$$\sqrt{J(p_0\Vert p_{f_{\lambda,n}})}\le
\sqrt{\inf_{p\in\Cal{P}}J(p_0\Vert p)}+O_{p_0}\left(n^{-\min\left\{\frac{1}{3},
\frac{2\beta+1}{4(\beta+1)}\right\}}\right)$$ with
$\lambda=n^{-\max\left\{\frac{1}{3},\frac{1}{2(\beta+1)}\right\}}$. If $\Vert
C^{-1}\Vert<\infty$, then for $\lambda=n^{-\frac{1}{2}}$,
$$\sqrt{J(p_0\Vert p_{f_{\lambda,n}})}\le
\sqrt{\inf_{p\in\Cal{P}}J(p_0\Vert p)}+O_{p_0}(n^{-1/2}).$$
with
$\lambda=n^{-\frac{1}{2}}$.
\end{theorem}
While the above result is useful and interesting, the assumption
about the existence of $f^*$ is quite restrictive. This is because if 
$p_0$ (which is not in $\Cal{P}$) belongs to a family $\Cal{Q}$ where
$\Cal{P}$ is dense in
$\Cal{Q}$ w.r.t.~$J$, then there is no $f\in\eu{H}$ that attains the
infimum, i.e., $f^*$ does not exist and therefore the proof
technique employed in
Theorem~\ref{Thm:misspecified} will fail. In the following, we present a result
(Theorem~\ref{Thm:misspecified-1}) that does not require the existence of
$f^\ast$ but attains the same result as
in Theorem~\ref{Thm:misspecified}, but requiring a more complicated
proof.
Before we present
Theorem~\ref{Thm:misspecified-1}, we need to introduce some notation. 

To
this end, let us return to the objective function under consideration, 
$$J(p_0\Vert p_f)=\frac{1}{2}\int_\Omega p_0(x) \left\Vert
\nabla\log \frac{p_0}{p_f}\right\Vert^2_2\,dx=
\frac{1}{2}\int_\Omega
p_0(x)\sum^d_{i=1}\left(\partial_i
f_\star-\partial_i f\right)^2\,dx,$$
where $f_\star=\log\frac{p_0}{q_0}$ 
and $p_0\notin\Cal{P}$. 
Define
\begin{equation}
\Cal{W}_2(\Omega,p_0):=\left\{f\in C^1(\Omega)\,:\,\partial^\alpha f\in
L^2(\Omega,p_0),\,\forall\,|\alpha|=1\right\}.\nonumber
\end{equation}
This is a reasonable class of
functions to consider as under the condition $J(p_0\Vert q_0)<\infty$, it is
clear that $f_\star\in\Cal{W}_2(\Omega,p_0)$. 
Endowed with a semi-norm, 
\begin{equation}\Vert
 f\Vert^2_{\Cal{W}_2}:=\sum_{|\alpha|=1}\Vert \partial^\alpha
f\Vert^2_{L^2(\Omega,p_0)},\nonumber
\end{equation}
$\Cal{W}_2(\Omega,p_0)$ is a vector space of functions, from which a normed space can
be constructed as follows. 
Let us
define $f,f^\prime\in \Cal{W}_2(\Omega,p_0)$ to be equivalent, i.e., 
$f\sim f^\prime$,
if $\Vert f-f^\prime\Vert_{\Cal{W}_2}=0$. In other words, $f\sim f^\prime$ if
and only if $f$ and $f^\prime$ differ by a constant
$p_0$-almost everywhere. Now define the quotient space $\Cal{W}^\sim_2(\Omega,p_0):=\left\{[f]_\sim:f\in\Cal{W}_2(\Omega,p_0)\right\}$ 
where $[f]_\sim:=\{f'\in \Cal{W}_2(\Omega,p_0):f\sim f^\prime\}$ denotes the
equivalence class of $f$. 
Defining $\Vert [f]_\sim\Vert_{\Cal{W}^\sim_2}:=\Vert f\Vert_{\Cal{W}_2}$, it is
easy to
verify that $\Vert\cdot\Vert_{\Cal{W}^\sim_2}$ defines a norm on
$\Cal{W}^\sim_2(p_0)$. In addition, endowing the following bilinear
form on $\Cal{W}^\sim_2(\Omega,p_0)$
$$\langle [f]_\sim,[g]_\sim\rangle_{\Cal{W}^\sim_2}:=\int_\Omega p_0(x)
\sum_{|\alpha|=1}(\partial^\alpha f)(x) (\partial^\alpha g)(x)\,dx$$
makes it a pre-Hilbert space. 
Let $W_2(\Omega,p_0)$ be the Hilbert
space obtained by completion of $\Cal{W}^\sim_2(\Omega,p_0)$. As shown in Proposition
\ref{pro:operators} below, under some assumptions, a continuous mapping $I_k:
\eu{H} \to W_2(\Omega,p_0), f\mapsto [f]_\sim$ can be defined, which is injective
modulo constant functions. Since addition of a constant does not contribute to
$p_f$, the space $W_2(\Omega,p_0)$ can be regarded as a parameter space extended from
$\eu{H}$. In addition to $I_k$, Proposition
\ref{pro:operators} (proved in Section~\ref{subsec:supp-operators}) describes the adjoint of $I_k$ and relevant
self-adjoint operators, which will be useful in analyzing $p_{f_{\lambda,n}}$ in Theorem~\ref{Thm:misspecified-1}.

\begin{proposition}\label{pro:operators}
Let $\emph{supp}(q_0)=\Omega$ where $\Omega\subset\bb{R}^d$ is non-empty and open. Suppose $k$ satisfies \textbf{\emph{(B)}} and $\partial_{i}\partial_{i+d}k(x,x)\in L^1(\Omega,p_0)$ 
for all $i\in[d]$.
Then $I_k:\eu{H}\to W_2(\Omega,p_0)$, $f\mapsto [f]_\sim$ defines a continuous mapping with the null space $\eu{H}\cap \mathbb{R}$. The
adjoint of $I_k$ is $S_k:W_2(\Omega,p_0)\rightarrow\eu{H}$ whose restriction to $\Cal{W}^\sim_2(\Omega,p_0)$ is
given by
$$
S_k[h]_\sim(y)= \int_\Omega \sum^d_{i=1}\partial_{i} k(x,y)\partial_i h(x)\,p_0(x)\,dx,\,\qquad\,[h]_\sim\in
\Cal{W}^\sim_2(\Omega,p_0),\,y\in\Omega.
$$
In addition, $I_k$ and $S_k$ are Hilbert-Schmidt and therefore compact. Also,
$E_k:=S_kI_k$ and $T_k:=I_kS_k$ are compact, positive and self-adjoint operators
on $\eu{H}$ and $W_2(\Omega,p_0)$ respectively where
$$
E_kg(y)=\int_\Omega \sum^d_{i=1}\partial_{i} k(x,y)\partial_i g(x)p_0(x)\,dx,\,\qquad\,g\in\eu{H},\,y\in\Omega
$$
and the restriction of $T_k$ to $\Cal{W}^\sim_2(\Omega,p_0)$ is given by
$$
T_k[h]_\sim=\left[\int_\Omega \sum^d_{i=1}\partial_{i} k(x,\cdot)\partial_i h(x)\,p_0(x)\,dx\right]_\sim,\,\qquad\,[h]_\sim\in
\Cal{W}^\sim_2(\Omega,p_0).
$$
\end{proposition}
Note that for $[h]_\sim\in\Cal{W}^\sim_2(\Omega,p_0)$, the derivatives $\partial_i h$
do not depend on the choice of a representative element almost surely
w.r.t.~$p_0$, and thus the above integrals are well defined. Having constructed $W_2(\Omega,p_0)$, it is clear that $J(p_0\Vert p_f)=\frac{1}{2}\Vert
[f_\star]_\sim - I_k f\Vert^2_{W_2}$, which means estimating $p_0$ is equivalent to
estimating $f_\star\in W_2(\Omega,p_0)$ by $f\in\Cal{F}$. With all these
preparations, we are now ready
to present a result (proved in Section~\ref{subsec:misspecified}) on consistency and convergence rate for $p_{f_{\lambda,n}}$ without
assuming the existence of $f^\ast$.
\begin{theorem}\label{Thm:misspecified-1}
Let $p_0,\,q_0\in C^1(\Omega)$ be
probability densities such that $J(p_0\Vert q_0)<\infty$. Assume that \textbf{\emph{(A)}}--\textbf{\emph{(D)}} hold with $\varepsilon=2$ 
and $\chi:=d\sup_{x\in\Omega,i\in[d]}\partial_i\partial_{i+d}k(x,x)<\infty$. 
Then the following hold.\vspace{2mm}\\
(i) 
As
$\lambda\rightarrow 0,\,\lambda
n\rightarrow\infty\,\,\text{and}\,\,n\rightarrow\infty,$
$J(p_0\Vert p_{f_{\lambda,n }})\rightarrow
\inf_{p\in\Cal{P}}J(p_0\Vert p).$\vspace{1mm}\\
(ii) Define $f_\star:=\log\frac{p_0}{q_0}$. If $[f_\star]_\sim\in
\overline{\Cal{R}(T_k)}$, then 
$$
J(p_0\Vert p_{f_{\lambda,n }})\rightarrow
0\,\,\,\text{as}\,\,\,\lambda\rightarrow 0,\,\lambda
n\rightarrow\infty\,\,\text{and}\,\,n\rightarrow\infty.$$ 
In addition, if
$[f_\star]_\sim\in\Cal{R}(T^\beta_k)$ for some $\beta> 0$, then for $\lambda=n^{-\max\left\{\frac{1}{3},\frac{1}{2\beta+1}\right\}}$
$$
J(p_0\Vert p_{f_{\lambda,n}})=O_{p_0}\Big(n^{-\min\left\{\frac{2}{3},
\frac{2\beta}{2\beta+1}\right\}}\Big)$$.
(iii) If $\Vert
E^{-1}_k\Vert<\infty$ and $\Vert T^{-1}_k\Vert<\infty$, then 
$
J(p_0\Vert p_{f_{\lambda,n}})=O_{p_0}\left(n^{-1}\right)
$ 
with $\lambda=n^{-\frac{1}{2}}$. 
\end{theorem}

%
\begin{rem}\label{rem:misspecified-compare} 
\emph{(i)} The result in
Theorem~\ref{Thm:misspecified-1}(ii) is
particularly interesting as it shows that $[f_\star]_\sim\in
W_2(\Omega,p_0)\backslash I_k(\eu{H})$ can be consistently estimated by
$f_{\lambda,n}\in\eu{H}$, which in turn implies that certain
$p_0\notin\Cal{P}$ can be consistently estimated by
$p_{f_{\lambda,n}}\in\Cal{P}$. In particular, if $S_k$ is injective,
then $I_k(\eu{H})$ is dense in $W_2(\Omega,p_0)$ w.r.t.~$\Vert\cdot\Vert_{W_2}$, which
implies $\inf_{p\in\Cal{P}}J(p_0||p)=0$ though there does not exist
$f^\ast\in\eu{H}$ for which $J(p_0||p_{f^*})=0$. While
Theorem~\ref{Thm:misspecified} cannot handle this situation, (i) and (ii)
in Theorem~\ref{Thm:misspecified-1} coincide showing that $p_0\notin\Cal{P}$ can
be consistently estimated by
$p_{f_{\lambda,n}}\in\Cal{P}$. While the question of when $I_k(\eu{H})$  is
dense in $W_2(\Omega,p_0)$ is open, we refer the reader to
Section~\ref{subsec:supp-dense} for a related
discussion.\vspace{2mm}\\
\emph{(ii)} Replicating the proof of Theorem 4.6 in \cite{Steinwart-12}, it is
easy
to show that for all $0<\gamma<1$,
$\Cal{R}(T^{\gamma/2}_k)=\left[W_2(\Omega,p_0),I_k(\eu{H})\right]_{\gamma,2}$,
where the r.h.s.~is an interpolation space obtained
through the real interpolation of $W_2(\Omega,p_0)$ and $I_k(\eu{H})$ (see
Section~\ref{subsec:interpolation} for the notation and definition). Here
$I_k(\eu{H})$ is endowed with the Hilbert space structure by $I_k(\eu{H})\cong
\eu{H}/\eu{H}\cap\mathbb{R}$. This interpolation space interpretation means
that, for $\beta\ge
\frac{1}{2}$, $\Cal{R}(T^\beta_k)\subset \eu{H}$ modulo constant functions. It
is nice to
note that the rates in Theorem~\ref{Thm:misspecified-1}(ii) for
$\beta\ge\frac{1}{2}$ match with the rates in Theorem~\ref{Thm:density} (i.e.,
the well-specified case) w.r.t.~$J$ for $0\le\beta\le\frac{1}{2}$. We
highlight the fact that $\beta=0$ corresponds to $\eu{H}$ in
Theorem~\ref{Thm:density} whereas $\beta=\frac{1}{2}$ corresponds to $\eu{H}$ in
Theorem~\ref{Thm:misspecified-1}(ii) and therefore the range of comparison is
for $\beta\ge\frac{1}{2}$ in Theorem~\ref{Thm:misspecified-1}(ii) versus
$0\le\beta\le\frac{1}{2}$ in Theorem~\ref{Thm:density}. In contrast,
Theorem~\ref{Thm:misspecified} is very limited as it only provides a
rate for the convergence of $J(p_0||p_{f_{\lambda,n}})$ to
$\inf_{p\in\Cal{P}}J(p_0||p)$ assuming that $f^*$ is sufficiently smooth.\vspace{-2mm}
\end{rem}
Based on the observation (i) in the above remark that 
$\inf_{p\in\Cal{P}}J(p_0\Vert p)=0$ if $I_k(\eu{H})$ is dense in $W_2(\Omega,p_0)$
w.r.t.~$\Vert\cdot\Vert_{W_2}$, it is possible to obtain an approximation result for $\Cal{P}$
(similar to those discussed in Section~\ref{Sec:approximation}) w.r.t.~Fisher divergence as shown
below, whose proof is provided in Section~\ref{subsec:supp-thm-approx-fd}.

\begin{proposition}\label{Thm:approx-fd}
Suppose $\Omega\subset\bb{R}^d$ is non-empty and open. Let $q_0\in C^1(\Omega)$ be a
probability density and
$$\Cal{P}_{\emph{FD}}:=\Big\{p\in C^1(\Omega):\int_\Omega p(x)\,dx=1, p(x)\ge
0,\,\forall\,x\in\Omega\,\,\emph{and}\,\,J(p\Vert q_0)<\infty 
\Big\}.$$ 
For any $p\in \Cal{P}_{\emph{FD}}$, if $I_k(\eu{H})$ is dense in
$W_2(\Omega,p)$ w.r.t.~$\Vert\cdot\Vert_{W_2}$, then for every
$\epsilon>0$, there exists
$\tilde{p}\in\Cal{P}$ such that $J(p\Vert \tilde{p})\le\epsilon$.
\end{proposition}
\section{Numerical Simulations}\label{Sec:experiments}
We have proposed an estimator of $p_0$ that is obtained by minimizing
the
regularized empirical Fisher divergence and presented its consistency along
with convergence rates. As discussed in Section~\ref{Sec:Introduction}, however one
can simply ignore the structure of $\Cal{P}$ and estimate $p_0$ in a completely
non-parametric fashion, for example using the kernel density estimator (KDE). In
fact, consistency and convergence rates of KDE are also
well-studied \cite[Chapter 1]{Tsybakov-09} and the kernel density
estimator is very simple to compute---requiring only $O(n)$
computations---compared to the proposed estimator, which is obtained by solving a
linear system of size $nd \times nd$.
This raises questions about the
applicability of the proposed estimator in practice, though it is very well known that KDE
performs poorly for moderate to large $d$ \cite[Section 6.5]{Wasserman-06}. In
this section, we numerically demonstrate that the proposed score matching estimator performs significantly better than the KDE, and in particular, that the
advantage with the proposed estimator grows as $d$ gets large.
Note further that the maximum
likelihood approach of \citet{Barron-91} and \citet{Fukumizu-09a} does not yield
estimators that are practically feasible, and therefore to the best of our
knowledge, the proposed estimator is the only viable estimator for
estimating densities through $\Cal{P}$.

In the following, we consider two simple scenarios of estimating a multivariate normal and mixture of normals using the proposed estimator and demonstrate the
superior performance of the proposed estimator over KDE. Inspired by this preliminary empirical investigation,
recently, the proposed estimator has been explored in two concrete applications of gradient-free adaptive MCMC sampler \citep{Heiko-15} and
graphical model structure learning \citep{Sun-15} where the superiority of working with the infinite dimensional family is demonstrated. We would like to again highlight that
the goal of this work is not to construct density estimators
that improve upon KDE but to provide a novel modeling technique of approximating an unknown density by a rich parametric family of densities
with the parameter being infinite dimensional in contrast to the classical approach of finite dimensional approximation.

We consider the problems of
estimating a standard normal
distribution on $\bb{R}^d$, $N(0,I_d)$ and mixture of Gaussians, $$
    p_0(x) = \frac{1}{2} \phi_d(x; \alpha \textbf{1}_n, I_d) + \frac{1}{2}
\phi_d(x; \beta \textbf{1}_n, I_d)$$
through
the score matching approach and KDE, and compare their estimation accuracies.
Here $\phi_d(x;\mu,\Sigma)$ is the p.d.f.~of $N(\mu,\Sigma I_d)$. By
choosing the kernel, $
    k(x,y) =  \exp(-\frac{\| x-y\|^2_2}{2\sigma^2})+ r (x^T y+c)^2,$ which is a Gaussian plus
polynomial of degree 2, it is easy to verify that Gaussian distributions lie in $\Cal{P}$, and
therefore the first problem considers the well-specified case while the second
problem deals with the misspecified case. In our simulations, we chose
$r=0.1$, $c=0.5$, $\alpha=4$ and $\beta=-4$.  The base measure of the exponential family is $N(0,10^2I_d)$.
The bandwidth parameter
$\sigma$ is chosen by cross-validation (CV) of the objective function $\hat{J}_\lambda$ (see
Theorem~\ref{Thm:score}(iv)) within the parameter set $\{0.1, 0.2, 0.4, 0.6, 0.8, 1, 1.2, 1.4, 1.6\}\times \sigma_*$, where $\sigma_*$ is the median of pairwise distances of data, and the regularization parameter $\lambda$ is set as $\lambda=0.1\times n^{-1/3}$ with sample size $n$.
For KDE, the Gaussian kernel is used for the
smoothing kernel, and the bandwidth parameter is chosen by CV from $\{0.02,0.04,0.06,0.08,0.1, 0.2, 0.4, 0.6, 0.8, 1.0\}\times \sigma_*$;
where for both the methods, 5-fold CV is applied.
\begin{figure}[t]
\begin{center}
  \includegraphics[height=5.5cm]{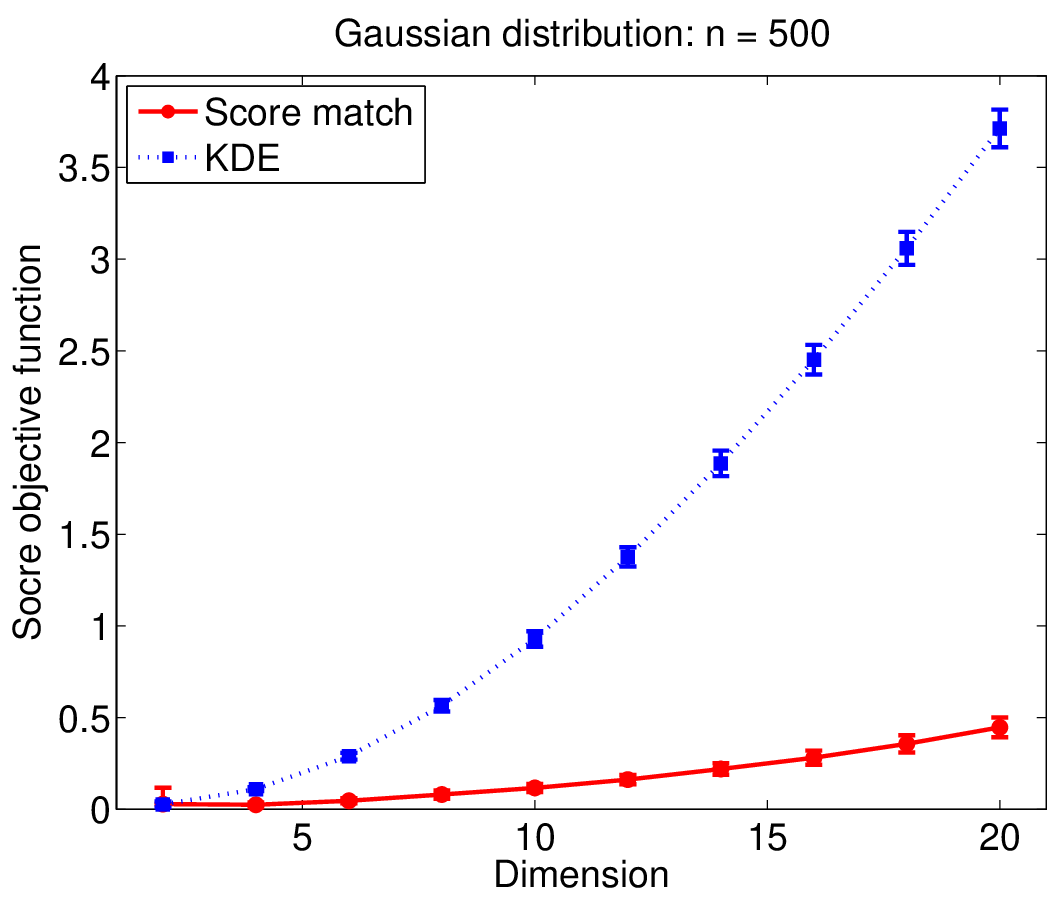}\hspace*{0.5cm}
  \includegraphics[height=5.5cm]{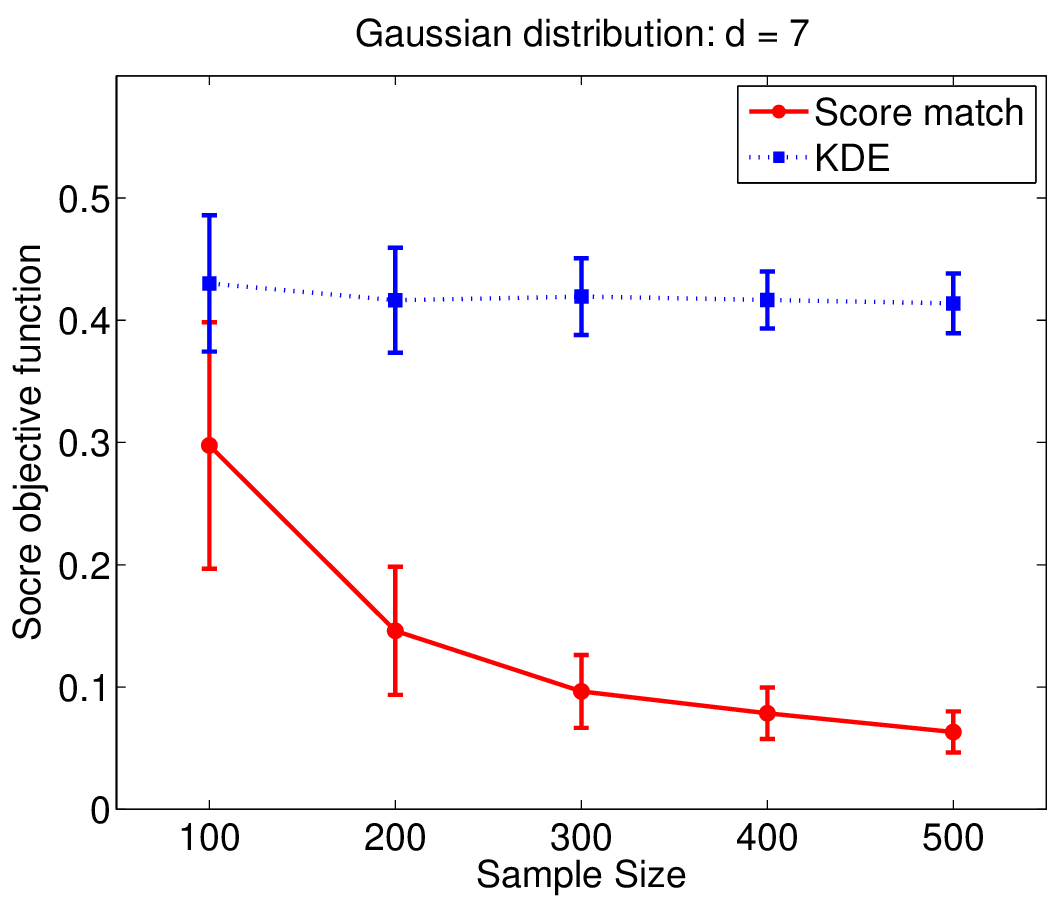}\\
\vspace{.5cm}
  \includegraphics[height=5.5cm]{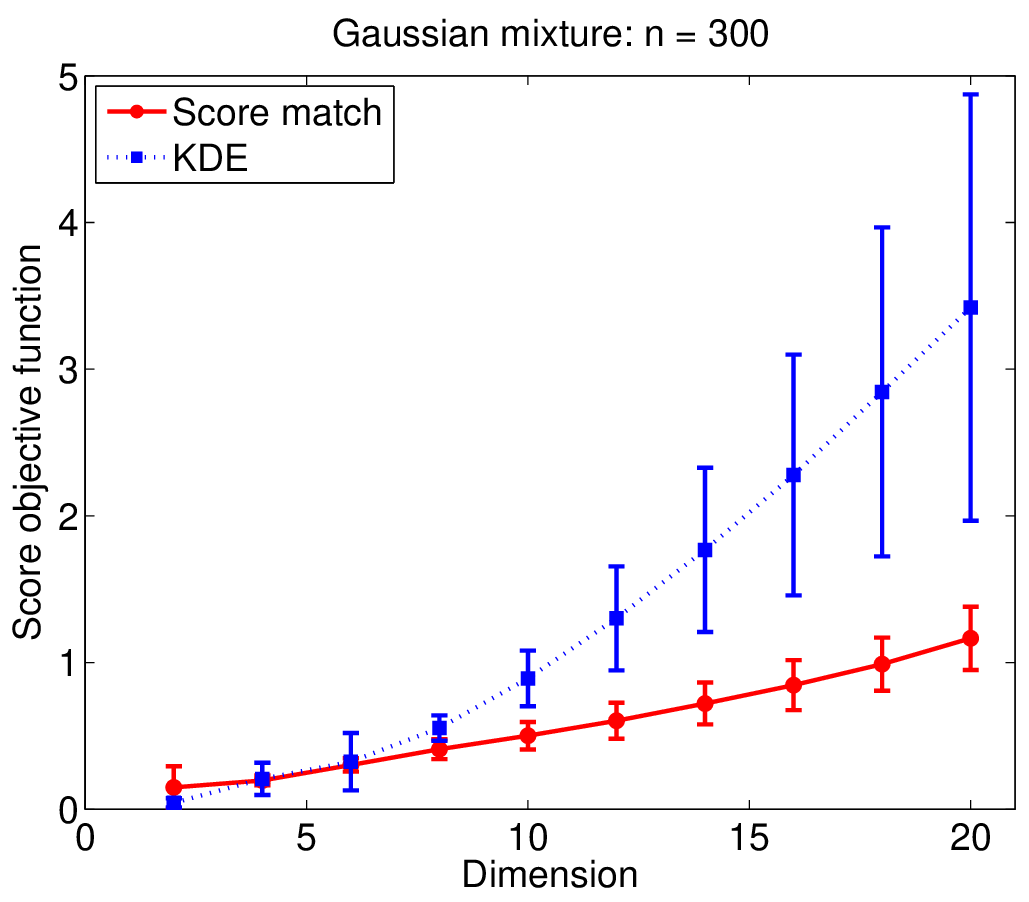}
  \includegraphics[height=5.5cm]{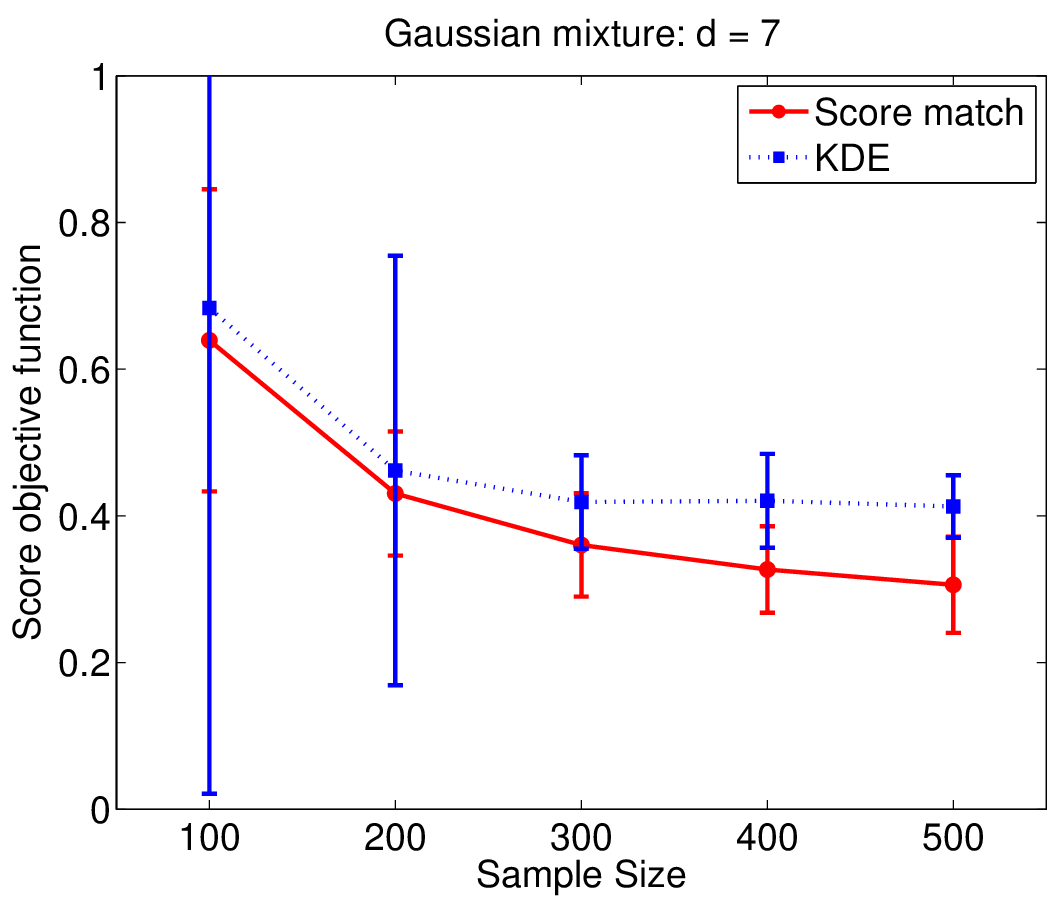}
  \caption{Experimental comparisons with the score objective function: proposed
method and
kernel density estimator}\label{fig:graphs_score}
\vspace{-13mm}
\end{center}
\end{figure}
Since it is difficult to accurately estimate the normalization constant in the
proposed method, we use two methods to evaluate the accuracy of estimation. One
is the objective function for the score matching method,
\[
    \tilde{J}(p) = \sum_{i=1}^d \int_\Omega \left( \frac{1}{2}\left| \partial_i
\log p(x) \right|^2 +
    \partial^2_i \log p(x)\right) p_0(x)dx,
\]
and the other is correlation of the estimator with the true density function,
\[
    \text{Cor}(p,p_0):=\frac{\bb{E}_R[p(X)p_0(X)]}{\sqrt{ \bb{E}_R[ p(X)^2
]\bb{E}_R[ p_0(X)^2 ]}},
\]
where $R$ is a probability distribution. For $R$, we use the empirical
distribution based on 10000 random samples drawn i.i.d.~from $p_0(x)$.

\begin{figure}[t]
\begin{center}
  \includegraphics[height=5.5cm]{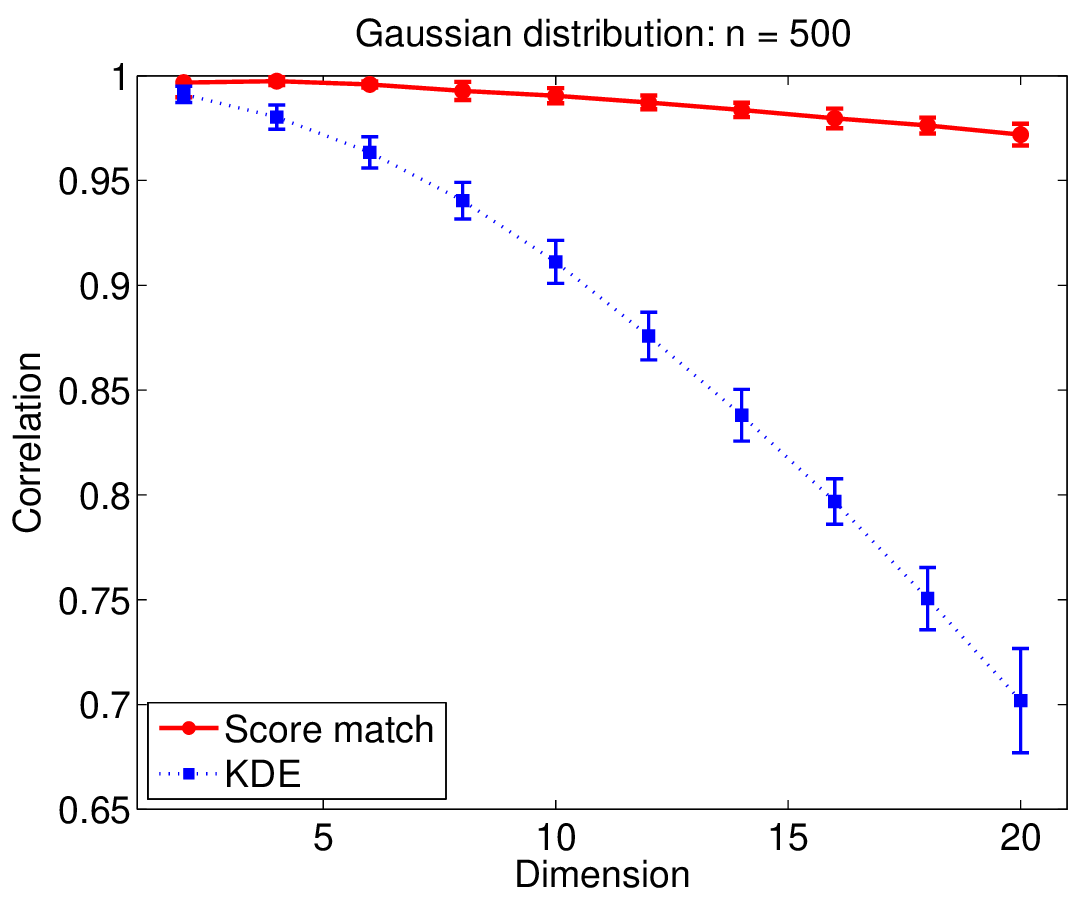}\hspace*{0.5cm}
  \includegraphics[height=5.5cm]{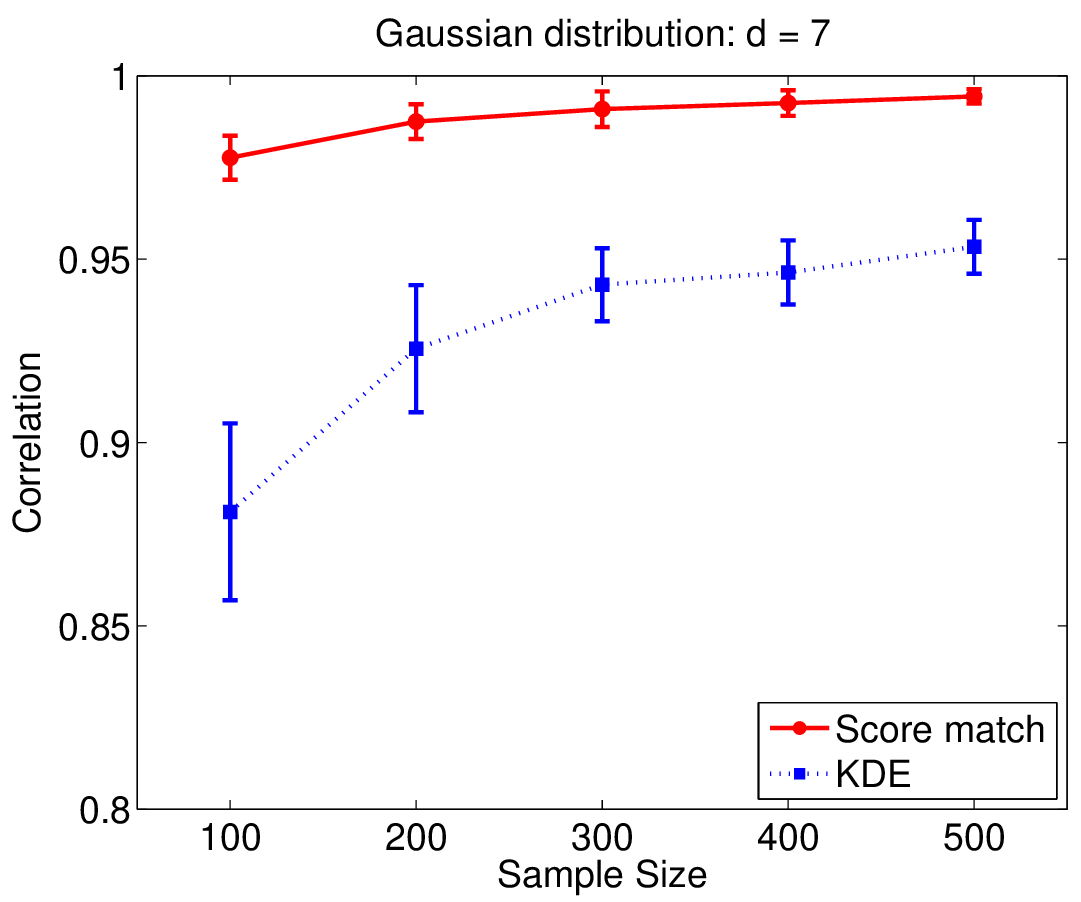}\\
\vspace{.5cm}
  \includegraphics[height=5.5cm]{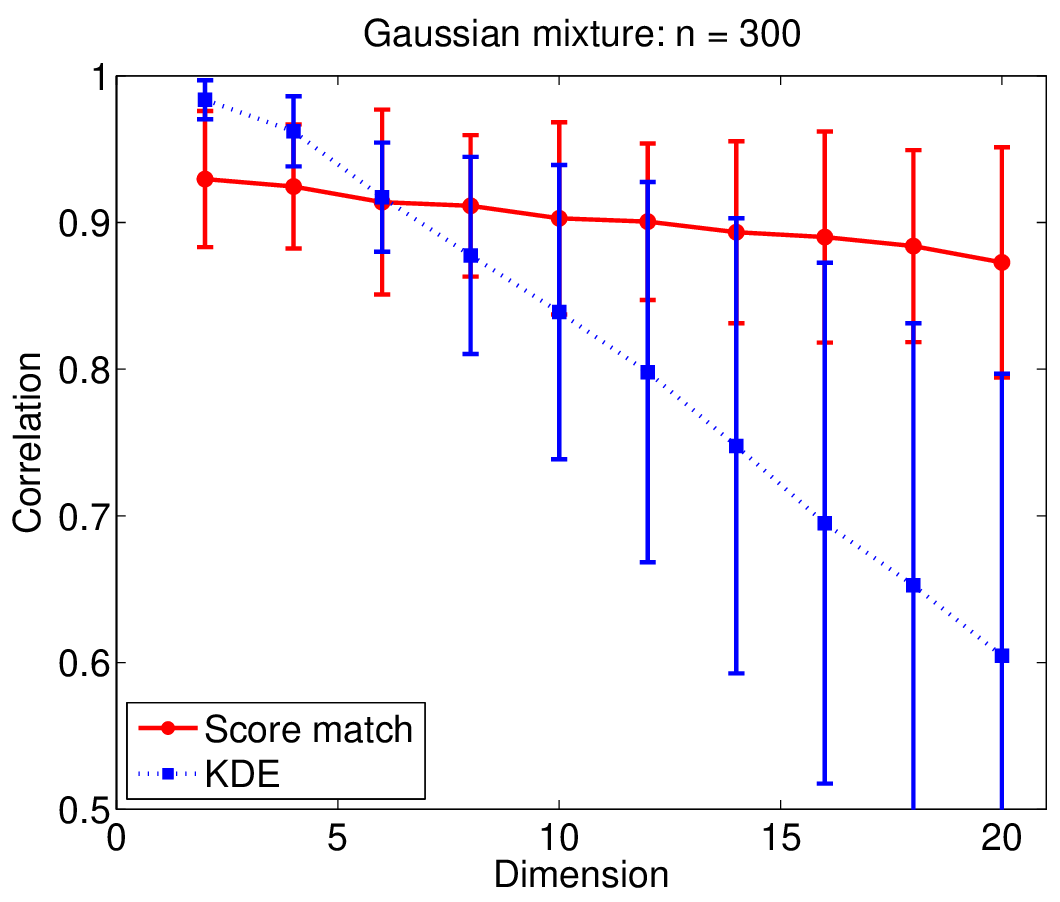}
  \includegraphics[height=5.5cm]{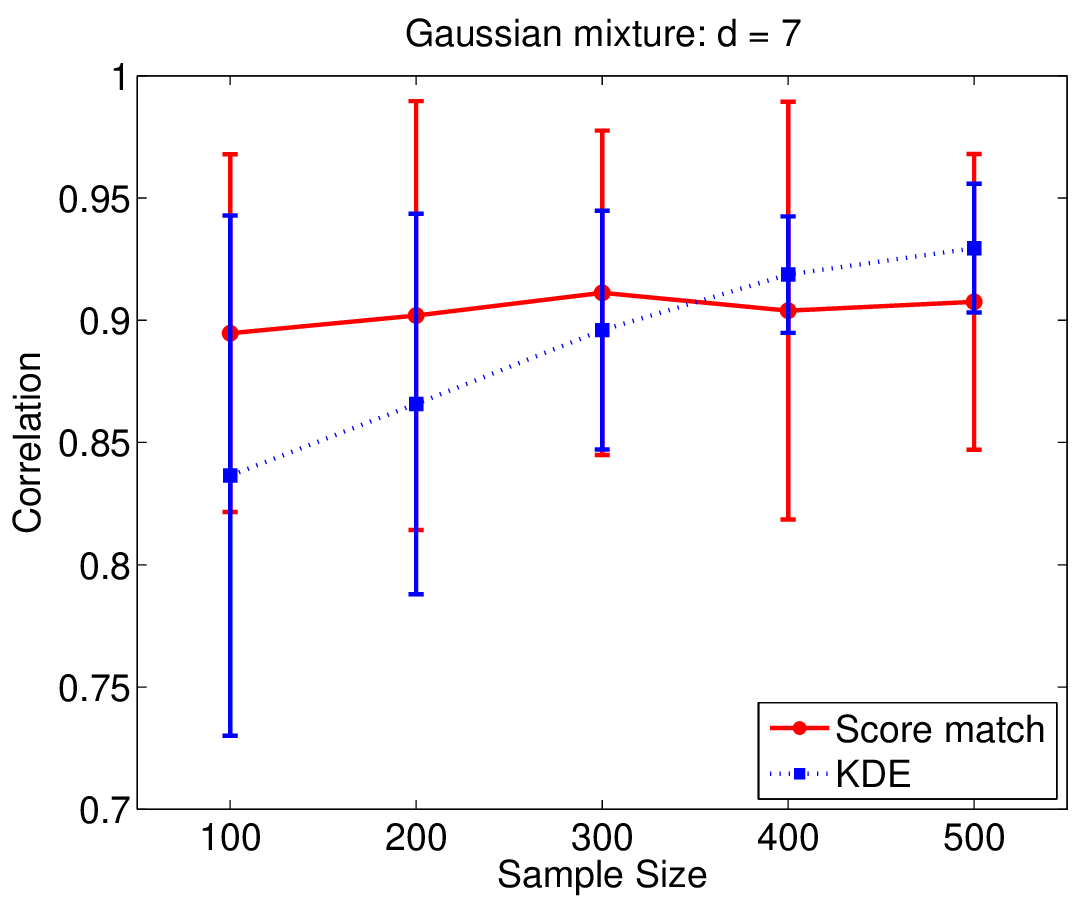}
  \caption{Experimental comparisons with the correlation: proposed
method and
kernel density estimator}\label{fig:graphs_cor}
\vspace{-12mm}
\end{center}
\end{figure}

Figures \ref{fig:graphs_score}
and \ref{fig:graphs_cor} show the score objective
function ($\tilde{J}(p)$) and the correlation ($\text{Cor}(p,p_0)$) (along with
their standard deviation as error bars) of the
proposed estimator and KDE for the tasks of estimating a Gaussian and a mixture
of Gaussians, for different sample sizes ($n$) and dimensions ($d$). From the
figures, we see that the proposed estimator
outperforms (i.e., lower function values) KDE in all the cases except
the low dimensional cases ($(n,d)=(500,2)$ for the Gaussian, and
$(n,d)=(300,2),(300,4)$ for the Gaussian mixture). In the case of the correlation measure, the
score matching method yields better results (i.e., higher correlation) besides
in the Gaussian mixture cases of $d=2,4,6$ (Fig.2, lower-left) and some cases of $d=7$ (lower-right). The proposed method shows
an increased advantage over KDE as the dimensionality
increases, thereby demonstrating the advantage of the proposed
estimator for high dimensional data.
\section{Summary \& Discussion}\label{Sec:Conclusions}	
\par We have considered an infinite dimensional generalization, $\Cal{P}$, of the
finite-dimensional exponential family, where the densities are indexed by
functions in a reproducing kernel Hilbert space (RKHS), $\eu{H}$. We showed that
$\Cal{P}$ is a rich object that
can approximate a large class of probability densities arbitrarily well in
Kullback-Leibler divergence, and addressed the main question of estimating
an unknown density, $p_0$ from finite samples drawn i.i.d.~from it, in
well-specified ($p_0\in\Cal{P}$) and misspecified ($p_0\notin\Cal{P}$)
settings. We proposed a density estimator based on minimizing the
regularized version of the empirical Fisher divergence, which results in solving
a simple finite-dimensional linear system. Our estimator provides a
computationally efficient alternative to maximum likelihood based estimators,
which suffer from the computational intractability of the log-partition
function. The proposed estimator is also shown to empirically outperform the
classical kernel density estimator, with  advantage 
 increasing as the dimension of the space increases. 
In addition to these computational and empirical
results, we have established the consistency and convergence rates under
certain smoothness assumptions (e.g., $\log p_0\in\Cal{R}(C^\beta)$) for both well-specified and misspecified
scenarios.

Three important
questions still remain open in this work which we intend to address
in our future work. 
First, the assumption $\log
p_0\in\Cal{R}(C^\beta)$ is not well understood. Though we presented a
necessary condition for this assumption (with $\beta=1$) to hold for bounded
continuous translation invariant kernels on $\bb{R}^d$, obtaining a sufficient
condition can throw light on the minimax optimality of the proposed
estimator. Another alternative is to directly study the minimax optimality of
the rates for $0<\beta\le 1$ (for $\beta>1$, we showed that the above mentioned
rates can be improved by an appropriate choice of the regularizer) by obtaining
minimax lower bounds under the source condition $\log
p_0\in\Cal{R}(C^\beta)$ and the eigenvalue decay rate of $C$, using the ideas in \cite{Devore-04}. Second, the
proposed estimator depends on the regularization parameter, which in turn
depends on the smoothness scale $\beta$. Since $\beta$ is not known in
practice, it is therefore of interest to construct estimators that are adaptive
to unknown $\beta$. Third, since the proposed estimator is computationally expensive as it involves solving a linear system of size
$nd\times nd$, it is important to study either alternate estimators or efficient implementations of the proposed estimator to
improve the applicability of the method.


\section{Proofs}\label{Sec:proofs}
We provide proofs of the results presented in
Sections~\ref{Sec:approximation}--\ref{Sec:misspecified}.

\subsection{Proof of Proposition~\ref{Thm:approx}}\label{subsec:thm-approx}
\citet[Proposition 5]{Sriperumbudur-11} showed that $\eu{H}$
is dense in
$C_0(\Omega)$ w.r.t.~uniform norm if and only if $k$ satisfies
(\ref{Eq:ispd}).
Therefore, the denseness in $L^1$, KL and
Hellinger distances follow trivially from Lemma~\ref{lem:distances}. For $L^r$
norm ($r>1$), the denseness follows by using the bound $\Vert
p_f-p_g\Vert_{L^r(\Omega)}\le 2e^{2\Vert
f-g\Vert_\infty}e^{2\Vert f\Vert_\infty}\Vert f-g\Vert_\infty\Vert
q_0\Vert_{L^r(\Omega)}$ obtained from Lemma~\ref{lem:distances}(i) with
$f\in C_0(\Omega)$ and $g\in\eu{H}$.\QEDA
\subsection{Proof of Corollary~\ref{cor:approx}}\label{subsec:cor-approx-1}
For any $p\in\Cal{P}_{c}$, define
$p_\delta:=\frac{p+\delta q_0}{1+\delta}$. Note that $p_\delta(x)>0$ for
all $x\in\Omega$ and $\Vert p-p_\delta\Vert_{L^r(\Omega)}=\frac{\delta\Vert
p-q_0\Vert_{L^r(\Omega)}}{1+\delta}$, implying that
$\lim_{\delta\rightarrow 0}\Vert
p-p_\delta\Vert_{L^r(\Omega)}=0$ for
any $1\le r\le\infty$. This means, for any $\epsilon>0$, $\exists
\delta_\epsilon>0$ such that for any $0<\theta<\delta_\epsilon$, we have $\Vert
p-p_\theta\Vert_{L^r(\Omega)}\le\epsilon$, where $p_\theta(x)>0$ for
all
$x\in\Omega$. 

Define
$f:=\log\frac{p_\theta}{q_0}-c_\theta$ where $c_\theta:=\log \frac{\ell+\theta}{1+\theta}$. It is clear that $f\in C(\Omega)$ since
$p,q\in C(\Omega)$. Fix any $\eta>0$ and define $$A:=\left\{x:f(x)\ge \eta\right\}=\left\{x:\frac{p(x)}{q_0(x)}-\ell\ge \left(\ell+\theta\right)(e^{\eta}-1)\right\}.$$
Since $\frac{p}{q_0}-\ell\in C_0(\Omega)$, it is clear that $A$ is compact and so $f\in C_0(\Omega)$. Also, it is easy to verify that
$p_\theta=e^{f-A(f)}q_0$ which
%
implies $p_\theta\in\Cal{P}_0$, where $\Cal{P}_0$ is defined in
Proposition~\ref{Thm:approx}. This means, for any $\epsilon>0$, there exists
$p_g\in\Cal{P}$ such that $\Vert
p_\theta-p_g\Vert_{L^r(\Omega)}\le\epsilon$ under the assumption that $q_0\in
L^1(\Omega)\cap L^r(\Omega)$. Therefore $\Vert
p-p_g\Vert_{L^r(\Omega)}\le 2\epsilon$ for any $1\le r\le\infty$, which
proves the denseness of $\Cal{P}$ in $\Cal{P}_c$ w.r.t.~$L^r$ norm for any
$1\le r\le\infty$. Since $h(p,q)\le \sqrt{\Vert p-q\Vert_{L^1(\Omega)}}$ for any
probability densities $p,q$, the denseness in Hellinger distance follows.

We now prove the denseness in KL divergence by noting that
\begin{eqnarray}KL(p\Vert p_\delta)
=\int_{\{p>0\}}p\log\frac{p+p\delta}{
p+q_0\delta } \,
dx&{}\le{}& \int_{\{p>0\}}p\left(\frac{p+p\delta}{p+q_0\delta}-1\right)\,
dx\nonumber\\
&{}={}&\delta\int_ { p>0 }
(p-q_0)\frac{p}{p+q_0\delta}\,dx
\le\delta\Vert p-q_0\Vert_{L^1(\Omega)}\le
2\delta,\nonumber
\end{eqnarray}
which implies $\lim_{\delta\rightarrow 0}KL(p\Vert p_\delta)=0$. This
implies, for any $\epsilon>0$, $\exists\delta_\epsilon>0$ such that for any
$0<\theta<\delta_\epsilon$, $KL(p\Vert p_\theta)\le\epsilon$. Arguing as above,
we
have $p_\theta\in\Cal{P}_0$, i.e., there exists $f\in C_0(\Omega)$ such that
$p_\theta=\frac{e^fq_0}{\int e^f q_0\,dx}$. Since $\eu{H}$ is dense in
$C_0(\Omega)$, for any $f\in C_0(\Omega)$ and any $\epsilon>0$, there exists
$g\in\eu{H}$ such that $\Vert f-g\Vert_\infty\le\epsilon$. For $p_g\in
\Cal{P}$, since
$\int p\,\log\frac{p_\theta}{p_g}\,dx\le \left\Vert
\log\frac{p_\theta}{p_g}\right\Vert_\infty\le 2\Vert f-g\Vert_\infty\le
2\epsilon,$ we have
$$KL(p\Vert p_g)=\int_\Omega p\log\frac{p}{p_g}\,dx=\int_\Omega
p\log\frac{p}{p_\theta}\,dx+\int_\Omega
p\log\frac{p_\theta}{p_g}\,dx\le 3\epsilon$$ and the result follows.\QEDA
\subsection{Proof of Theorem~\ref{Thm:score}}\label{subsec:thm-score}
(i) By the reproducing property of $\eu{H}$, since $\partial_i
f(x)=\left\langle f,\partial_{i} k(x,\cdot)\right\rangle_\eu{H}$ for all $i\in[d]$, it is easy to verify that
\begin{eqnarray}
 J(f)
&{}={}&\frac{1}{2}\int_\Omega p_0(x)\sum^d_{i=1}\left\langle
f-f_0,\partial_{i} k(x,\cdot)\right\rangle^2_\eu{H}\,dx\nonumber\\
&{}={}&\frac{1}{2}\int_\Omega p_0(x)\sum^d_{i=1} \left\langle
f-f_0,\left(\partial_{i} k(x,\cdot)\otimes\partial_{i} k(x,\cdot)\right)(f-f_0)\right\rangle_\eu{H}\,dx\nonumber\\
&{}={}&\frac{1}{2}\int_\Omega p_0(x)\left\langle
f-f_0,C_x(f-f_0)\right\rangle_\eu{H}\,dx,\label{Eq:temp}
\end{eqnarray}
where in the second line, we used $\langle a,b\rangle^2_H=\langle
a,b\rangle_H\langle a,b\rangle_H=\langle
a,(b\otimes b)a\rangle_H$ for $a,b\in H$ with $H$ being a Hilbert space
and \begin{equation}C_x:=\sum^d_{i=1}\partial_{i} k(x,\cdot)\otimes\partial_{i} k(x,\cdot).\label{Eq:cx}\end{equation} 
Observe that for all $x\in\Omega$, $C_x$ is a Hilbert-Schmidt operator
as $\Vert C_x\Vert_{HS}\le \sum^d_{i=1}\left\Vert\partial_{i}
k(x,\cdot) \right\Vert^2_\eu{H}$ $=\sum^d_{i=1}\partial_i\partial_{i+d}k(x,x)<\infty$ and $(f-f_0)\otimes
(f-f_0)$ is also Hilbert-Schmidt as $\Vert (f-f_0)\otimes
(f-f_0)\Vert_{HS}=\Vert f-f_0\Vert^2_\eu{H}<\infty$. Therefore,
(\ref{Eq:temp}) is equivalent to
$$J(f)=\frac{1}{2}\int_\Omega p_0(x)\left\langle (f-f_0)\otimes
(f-f_0),C_x\right\rangle_{HS}\,dx.$$ 
Since the first condition in \textbf{(D)} implies $\int_\Omega \Vert C_x\Vert_{HS}p_0(x)\,dx<\infty$, $C_x$ is
$p_0$-integrable in the Bochner sense (see \citealp*[Definition 1 and Theorem
2]{Diestel-77}), and therefore it follows from \citet[Theorem 6]{Diestel-77} that $$J(f)=\frac{1}{2}\left\langle
(f-f_0)\otimes (f-f_0),\int_\Omega C_x\;p_0(x)\,dx\right\rangle_{HS},$$ where $C:=\int_\Omega C_x\;p_0(x)\,dx$ is the Bochner
 integral of $C_x$, thereby yielding (\ref{Eq:population}).
\par We now show that $C$ is trace-class. Let $(e_l)_{l\in\bb{N}}$ be an
orthonormal basis in $\eu{H}$ (a countable ONB exists
as $\eu{H}$ is separable---see Remark~\ref{rem:assumptions}(i)). Define
$B:=\sum_l\langle
Ce_l,e_l\rangle_\eu{H}$ so that
\begin{eqnarray}
B&{}={}&\sum_l\int_\Omega \langle
e_l,C_x
e_l\rangle_\eu{H}p_0(x)\,dx=\sum_l\int_\Omega\sum^d_{i=1}\left\langle
e_l,\partial_{i}
k(x,\cdot)\right\rangle^2_{\eu{H}}p_0(x)\,dx\nonumber\\
&\stackrel{(*)}{=}&\int_\Omega\sum_{i\in[d],l}\left\langle e_l,\partial_{i}
k(x,\cdot)\right\rangle^2_{\eu{H}}p_0(x)\,dx
\stackrel{(**)}{=}
\int_\Omega\sum^d_{i=1}
\left\Vert \partial_{i}
k(x,\cdot)\right\Vert^2_{\eu{H}}p_0(x)\,dx<\infty,\nonumber
\end{eqnarray}
which means $C$ is trace-class and therefore compact. Here,
we
used monotone convergence theorem in $(*)$ and Parseval's identity in
$(**)$. Note that $C$ is positive since $\langle f,Cf\rangle_\eu{H}=\int_\Omega
p_0(x)\left\Vert\nabla
f\right\Vert^2_2\,dx\ge 0,\,\forall\,f\in\eu{H}.$
\vspace{2mm}\\
(ii) From (\ref{Eq:population}), we have $J(f)=\frac{1}{2}\langle
f,Cf\rangle_\eu{H}-\langle f,Cf_0\rangle_\eu{H}+\frac{1}{2}\langle
f_0,Cf_0\rangle_\eu{H}$. Using $\partial_i
f_0(x)=\partial_i
\log p_0(x)-\partial_i \log q_0(x)$ for all $i\in[d]$, we obtain that for any $f\in\eu{H}$,
\begin{eqnarray}
\langle f,Cf_0\rangle_\eu{H}&{}={}&\int_\Omega p_0(x)\sum^d_{i=1} \partial_i
f(x)\partial_i
f_0(x)\,dx\nonumber\\
&{}={}&\int_\Omega \sum^d_{i=1}\partial_i
f(x)\partial_i
p_0(x)\,dx-\int_\Omega p_0(x)\sum^d_{i=1}\partial_i
f(x)\partial_i \log q_0(x)\,dx\nonumber\\
&{}\stackrel{(b)}{=}{}&-\int_\Omega p_0(x)\sum^d_{i=1}\partial^2_i
f(x)\,dx-\int_\Omega p_0(x)\sum^d_{i=1}\partial_i
f(x)\partial_i \log q_0(x)\,dx\nonumber\\
&{}={}&-\int_\Omega p_0(x)\left\langle f,\overbrace{\sum^d_{i=1}\partial^2_{i}
k(x,\cdot)+\partial_{i}
k(x,\cdot)\partial_i \log q_0(x)}^{\xi_x}\right\rangle_\eu{H}\,dx\stackrel{(c)}{=}\langle
f,-\xi\rangle_\eu{H},\label{Eq:xix}
\end{eqnarray}
where $(b)$ follows from integration by parts under 
\textbf{(C)} and
the equality in $(c)$ is valid as $\xi_x$ 
is Bochner $p_0$-integrable 
under \textbf{(D)} with $\varepsilon=1$. Therefore $Cf_0=-\xi$. 
For the third term, $\langle f_0,Cf_0\rangle_\eu{H}=\int_\Omega
p_0(x)\sum^d_{i=1}\left(\partial_i
f_0(x)\right)^2\,dx
$
and the result follows.
\vspace{2mm}\\
(iii) Define $c_0:=J(p_0\Vert q_0)$. For any $\lambda>0$, it is easy to verify that
\begin{eqnarray}
J_\lambda(f)&{}={}&\frac{1}{2}\Vert (C+\lambda I)^{1/2}f+(C+\lambda
I)^{-1/2}\xi\Vert^2_\eu{H}-\frac{1}{2}\langle \xi, (C+\lambda
I)^{-1}\xi\rangle_\eu{H}+c_0.\nonumber 
\end{eqnarray}
Clearly, $J_\lambda(f)$ is minimized if and only if $(C+\lambda
I)^{1/2}f=-(C+\lambda
I)^{-1/2}\xi$ and therefore $f_\lambda=-(C+\lambda I)^{-1}\xi$ is the unique
minimizer of $J_\lambda(f)$.
\vspace{2mm}\\
(iv) Since $(iv)$ is similar to $(iii)$ with $C$ replaced by $\hat{C}$ and
$\xi$ replaced by $\hat{\xi}$, we obtain $f_{\lambda,n}=(\hat{C}+\lambda
I)^{-1}\hat{\xi}$.\QEDA
\subsection{Proof of
Theorem~\ref{Thm:representer}}\label{subsec:thm-representer}
We prove the result based on the general representer theorem (Theorem~\ref{thm:representer}). 
From Theorem~\ref{Thm:score}(iv), we have
\begin{eqnarray}
f_{\lambda,n}&{}={}&\arg\inf_{f\in\eu{H}}\frac{1}{2}\langle f,\hat{C}f\rangle_\eu{H}+\langle
f,\hat{\xi}\rangle_\eu{H}+\frac{\lambda}{2}\Vert f\Vert^2_\eu{H}\nonumber\\
&{}={}&\arg\inf_{f\in\eu{H}}\frac{1}{2n}\sum^n_{a=1}\sum^d_{i=1}\left\langle
f,\partial_i k(X_a,\cdot)\right\rangle^2_\eu{H}+\langle
f,\hat{\xi}\rangle_\eu{ H } +\frac{\lambda}{2} \Vert
f\Vert^2_\eu{H}\nonumber\\
&{}={}&\arg\inf_{f\in\eu{H}} V(\langle f,\phi_1\rangle_\eu{H},\ldots,\langle f,\phi_{nd}\rangle_\eu{H},\langle f,\phi_{nd+1}\rangle_\eu{H})+\frac{\lambda}{2} \Vert
f\Vert^2_\eu{H},\nonumber
\end{eqnarray}
where $V(\theta_1,\ldots,\theta_{nd},\theta_{nd+1}):=\frac{1}{2n}\sum^n_{a=1}\sum^d_{i=1}\theta^2_{(a-1)d+i}+\theta_{nd+1}$, $\phi_{(a-1)d+i}:=\partial_ik(X_a,\cdot),\,a\in[n],\,i\in[d]$
and $\phi_{nd+1}:=\hat{\xi}$. Therefore, it follows from Theorem~\ref{thm:representer} that
\begin{equation}
f_{\lambda,n}=\delta\hat{\xi}+\sum^n_{a=1}\sum^d_{i=1}\beta_{(a-1)d+i}\phi_{(a-1)d+i}
\label{Eq:rep}
\end{equation}
where $\delta$ and $\bm{\beta}$ satisfy 
\begin{equation}
\lambda \begin{pmatrix}\bm{\beta}\\\delta\end{pmatrix}+\nabla V\left(\bm{K}\begin{pmatrix}\bm{\beta}\\\delta\end{pmatrix}\right)=0
\label{Eq:condition}
\end{equation}
with $\bm{K}=\begin{pmatrix} \bm{G}\, &\,\bm{h} \\
    \bm{h}^T\, &\,\Vert\hat{\xi}\Vert^2_\eu{H}
    \end{pmatrix}.
$
Since $\nabla V\begin{pmatrix}\bm{z}\\t\end{pmatrix}=\begin{pmatrix}\frac{1}{n}\bm{z}\\1\end{pmatrix}$, (\ref{Eq:condition}) reduces to $\lambda\delta+1=0$ and $\lambda\bm{\beta}+\frac{1}{n}\bm{G\beta}+\frac{\delta}{n}\bm{h}=0$
yielding $\delta=-\frac{1}{\lambda}$ and $(\frac{1}{n}\bm{G}+\lambda I)\bm{\beta}=\frac{1}{n\lambda}\bm{h}$.\QEDA
\begin{rem}
Instead of using the general representer theorem (Theorem~\ref{thm:representer}), it is possible to see that the standard representer theorem \citep{Kimeldorf-71,Scholkopf-01} 
gives a similar, but slightly different linear system, and the solutions are the same if $\bm{K}$ is non-singular.
The general representer theorem yields that $\bm{\beta}$ and $\delta$ are solution to $\bm{F}\begin{pmatrix}\bm{\beta}\\\delta\end{pmatrix}=\begin{pmatrix}\bm{0}\\1\end{pmatrix}$, where
$\bm{F}=\begin{pmatrix}\frac{1}{n}\bm{G}+\lambda I\, &\, \frac{1}{n}\bm{h}\\
   \bm{0}^T\, &\,\lambda\end{pmatrix}$. On the other hand, by using 
the standard representer theorem, it is easy to show that $f_{\lambda,n}$ has the form in (\ref{Eq:rep}) with $\delta$ and $\bm{\beta}$ being solution to 
$\bm{K}\bm{F}\begin{pmatrix}\bm{\beta}\\\delta\end{pmatrix}=\bm{K}\begin{pmatrix}\bm{0}\\1\end{pmatrix}$. Clearly, both the solutions match if $\bm{K}$ is invertible while the latter has 
many solutions if $\bm{K}$ is not invertible.
\end{rem}
\subsection{Proof of Theorem~\ref{Thm:rates}}\label{subsec:ratesproof}
Consider
\begin{eqnarray}
  f_{\lambda,n}-f_\lambda
&{}={}&-(\hat{C}+\lambda
I)^{-1}\Big(\hat{\xi}+(\hat{C}+\lambda I)f_\lambda\Big)
\stackrel{(\ast)}{=}
-(\hat{C}+\lambda
I)^{-1}\left(\hat{\xi}+\hat{C}f_\lambda+C(f_0-f_\lambda)\right)\nonumber\\
&{}={}&(\hat{C}+\lambda
I)^{-1}(C-\hat{C})(f_\lambda-f_0)-(\hat{C}+\lambda
I)^{-1}(\hat{\xi}+\hat{C}f_0)\nonumber\\
&{}={}&(\hat{C}+\lambda
I)^{-1}(C-\hat{C})(f_\lambda-f_0)-(\hat{C}+\lambda
I)^{-1}(\hat{\xi}-\xi)+(\hat{C}+\lambda
I)^{-1}(C-\hat{C})f_0,\nonumber
\end{eqnarray}
where we used $\lambda f_\lambda=C(f_0-f_\lambda)$ in ($\ast$). Define $S_1:=\Vert (\hat{C}+\lambda
I)^{-1}(C-\hat{C})(f_\lambda-f_0)\Vert_\eu{H}$, $S_2:=\Vert (\hat{C}+\lambda
I)^{-1}(\hat{\xi}-\xi)\Vert_\eu{H}$ and $S_3:=\Vert (\hat{C}+\lambda
I)^{-1}(C-\hat{C})f_0\Vert_\eu{H}$ so that
\begin{eqnarray}
\Vert f_{\lambda,n}-f_0\Vert_\eu{H}&{}\le{}&\Vert
f_{\lambda,n}-f_\lambda\Vert_\eu{H}+\Vert
f_\lambda-f_0\Vert_\eu{H}
\le S_1+S_2+S_3+\Cal{A}_0(\lambda),
\label{Eq:chain-abcd}
\end{eqnarray}
where $\Cal{A}_0(\lambda):=\Vert f_\lambda-f_0\Vert_\eu{H}$. We now bound $S_1$, $S_2$ and $S_3$ using Proposition~\ref{pro:certain}. Note that $C=\int_\Omega C_x\,p_0(x)\,dx$ where 
$C_x$ is defined in (\ref{Eq:cx}) is a positive, self-adjoint,
trace-class operator and $\textbf{(D)}$ (with $\varepsilon=2$) implies that 
$$\int_\Omega \Vert C_x\Vert^2_{HS}p_0(x)\,dx\le \int_\Omega \left(\sum^d_{i=1}\left\Vert\partial_i k(x,\cdot)\right\Vert^2_\eu{H}\right)^2p_0(x)\,dx
\le d\sum^d_{i=1}\int_\Omega \left\Vert\partial_i k(x,\cdot)\right\Vert^4_\eu{H}p_0(x)\,dx<\infty.$$ 
Therefore, by Proposition~\ref{pro:certain}(i,iii), \begin{equation}S_1\le \Vert (\hat{C}+\lambda I)^{-1}\Vert \Vert (C-\hat{C})(f_\lambda-f_0)\Vert_\eu{H}=O_{p_0}\left(\frac{\Cal{A}_0(\lambda)}{\lambda\sqrt{n}}\right)\label{Eq:s1}\end{equation}
and \begin{equation}S_2\le \Vert (\hat{C}+\lambda I)^{-1}\Vert\Vert \hat{\xi}-\xi\Vert_\eu{H}=O_{p_0}\left(\frac{1}{\lambda\sqrt{n}}\right),\label{Eq:s2}\end{equation} where by using the technique in the proof of Proposition~\ref{pro:certain}(i), we show below that $\Vert \hat{\xi}-\xi\Vert_\eu{H}=O_{p_0}(n^{-1/2})$. 
Note that $\bb{E}_{p_0}\Vert \hat{\xi}-\xi\Vert^2_\eu{H}=\frac{\int_\Omega \Vert\xi_x\Vert^2_\eu{H}p_0(x)\,dx-\Vert\xi\Vert^2_\eu{H}}{n}\le\frac{\int_\Omega \Vert\xi_x\Vert^2_\eu{H}p_0(x)\,dx}{n},$ where
$\xi_x\in\eu{H}$ is defined in (\ref{Eq:xix}) and $\textbf{(D)}$ (with $\varepsilon=2$) implies that $\int_\Omega \Vert\xi_x\Vert^2_\eu{H}p_0(x)\,dx<\infty$. Therefore $\Vert\hat{\xi}-\xi\Vert_\eu{H}=O_{p_0}(n^{-1/2})$
follows from an application of Chebyshev's inequality. Again using Proposition~\ref{pro:certain}(i,iii), we obtain that \begin{equation}S_3\le \Vert (\hat{C}+\lambda I)^{-1}\Vert \Vert (C-\hat{C})f_0\Vert_\eu{H}=O_{p_0}\left(\frac{1}{\lambda\sqrt{n}}\right).\label{Eq:s3}\end{equation}
Using the bounds in $S_1$, $S_2$ and $S_3$ in (\ref{Eq:chain-abcd}), we obtain
\begin{equation}\Vert f_{\lambda,n}-f_0\Vert_\eu{H}=O_{p_0}\left(\frac{1}{\lambda\sqrt{n}}+\frac{\Cal{A}_0(\lambda)}{\lambda\sqrt{n}}\right)+\Cal{A}_0(\lambda).\label{eq:ffinal}\end{equation}
(i) By Proposition~\ref{pro:approxerror}(i), we have that $\Cal{A}_0(\lambda)\rightarrow 0$ as $\lambda\rightarrow 0$ if $f_0\in\overline{\Cal{R}(C)}$. Therefore, it follows from 
(\ref{eq:ffinal}) that $\Vert f_{\lambda,n}-f_0\Vert_\eu{H}\rightarrow 0$ as $\lambda\rightarrow 0$, $\lambda\sqrt{n}\rightarrow\infty$ and $n\rightarrow\infty$.\vspace{2mm}\\
(ii) If $f_0\in\Cal{R}(C^\beta)$ for $\beta>0$,
it follows from Proposition~\ref{pro:approxerror}(ii) that $$\Cal{A}_0(\lambda)\le \max\{1,\Vert C\Vert^{\beta-1}\}\Vert C^{-\beta}f_0\Vert_\eu{H}\lambda^{\min\{1,\beta\}}$$ and therefore
the result follows by choosing $\lambda=n^{-\max\left\{\frac{1}{4},\frac{1}{2(\beta+1)}\right\}}$.\vspace{2mm}\\
(iii) Note that $$S_1=\Vert (\hat{C}+\lambda I)^{-1}(C-\hat{C})(f_\lambda-f_0)\Vert_\eu{H}\le \Vert C(\hat{C}+\lambda I)^{-1}\Vert \Vert C^{-1}\Vert \Vert (C-\hat{C})(f_\lambda-f_0)\Vert_\eu{H},$$
$$S_2=\Vert (\hat{C}+\lambda I)^{-1}(\hat{\xi}-\xi)\Vert_\eu{H}\le \Vert C(\hat{C}+\lambda I)^{-1}\Vert \Vert C^{-1}\Vert \Vert\hat{\xi}-\xi\Vert_\eu{H},$$
$$S_3=\Vert (\hat{C}+\lambda I)^{-1}(C-\hat{C})f_0\Vert_\eu{H}\le \Vert C(\hat{C}+\lambda I)^{-1}\Vert \Vert C^{-1}\Vert \Vert (C-\hat{C})f_0\Vert_\eu{H}$$ and 
$$\Cal{A}_0(\lambda)=\Vert f_\lambda-f_0\Vert_\eu{H}\le \Vert C^{-1}\Vert \Vert C(f_\lambda-f_0)\Vert_\eu{H}.$$
It follows from Proposition~\ref{pro:certain}(v) that $\Vert C(\hat{C}+\lambda I)^{-1}\Vert\lesssim 1$ for $n\ge \frac{c}{\lambda^2}$ where $c$ is a sufficiently large constant 
that depends on $\sum^d_{i=1}\int_\Omega(\partial_i\partial_{i+d}k(x,x))^2p_0(x)\,dx$ but not on $n$ and $\lambda$. 
Using the bounds on $\Vert (C-\hat{C})(f_\lambda-f_0)\Vert_\eu{H}$, $\Vert\hat{\xi}-\xi\Vert_\eu{H}$ and $\Vert (C-\hat{C})f_0\Vert_\eu{H}$ from part (i) and the bound on $\Vert C(f_\lambda-f_0)\Vert_\eu{H}$
from Proposition~\ref{pro:approxerror}(ii), we therefore obtain
\begin{equation}\Vert f_{\lambda,n}-f_0\Vert_\eu{H}\lesssim O_{p_0}\left(\frac{1}{\sqrt{n}}\right)+\lambda\label{Eq:fff}\end{equation}
as $n\rightarrow\infty$ and the result follows.\QEDA\vspace{-1mm}
\begin{rem}\label{rem:better-rate}
Under slightly strong assumptions on the kernel, the bound on $S_1$ in (\ref{Eq:s1}) can be improved to obtain $S_1=O_{p_0}(n^{-1/2})$ while the one on $S_3$ in (\ref{Eq:s3}) can be refined to obtain $S_3=O_{p_0}\left(\sqrt{\frac{\eu{N}(\lambda)}{\lambda n}}\right)$ where
$\eu{N}(\lambda):=\emph{Tr}((C+\lambda I)^{-1}C)$ is the intrinsic dimension of $\eu{H}$. Using the fact that $\eu{N}(\lambda)\le\frac{1}{\lambda}$, it is easy to verify that the latter is an 
improved bound than the one in (\ref{Eq:s3}). In addition $S_3$ dominates $S_1$. However, if $S_2$ in (\ref{Eq:s2}) is not improved, then $S_2$ dominates $S_3$, thereby resulting in a bound that does not
capture the smoothness of $k$ (or the corresponding $\eu{H}$). Unfortunately, even with a refined analysis (not reported here), we are not able to improve the bound on $S_2$ wherein the difficulty lies with
handling $\xi$.
\end{rem}

\subsection{Proof of Theorem~\ref{Thm:density}}\label{subsec:densityproof}
Before we prove the result, we present a lemma.\vspace{-2mm}
\begin{lemma}\label{lem:support}
Suppose $\sup_{x\in\Omega}k(x,x)<\infty$ and $\emph{supp}(q_0)=\Omega$. Then $\Cal{F}=\eu{H}$ and for
any
$f_0\in\eu{H}$ there exists $\tilde{f_0}\in\overline{\Cal{R}(C)}$ such that
$p_{\tilde{f_0}}=p_0$.\vspace{-2mm}
\end{lemma}
\begin{proof}
Since $\sup_{x\in\Omega}k(x,x)<\infty$, it implies that, for every $f\in\eu{H}$, $\int_\Omega e^{f(x)}q_0(x)\,dx<\infty$ and hence $\Cal{F}=\eu{H}$. Also,
under the assumptions on $k$ and $q_0$, it is easy to verify
that $\text{supp}(p_0)=\Omega$, which implies 
$$\Cal{N}(C)=\left\{f\in\eu{H}\,:\, \int_\Omega \left\Vert
\nabla f\right\Vert^2_2
p_0(x)\,dx=0\right\}$$
is either $\bb{R}$ or $\{0\}$, where $\Cal{N}(C)$ denotes the null space of $C$.
Let $\tilde{f_0}$ be the orthogonal projection of $f_0$ onto
$\overline{\Cal{R}(C)}=\Cal{N}(C)^{\perp}$. Then $\tilde{f_0}-f_0\in\bb{R}$ and
therefore $p_{\tilde{f_0}}=p_{f_0}$.
\end{proof}
\emph{Proof of Theorem~\ref{Thm:density}.} 
From Theorem~\ref{Thm:score}(iii), $f_\lambda=(C+\lambda
I)^{-1}Cf_0=(C+\lambda I)^{-1}C\tilde{f_0}$ where the second equality follows
from the proof of Lemma~\ref{lem:support}. Now, carrying out the decomposition
as in the proof of Theorem~\ref{Thm:rates}(i), we obtain
$f_{\lambda,n}-f_\lambda=(\hat{C}+\lambda
I)^{-1}(C-\hat{C})(f_\lambda-\tilde{f_0})-(\hat{C}+\lambda
I)^{-1}(\hat{\xi}-\xi)+(\hat{C}+\lambda
I)^{-1}(C-\hat{C})\tilde{f_0}$ 
and therefore,
$$\Vert f_{\lambda,n}-\tilde{f_0}\Vert_\eu{H}\le \Vert(\hat{C}+\lambda
I)^{-1}\Vert\left(
\Vert(C-\hat{C})(f_\lambda-\tilde{f}_0)\Vert_\eu{H} +  \Vert
\xi-\hat{\xi}\Vert_\eu{H}+\Vert (C-\hat{C})\tilde{f_0}\Vert_\eu{H}\right)+\tilde{\Cal{A}_0}(\lambda),$$
where $\tilde{\Cal{A}_0}(\lambda)=\Vert f_{\lambda}-\tilde{f_0}\Vert_\eu{H}$. The
bounds on these quantities follow those in the proof of
Theorem~\ref{Thm:rates}(i) verbatim and so the consistency result in
Theorem~\ref{Thm:rates}(i) holds for $\Vert
f_{\lambda,n}-\tilde{f_0}\Vert_\eu{H}$. By
Lemma~\ref{lem:support}, since $p_{f_0}=p_{\tilde{f_0}}$, it is sufficient to
consider the convergence of $p_{f_{\lambda,n}}$ to $p_{\tilde{f_0}}$. Therefore, the convergence (along with rates) in $L^r$ (for any $1\le r\le \infty)$, Hellinger and
KL distances follow from using the bound 
$\Vert
f_{\lambda,n}-\tilde{f_0}\Vert_\infty\le
\sqrt{\Vert k\Vert_\infty}\Vert
f_{\lambda,n}-\tilde{f_0}\Vert_\eu{H}$ (obtained through the reproducing property of $k$) in Lemma~\ref{lem:distances} and
invoking Theorem~\ref{Thm:rates}. 

In the following, we obtain
a bound
on 
$J(p_0\Vert p_{f_{\lambda,n}})=\frac{1}{2}\Vert
\sqrt{C}(f_{\lambda,n}-f_0)\Vert^2_\eu{H}$. While one can trivially use the
bound $\Vert
\sqrt{C}(f_{\lambda,n}-f_0)\Vert^2_\eu{H}\le \Vert\sqrt{C}\Vert^2\Vert
f_{\lambda,n}-f_0\Vert^2_\eu{H}$ to obtain a rate on
$J(p_0\Vert p_{f_{\lambda,n}})$ through the result in 
Theorem~\ref{Thm:rates}(ii), a better rate can be obtained by carefully
bounding $\Vert
\sqrt{C}(f_{\lambda,n}-f_0)\Vert^2_\eu{H}$ as shown below. Consider
\begin{eqnarray}\Vert
\sqrt{C}(f_{\lambda,n}-f_0)\Vert_\eu{H}&{}\le{}& \Vert
\sqrt{C}(f_{\lambda,n}-f_\lambda)\Vert_\eu{H}+\tilde{\Cal{A}_{\frac{1}{2}}}(\lambda)+\Cal{A}^\star_{\frac{1}{2}}(\lambda)\nonumber
\end{eqnarray}
\begin{eqnarray}
&{}\le{}&\Vert \sqrt{C}(\hat{C}+\lambda
I)^{-1}\Vert\left(\Vert(C-\hat{C})(f_\lambda-\tilde{f_0})\Vert_\eu{H} +  \Vert
\xi-\hat{\xi}\Vert_\eu{H}+\Vert (C-\hat{C})\tilde{f_0}\Vert_\eu{H}\right)\nonumber\\
&{}{}&\qquad\qquad\qquad+\tilde{\Cal{A}_{\frac{1}{2}}}(\lambda)+\Cal{A}^\star_{\frac{1}{2}}(\lambda),\nonumber
\end{eqnarray}
where $\tilde{\Cal{A}_{\frac{1}{2}}}(\lambda):=\Vert
\sqrt{C}(f_\lambda-\tilde{f_0})\Vert_\eu{H}$ and $\Cal{A}^\star_{\frac{1}{2}}(\lambda):=\Vert
\sqrt{C}(\tilde{f_0}-f_0)\Vert_\eu{H}$. It follows from Theorem~\ref{Thm:score}(i) and Lemma~\ref{lem:support} that $\Cal{A}^\star_{\frac{1}{2}}(\lambda)=J(p_0\Vert p_{\tilde{f_0}})=0$.
Also it follows from Proposition~\ref{pro:certain}(v) that $\Vert \sqrt{C}(\hat{C}+\lambda
I)^{-1}\Vert\lesssim \frac{1}{\sqrt{\lambda}}$ for $n\ge \frac{c}{\lambda^2}$ where $c$ is a large enough constant that does not depend on $n$ and $\lambda$
and depends only on $\sum^d_{i=1}\int \Vert\partial_i k(x,\cdot)\Vert^4_{\eu{H}}\,p_0(x)\,dx$. Using the bounds from the proof of Theorem~\ref{Thm:rates}(i) 
for the rest of the terms within paranthesis, we obtain
\begin{equation}\Vert
\sqrt{C}(f_{\lambda,n}-f_0)\Vert_\eu{H}\le O_{p_0}\left(\frac{1}{\sqrt{\lambda n}}\right)+\tilde{\Cal{A}_{\frac{1}{2}}}(\lambda).\label{Eq:ff}\end{equation}
The consistency result therefore follows from Proposition~\ref{pro:approxerror}(i) by noting that $\tilde{\Cal{A}_{\frac{1}{2}}}(\lambda)\rightarrow 0$ as $\lambda\rightarrow 0$. If
$f_0\in\Cal{R}(C^\beta)$ for some $\beta\ge 0$, then Proposition~\ref{pro:approxerror}(ii) yields $\tilde{\Cal{A}_{\frac{1}{2}}}(\lambda)\le \max\{1,\Vert C\Vert^{\beta-\frac{1}{2}}\}\lambda^{\min\{1,\beta+\frac{1}{2}\}}\Vert C^{-\beta}f_0\Vert_\eu{H}$
which when used in (\ref{Eq:ff}) provides the desired rate with
$\lambda=n^{-\max\{\frac{1}{3},\frac{1}{2(\beta+1)}\}}$.
If $\Vert C^{-1}\Vert<\infty$, then the result follows by
noting $\Vert \sqrt{C}(f_{\lambda,n}-f_0)\Vert_\eu{H}\le \Vert \sqrt{C}\Vert
\Vert f_{\lambda,n}-f_0\Vert_\eu{H}$ and invoking the bound in
(\ref{Eq:fff}).\QEDA
\subsection{Proof of Proposition~\ref{pro:range}}\label{subsec:supp-rangeproof}
\emph{Observation 1:} By \cite[Theorem 10.12]{Wendland-05}, we have
$$\eu{H}=\left\{f\in L^2(\bb{R}^d)\cap
C_b(\bb{R}^d):\frac{f^\wedge}{\sqrt{\psi^\wedge}}\in
L^2(\bb{R}^d)\right\},$$ where $f^\wedge$ is defined in $L^2$ sense. Since
$$\int_{\bb{R}^d} \vert f^\wedge(\omega)\vert\,d\omega\le \left(\int_{\bb{R}^d}
\frac{|f^\wedge(\omega)|^2}{\psi^\wedge(\omega)}\,d\omega\right)^{\frac{1}
{
2 } }
\left(\int_{\bb{R}^d} \psi^\wedge(\omega)\,d\omega\right)^{\frac{1}{2}}<\infty$$ where
we
used $\psi^\wedge\in L^1(\bb{R}^d)$ (see \citealp*[Corollary 6.12]{Wendland-05}),
we
have $f^\wedge\in L^1(\bb{R}^d)$. Hence Plancherel's theorem and continuity
of $f$ along with the inverse Fourier transform of $f^\wedge$ allow to
recover any $f\in\eu{H}$ pointwise from
its Fourier transform as
\begin{equation}f(x)=\frac{1}{(2\pi)^{d/2}}\int_{\bb{R}^d}
e^{ix^T\omega}f^\wedge(\omega)\,d\omega,\,x\in\bb{R}^d.\label{Eq:ft-1}
\end{equation} 
\emph{Observation 2:} Since $\psi^\wedge\in L^1(\bb{R}^d)$ and
$\psi^\wedge\ge 0$,
we have for all $j=1,\ldots,d$, \begin{eqnarray}\left(\int_{\bb{R}^d}
|\omega_j|\psi^\wedge(\omega)\,d\omega\right)^2&=&\left(\int_{\bb{R}^d}
\psi^\wedge(\omega)\,
d\omega\right)^2\left(\int_{\bb{R}^d} |\omega_j|\frac{\psi^\wedge(\omega)}{\int_{\bb{R}^d}
\psi^\wedge(\omega)\,d\omega}\,d\omega\right)^2\nonumber\\
&\stackrel{(\ast)}{\le}&
\left(\int_{\bb{R}^d}
\psi^\wedge(\omega)\,
d\omega\right)\left(\int_{\bb{R}^d}
|\omega_j|^2\psi^\wedge(\omega)\,d\omega\right)\nonumber\\
&\le&\left(\int_{\bb{R}^d}
\psi^\wedge(\omega)\,
d\omega\right)\left(\int_{\bb{R}^d} \Vert
\omega\Vert^2\psi^\wedge(\omega)\,d\omega\right)\stackrel{(i)}{<}\infty,
\nonumber
\end{eqnarray}
where we used Jensen's inequality in ($\ast$). This means
$\omega_j\psi^\wedge(\omega)\in L^1(\bb{R}^d),\,\forall\,j\in[d]$ which
ensures the existence of its Fourier transform and so
\begin{equation}\partial_j\psi(x)=\frac{1}{(2\pi)^{d/2}}\int_{\bb{R}^d}
(i\omega_j)\psi^\wedge(\omega)e^{ix^T\omega}\,d\omega,\,x\in\bb{R}^d,\,\,
\forall\,j\in[d].\label{Eq:ft-2}\end{equation}
\emph{Observation 3:} For $g\in\eu{H}$, we have for all $j\in[d]$,
\begin{eqnarray}\int_{\bb{R}^d}
|\omega_j||g^\wedge(\omega)|\,d\omega&{}\le{}&\left(\int_{\bb{R}^d}\frac{|g^\wedge
(\omega)|^2
} {
\psi^\wedge(\omega)}\,d\omega\right)^{\frac{1}{2}}
\left(\int_{\bb{R}^d}|\omega_j|^2\psi^\wedge(\omega)\,d\omega\right)^{\frac{1}{2}}
\nonumber\\
&{}\le{} &
\left(\int_{\bb{R}^d}\frac{|g^\wedge(\omega)|^2}{\psi^\wedge(\omega)}\,
d\omega\right)^{ \frac{
1}{2}}\left(\int_{\bb{R}^d}\Vert\omega\Vert^2\psi^\wedge(\omega)\,d\omega\right)^{\frac{
1}{2}}\stackrel{(i)}{<}\infty,\nonumber
\end{eqnarray}
which implies $\omega_jg^\wedge(\omega)\in
L^1(\bb{R}^d),\,\forall\,j=1,\ldots,d$. Therefore,
\begin{equation}\partial_j g(x)=\frac{1}{(2\pi)^{d/2}}\int_{\bb{R}^d}
(i\omega_j)g^\wedge(\omega)e^{ix^T\omega}\,d\omega,\,x\in\bb{R}^d,\,\,
\forall\,j\in[d].\label{Eq:ft-3}\end{equation}
\emph{Observation 4:} For any $g\in \eu{G}$, we have
$$\int_{\bb{R}^d} \frac{|g^\wedge(\omega)|^2}{\psi^\wedge(\omega)}\,d\omega=\int_{\bb{R}^d}
\frac{|g^\wedge(\omega)|^2}{\phi^\wedge(\omega)}\frac{\phi^\wedge
(\omega)}{\psi^\wedge(\omega)}\,d\omega\le \Vert
g\Vert^2_{\eu{G}}\left\Vert\frac{\phi^\wedge}{\psi^\wedge}
\right\Vert_\infty\stackrel{(ii)}{<}\infty,$$
which implies $g\in\eu{H}$, i.e., $\eu{G}\subset\eu{H}$. 
\par We now use these observations to prove the result. Since $f_0\in
\Cal{R}(C)$, there exists $g\in\eu{H}$ such that $f_0=Cg$, which means
\begin{eqnarray}
 f_0(y)&{}={}&\int_{\bb{R}^d} \sum^d_{j=1}\partial_j k(x,y)\,\partial_j
\,p_0(x)\,dx\nonumber\\
&{}\stackrel{(\ref{Eq:ft-2})}{=}{}&\int_{\bb{R}^d}\sum^d_{j=1}\frac{1}{(2\pi)^{d/2}}\int_{\bb{R}^d}
e^{i(x-y)^T\omega} (i\omega_j)\psi^\wedge(\omega)\,d\omega\,\partial_j
g(x)\,p_0(x)\,dx\nonumber\\
&{}\stackrel{(\dagger)}{=}{}&\int_{\bb{R}^d}\sum^d_{j=1}\left(\frac{1}{(2\pi)^{d/2}}\int_{\bb{R}^d}
e^{ix^T\omega}
\partial_j g(x)\,p_0(x)\,dx\right)(i\omega_j)\psi^\wedge(\omega)e^{-iy^T\omega}\,
d\omega\nonumber\\ 
&{}\stackrel{(\ref{Eq:ft-3})}{=}{}&\frac{1}{(2\pi)^{d/2}}\int_{\bb{R}^d}\sum^d_{j=1}
\overline{
\left(i(\cdot)_jg^\wedge
\ast p^\wedge_0\right)(\omega)}(i\omega_j)\psi^\wedge
(\omega)e^{-iy^T\omega}\,d\omega\nonumber
\end{eqnarray}
which from (\ref{Eq:ft-1}) means $f^\wedge_0(\omega)=\sum^d_{j=1}\overline{\left(i(\cdot)_jg^\wedge
\ast p^\wedge_0\right)(-\omega)}(-i\omega_j)\psi^\wedge
(\omega)$
where we have invoked Fubini's theorem in $(\dagger)$ and $\ast$ represents the
convolution. Define
$\Vert\cdot\Vert_{L^r(\bb{R}^d)}:=\Vert\cdot\Vert_r$ and $\theta:=\frac{r}{r-1}$. Consider
\begin{eqnarray}
\Vert f_0\Vert^2_\eu{G}&{}=&{}\int_{\bb{R}^d} \frac{\vert
f^\wedge_0(\omega)\vert^2}{\phi^\wedge(\omega)}\,d\omega=\int_{\bb{R}^d}
\left|\sum^d_{j=1}\overline{\left(i(\cdot)_jg^\wedge
\ast
p^\wedge_0\right)(-\omega)}
(i\omega_j)\right|^2(\psi^\wedge)^2(\omega)(\phi^\wedge(\omega))^{-1}\,
d\omega\nonumber\\
&\le &\int_{\bb{R}^d}
\left(\sum^d_{j=1}\left|i(\cdot)_jg^\wedge
\ast
p^\wedge_0\right|(-\omega)|\omega_j|\right)^2(\psi^\wedge)^2(\omega)(\phi^\wedge
(\omega))^{-1}\,d\omega\nonumber\\
&\le &\int_{\bb{R}^d}
\sum^d_{j=1}\left|i(\cdot)_jg^\wedge
\ast p^\wedge_0\right|^2(-\omega)\Vert
\omega\Vert^2 (\psi^\wedge)
^2(\omega)(\phi^\wedge(\omega))^{-1}\,d\omega\nonumber\\
&\le &\left\Vert
\sum^d_{j=1}\left|i\omega_jg^\wedge(\omega)
\ast p^\wedge_0(\omega)\right|^2(\cdot)\right\Vert_{\frac{\theta}{2}}
\left\Vert\Vert
\cdot\Vert^2(\psi^\wedge)
^2(\cdot)(\phi^\wedge(\cdot))^{-1}\right\Vert_{\frac{r}{2-r}}
\stackrel{(iii)}{<}\infty,\nonumber
\end{eqnarray}
where in the following we show that $\sum^d_{j=1}\left|i\omega_jg^\wedge
(\omega)\ast p^\wedge_0(\omega)\right|^2(\cdot)\in
L^{\frac{\theta}{2}}(\bb{R}^d)$, i.e.,
\begin{eqnarray}
\lefteqn{\left\Vert
\sum^d_{j=1}\left|i\omega_jg^\wedge
(\omega)\ast p^\wedge_0(\omega)\right|^2(\cdot)\right\Vert_{\frac{\theta}{2
}}
\le
\sum^d_{j=1}\left\Vert\left|i\omega_jg^\wedge
(\omega)\ast p^\wedge_0(\omega)\right|^2(\cdot)\right\Vert_{\frac{\theta}{2
}}= \sum^d_{j=1}\left\Vert i\omega_jg^\wedge
(\omega)\ast p^\wedge_0(\omega)\right\Vert^2_{\theta}}\nonumber\\
&\qquad\qquad\qquad\qquad\qquad\qquad\stackrel{(\ast)}{\le} \sum^d_{j=1}\left\Vert i\omega_jg^\wedge
(\omega)\right\Vert^2_{1}\left\Vert p^\wedge_0\right\Vert^2_{\theta}
\stackrel{(\ast\ast)}{\le} \left\Vert
p_0\right\Vert^2_{r}\sum^d_{j=1}\left\Vert
i\omega_jg^\wedge
(\omega)\right\Vert^2_{1}\stackrel{(\ddagger)}{<}\infty,\nonumber
\end{eqnarray}
where we have invoked generalized
Young's inequality \cite[Proposition 8.9]{Folland-99} in
($\ast$), Hausdorff-Young inequality \cite[p.~253]{Folland-99} in ($\ast\ast$),
and
observation 3 combined with $(iv)$ in ($\ddagger$). This shows that $f_0\in
\Cal{R}(C)\Rightarrow f_0\in \eu{G}$, i.e., $\Cal{R}(C)\subset\eu{G}$.\QEDA
\subsection{Proof of Theorem~\ref{Thm:rates-new}}\label{subsec:supp-rates-new}
To prove Theorem~\ref{Thm:rates-new}, we need the following lemma
\cite[Lemma 5]{devito-12}, which is due to Andreas Maurer.
\begin{lemma}\label{lem:lipschitz}
Suppose $A$ and $B$ are self-adjoint Hilbert-Schmidt operators on a
separable Hilbert space $H$ with spectrum contained in the interval $[a,b]$,
and let $(\sigma_i)_{i\in I}$ and $(\tau_j)_{j\in J}$ be the eigenvalues of $A$
and $B$, respectively. Given a function $r:[a,b]\rightarrow\bb{R}$, if
there exists a finite constant $L$ such that
$|r(\sigma_i)-r(\tau_j)|\le L|\sigma_i-\tau_j|,\,\,\forall\,i\in I,\,j\in J,$
then 
$\Vert r(A)-r(B)\Vert_{HS}\le L\Vert
A-B\Vert_{HS}.$
\end{lemma}
\noindent\emph{Proof of Theorem~\ref{Thm:rates-new}.}
(i) The proof follows the ideas in the proof of Theorem 10 in
\cite{Bauer-07}, which is a more general result dealing with the smoothness
condition, $f_0\in\Cal{R}(\Theta(C))$ where $\Theta$ is operator monotone.
Recall that $\Theta$ is operator monotone on $[0,b]$ if for any pair of
self-adjoint operators $U$, $V$ with spectra in $[0,b]$ such that $U\le V$,
we have $\Theta(U)\le\Theta(V)$, where ``$\le$'' is the partial ordering for
self-adjoint operators on some Hilbert space $H$, which means for any $f\in
H$, $\langle f,Uf\rangle_H\le\langle f,Vf\rangle_H$. In our case, we adapt the
proof for $\Theta(C)=C^\beta$. Define
$r_\lambda(\alpha):=g_\lambda(\alpha)\alpha-1$. Since
$f_0\in\Cal{R}(C^\beta)$, there exists $h\in\eu{H}$ such that $f_0=C^\beta h$,
which yields
\begin{eqnarray}
f_{g,\lambda,n}-f_0&{}={}&-g_\lambda(\hat{C})\hat{\xi}-f_0=
-g_\lambda(\hat{C})(\hat{\xi}+\hat{C}f_0)+r_\lambda(\hat{C})C^\beta h\nonumber\\
&{}={}&-g_\lambda(\hat{C})(\hat{\xi}-\xi)+g_\lambda(\hat{C})(C-\hat{C})f_0+r_\lambda(\hat { C } )\hat { C } ^\beta
h+r_\lambda(\hat{C})(C^\beta-\hat{C}^\beta)h.\label{Eq:tempoo}
\end{eqnarray}
so that 
\begin{eqnarray}
\Vert f_{g,\lambda,n}-f_0\Vert_\eu{H}&{}\le{}& \underbrace{\Vert g_\lambda(\hat{C})(\hat{\xi}-\xi)\Vert_\eu{H}}_{(A)}+\underbrace{\Vert g_\lambda(\hat{C})(\hat{C}-C)f_0\Vert_\eu{H}}_{(B)}+\underbrace{\Vert r_\lambda(\hat{C})C^\beta h\Vert_\eu{H}}_{(C)}\nonumber\\
&{}{}&\qquad\qquad+\underbrace{\Vert r_\lambda(\hat{C})(C^\beta-\hat{C}^\beta)h \Vert_\eu{H}}_{(D)}.\nonumber
\end{eqnarray}
We now bound $(A)$--$(D)$. Since $(A)\le \Vert g_\lambda(\hat{C})\Vert \Vert \hat{\xi}-\xi\Vert_\eu{H}$, we have $(A)=O_{p_0}\left(\frac{1}{\lambda \sqrt{n}}\right)$ where we 
used $(b)$ in \textbf{(E)} and the bound on $\Vert
\hat{\xi}-\xi\Vert_\eu{H}$ from the proof of Theorem~\ref{Thm:rates}(i). Similarly, $(B)\le \Vert g_\lambda(\hat{C})\Vert \Vert (\hat{C}-C)f_0\Vert_\eu{H}$ implies 
$(B)=O_{p_0}\left(\frac{1}{\lambda \sqrt{n}}\right)$ where $(b)$ in \textbf{(E)} and Proposition~\ref{pro:certain}(i) are invoked. Also, $(d)$ in \textbf{(E)} implies that
$$(C)\le\Vert r_\lambda(\hat{C})\hat{C}^\beta\Vert \Vert
h\Vert_\eu{H}\le \max\{\gamma_\beta,
\gamma_{\eta_0}\}\lambda^{\min\{\beta,\eta_0\}}\Vert
C^{-\beta}f_0\Vert_\eu{H}.$$ $(D)$ can be bounded as
$$(D)\le \Vert r_\lambda(\hat{C})\Vert \Vert C^\beta-\hat{C}^\beta\Vert \Vert C^{-\beta}f_0\Vert_\eu{H}.$$
We now consider two cases:\vspace{1mm}\\
\underline{$\beta\le 1$}: 
Since $\alpha\mapsto\alpha^\theta$ is
operator monotone on $[0,\chi]$ for $0\le\theta\le 1$, by Theorem 1 in
\cite{Bauer-07}, there exists a constant $c_\theta$ such that $\Vert
\hat{C}^\theta-C^\theta\Vert\le c_\theta \Vert \hat{C}-C\Vert^\theta\le
c_\theta\Vert
\hat{C}-C\Vert^\theta_{HS}$. We now obtain a bound on $\Vert \hat{C}-C\Vert_{HS}$. To this end, consider 
\begin{eqnarray}
\bb{E}\Vert \hat{C}-C\Vert^2_{HS}&{}={}&\bb{E}\Vert \hat{C}\Vert^2_{HS}-\Vert C\Vert^2_{HS}\nonumber\\
&{}\le{}& \frac{1}{n}\int \left\Vert \sum^d_{i=1}\partial_i k(x,\cdot)\otimes \partial_i k(x,\cdot)\right\Vert^2_{HS}p_0(x)\,dx
\le\frac{d}{n}\sum^d_{i=1}\int\left\Vert\partial_i k(x,\cdot)\right\Vert^4p_0(x)\,dx,\nonumber
\end{eqnarray}
which by Chebyshev's inequality implies that \begin{equation}\Vert \hat{C}-C\Vert_{HS}=O_{p_0}(n^{-1/2})\nonumber
\end{equation}
and therefore $(D)=O_{p_0}(n^{-\beta/2})$. Since $\lambda\ge n^{-1/2}$, we have $(D)=O_{p_0}(\lambda^\beta)$.\vspace{1mm}\\
\underline{$\beta>1$}: Since $\alpha\mapsto\alpha^\theta$ is
Lipschitz on $[0,\chi]$ for $\theta\ge 1$, by Lemma~\ref{lem:lipschitz},
$\Vert C^{\beta}-\hat{C}^{\beta}\Vert\le \Vert
C^{\beta}-\hat{C}^{\beta}\Vert_{HS}\le
\beta\chi^{\beta-1} \Vert C-\hat{C}\Vert_{HS}$ and therefore $(C)=O_{p_0}(n^{-1/2})$.

Collecting
all the above bounds, we obtain 
$$\Vert f_{g,\lambda,n}-f_0\Vert_\eu{H}\le O_{p_0}\left(\frac{1}{\lambda\sqrt{n}}\right)+O_{p_0}\left(\lambda^{\min\{\beta,\eta_0\}}\right)$$
and the result follows. The proofs of the claims involving $L^r$, $h$ and $KL$ follow exactly the same ideas
as in
the proof of Theorem~\ref{Thm:density} by using the above bound on $\Vert f_{g,\lambda,n}-f_0\Vert_\eu{H}$ in Lemma~\ref{lem:distances}.\vspace{2mm}\\
(ii) We now bound
$J(p_0\Vert p_{f_{g,\lambda,n}})=\Vert
\sqrt{C}(f_{g,\lambda,n}-f_0)\Vert^2_\eu{H}$
as follows. Note that
$$\sqrt{C}(f_{g,\lambda,n}-f_0)=\underbrace{(\sqrt{C}-\sqrt{\hat{C}})(f_{g,
\lambda , n }
-f_0)}_{(I^\prime)}+\underbrace{\sqrt{\hat{C}}(f_{g,\lambda,n}-f_0)}_{(II^\prime
)}.$$
We bound $\Vert (I^\prime)\Vert_\eu{H}$ as 
\begin{eqnarray}\Vert (I^\prime)\Vert_\eu{H}\le \Vert \sqrt{C}-\sqrt{\hat{C}}\Vert
\Vert f_{g,\lambda,n}-f_0\Vert_\eu{H}&{}\le{}& c_{\frac{1}{2}}\sqrt{\Vert
C-\hat{C}\Vert}_{HS} \Vert
f_{g,\lambda,n}-f_0\Vert_\eu{H}\nonumber\\
&{}={}&  O_{p_0}\left(\frac{1}{\sqrt{\lambda n}}\right)+O_{p_0}\left(\lambda^{\min\{\beta,\eta_0\}+\frac{1}{2}}\right),\nonumber
\end{eqnarray}
where we used the fact that $\alpha\mapsto\sqrt{\alpha}$ is operator
monotone along with $\lambda\ge n^{-1/2}$. Using (\ref{Eq:tempoo}), $\Vert (II^\prime)\Vert_\eu{H}$ can be
bounded as
\begin{eqnarray}
\Vert (II^\prime)\Vert_\eu{H}
&{}\le{}& \Vert \sqrt{\hat{C}}g_\lambda(\hat{C})\Vert \Vert \hat{\xi}+\hat{
C } f_0\Vert_\eu{H}+ \Vert \sqrt {\hat{C} } r_\lambda(\hat { C } )\hat { C }
^\beta\Vert 
\Vert
C^{-\beta}f_0\Vert_\eu{H}\nonumber\\
&{}{}&\qquad\qquad+\Vert
\sqrt{\hat{C}}r_\lambda(\hat{C})\Vert \Vert C^{\beta}-\hat{C}^{\beta}\Vert \Vert C^{-\beta}f_0\Vert_\eu{H}\nonumber
\end{eqnarray}
where $$\Vert \sqrt{\hat{C}}g_\lambda(\hat{C})\Vert\le
\sqrt{\frac{A_gB_g}{\lambda}},\,\,\,\Vert
\sqrt{\hat{C}}r_\lambda(\hat{C})\hat{C}^\beta\Vert\le
(\gamma_{\beta+\frac{1}{2}}\vee
\gamma_{\eta_0})\lambda^{\min\{\beta+\frac{1}{2},\eta_0\}}$$ and $$
\Vert \sqrt{\hat{C}}r_\lambda(\hat{C})\Vert\le
(\gamma_{\frac{1}{2}}\vee\gamma_{\eta_0})\lambda^{\min\{\frac{1}{2},\eta_0\}}$$
with $\Vert \hat{C}f_0+\hat{\xi}\Vert$ and $\Vert
C^{\beta}-\hat{C}^{\beta}\Vert$ bounded as in part (i) above. Here $(a\vee
b):=\max\{a,b\}$. Combining $\Vert (I^\prime)\Vert_\eu{H}$ and $\Vert (II^\prime)\Vert_\eu{H}$,
we obtain the required result.
\vspace{2mm}\\
(iii) The proof follows the ideas in the proof of Theorems~\ref{Thm:rates} and
\ref{Thm:density}. Consider
$f_{g,\lambda,n}-f_0=-g_\lambda(\hat{C})(\hat{C}f_0+\hat{\xi})+r_\lambda(\hat{C
})f_0$
so that
\begin{eqnarray}\Vert
f_{g,\lambda,n}-f_0\Vert_\eu{H}&{}\le{}&\Vert
C^{-1}\Vert \Vert C g_\lambda(\hat{C})(\hat{C}f_0+\hat{\xi})\Vert_\eu{H}
+\Vert C^{-1}\Vert \Vert C r_\lambda(\hat { C
})f_0\Vert_\eu{H}\nonumber\\
&{}\le{}& \Vert
C^{-1}\Vert \Vert \hat{C}f_0+\hat{\xi}\Vert_\eu{H} \left(\Vert \hat{C}
g_\lambda(\hat{C})\Vert + \Vert \hat{C}-C\Vert
\Vert g_\lambda(\hat{C})\Vert \right)\nonumber\\
&{}{}& \qquad+\Vert C^{-1}\Vert \Vert f_0\Vert_\eu{H}\left(\Vert \hat{C}
r_\lambda(\hat { C
})\Vert + \Vert\hat{C}-C\Vert
\Vert r_\lambda(\hat { C
})\Vert\right).\nonumber
\end{eqnarray}
Therefore $\Vert
f_{g,\lambda,n}-f_0\Vert_\eu{H}=O_{p_0}(n^{-1/2})+O\left(\lambda^{\min\{1,\eta_0\}}\right)$ where we used the fact that $\lambda\ge n^{-1/2}$ and the result follows.\QEDA
\subsection{Proof of Theorem~\ref{Thm:misspecified}}\label{subsec:supp-misspecified}
Before we analyze
$J(p_0\Vert p_{f_{\lambda,n}})$, we need a small calculation for notational
convenience. For any probability densities $p,q\in C^1$,
it
is clear that $\sqrt{2J(p\Vert q)}=\lv \lv\nabla\log p-\nabla\log q \rv_2
\rv_{L^2(p)}$. We generalize this by defining 
$$\sqrt{2J(p\Vert q\Vert \mu)}:=\lv \lv\nabla\log p-\nabla\log q \rv_2
\rv_{L^2(\mu)}.$$
Clearly, if $\mu=p$, then $J(p\Vert q\Vert\mu)$ matches with $J(p\Vert q)$.
Therefore, for probability densities $p,q,r\in C^1$,
\begin{equation}\sqrt{J(p\Vert r\Vert p)}\le
\sqrt{J(p\Vert q\Vert p)}+\sqrt{J(q\Vert
r\Vert p)}.\label{Eq:triangle}\end{equation}
Based on (\ref{Eq:triangle}), we have
\begin{eqnarray}\sqrt{\inf_{p\in\Cal{P}}J(p_0\Vert p)}\le
\sqrt{J(p_0\Vert p_{f_{\lambda,n}}\Vert p_0)}&{}\le{}&
\sqrt{J(p_0\Vert p_{f^\ast}\Vert
p_0)}+\sqrt{J(p_{f^\ast}\Vert p_{f_{\lambda,n}}\Vert p_0)}
\nonumber
\end{eqnarray}
\begin{eqnarray}
&{}={}&\sqrt{\inf_{p\in\Cal{P}}J(p_0\Vert p\Vert
p_0)}+\sqrt{J(p_{f^\ast}\Vert p_{f_{\lambda,n}}\Vert p_0)}\nonumber\\
&{}={}&\sqrt{\inf_{p\in\Cal{P}}J(p_0\Vert p)}+\frac{1}{\sqrt{2}}\sqrt{\langle
f_{\lambda,n}-f^*,C(f_{\lambda,n}-f^*)\rangle_\eu{H}} \nonumber\\
&{}={}&\sqrt{\inf_{p\in\Cal{P}}J(p_0\Vert
p)}+\frac{1}{\sqrt{2}}\Vert\sqrt{C}(f_{
\lambda,n}-f^*)\Vert_\eu{H}\label{Eq:main}\\
&{}={} &
\sqrt{\inf_{p\in\Cal{P}}J(p_0\Vert p)}+\frac{1}{\sqrt{2}}\Vert\sqrt{C}(f_{
\lambda,n}-f_\lambda)\Vert_\eu{H}+\frac{1}{\sqrt{2}}\Cal{A}^*(\lambda),\nonumber
\end{eqnarray}
where $\Cal{A}^*(\lambda)=\Vert\sqrt{C}(f_{\lambda}
-f^*)\Vert_\eu{H}$. The result simply follows from the proof of
Theorem~\ref{Thm:density}, where we showed that
$\Vert
\sqrt{C}(f_{\lambda,n}-f_\lambda)\Vert_\eu{H}=O_{p_0}\left(\frac{1}{\sqrt{\lambda
n}}\right)$ and $\Cal{A}^*(\lambda)=O(\lambda^{\min\{1,\beta+\frac{1}{2}\}})$ if
$f^\ast\in\Cal{R}(C^\beta)$ for $\beta\ge 0$ as
$\lambda\rightarrow 0$, $n\rightarrow \infty$. When $\Vert C^{-1}\Vert<\infty$,
we bound $\Vert\sqrt{C}(f_{
\lambda,n}-f^*)\Vert_\eu{H}$ in (\ref{Eq:main}) as $\Vert\sqrt{C}\Vert \Vert f_{
\lambda,n}-f^*\Vert_\eu{H}$ where $\Vert f_{
\lambda,n}-f^*\Vert_\eu{H}$ is in turn bounded as
in (\ref{Eq:fff}).\QEDA
\subsection{Proof of
Proposition~\ref{pro:operators}}\label{subsec:supp-operators}
For $f\in\eu{H}$, we have 
\begin{eqnarray}
\Vert
f\Vert^2_{\Cal{W}_2}&{}={}&\int_\Omega
\sum^d_{i=1} \left(\partial_i f\right)^2p_0(x)\,dx
\le
\Vert f\Vert^2_\eu{H}\int_\Omega
\sum^d_{i=1}\left\Vert\partial_i k(x,\cdot)\right\Vert^2_\eu{H}p_0(x)\,dx<\infty,\nonumber
\end{eqnarray}
which means $f\in\Cal{W}_2(\Omega,p_0)$ and therefore $[f]_\sim\in W_2(\Omega,p_0)$. Since $\Vert I_kf\Vert_{W_2}=\Vert [f]_\sim\Vert_{\Cal{W}^\sim_2}=\Vert f\Vert_{\Cal{W}_2}\le c\Vert f\Vert_\eu{H}<\infty$
where $c$ is some constant, it is clear that $I_k$ is a continuous map from $\eu{H}$ to $W_2(\Omega,p_0)$. The
adjoint
$S_k:W_2(\Omega,p_0)\rightarrow \eu{H}$ of $I_k:\eu{H}\rightarrow W_2(\Omega,p_0)$ is
defined by the relation
$\langle S_kf,g\rangle_\eu{H}=\langle f,I_kg\rangle_{W_2},\,\,f\in
W_2(\Omega,p_0),\,g\in\eu{H}$. If $f:=[h]_\sim\in\Cal{W}^\sim_2(\Omega,p_0)$, then 
$$\langle [h]_\sim,I_kg\rangle_{W_2}=\langle
[h]_\sim,[g]_\sim\rangle_{\Cal{W}^\sim_2}=\sum_{|\alpha|=1}\int_\Omega (\partial^\alpha
h)(x)(\partial^\alpha g)(x)\,p_0(x)\,dx.$$
For $y\in\Omega$ and $g=k(\cdot,y)$, this yields
\begin{eqnarray}S_k[h]_\sim(y)=\langle
S_k[h]_\sim,k(\cdot,y)\rangle_\eu{H}&{}={}&\langle
[h]_\sim, I_k k(\cdot,y)\rangle_{W_2}=\int_\Omega
\sum^d_{i=1}\partial_i k(x,y)\partial_ih(x)p_0(x)\,dx.\nonumber
\end{eqnarray}
We now show that $I_k$ is Hilbert-Schmidt. Since $\eu{H}$ is separable,
let $(e_l)_{l\ge 1}$ be an ONB of $\eu{H}$. Then we have
\begin{eqnarray}
\sum_{l}\Vert I_k e_l\Vert^2_{W_2}=\sum_{l}\int_\Omega
\sum^d_{i=1}\left(\partial_i e_l(x)\right)^2p_0(x)\,dx&{}={}&
\int_\Omega\sum^d_{i=1}\sum_l\left\langle e_l,\partial_i
k(x,\cdot)\right\rangle^2_\eu{H}p_0(x)\,dx\nonumber\\
&{}={}&\int_\Omega\sum^d_{i=1}\left\Vert\partial_i
k(x,\cdot)\right\Vert^2_\eu{H}p_0(x)\,dx<\infty,\nonumber
\end{eqnarray}
which proves that $I_k$ is Hilbert-Schmidt (hence compact) and
therefore $S_k$ is also Hilbert-Schmidt and compact. The other
assertions about
$S_kI_k$ and $I_kS_k$ are straightforward.\QEDA
\subsection{Proof of Theorem~\ref{Thm:misspecified-1}}\label{subsec:misspecified}
By slight abuse of notation, $f_\star$ is used to denote $[f_\star]_\sim$ in the
proof for simplicity. For $f\in\Cal{F}$, we have
\begin{equation}
J(p_0\Vert p_f)=\frac{1}{2}\Vert
I_k f-f_\star\Vert^2_{W_2}=\frac{1}{2}\langle E_k
f,f\rangle_\eu{H}-\langle
S_kf_\star,f\rangle_\eu{H}+\frac{1}{2}\Vert f_\star\Vert^2_{W_2}.\nonumber
\end{equation}
Since $k$ satisfies \textbf{(C)} it is easy to verify that $\langle
S_kf_\star,f\rangle_\eu{H}=\langle f,-\xi\rangle_\eu{H},\,\forall\,f\in\eu{H}$ (see proof of Theorem~\ref{Thm:score}(ii)). This implies $S_kf_\star=-\xi$ 
and 
\begin{equation} J(p_0\Vert p_f)=\frac{1}{2}\langle E_k
f,f\rangle_\eu{H}+\langle
f,\xi\rangle_\eu{H}+\frac{1}{2}\Vert
f_\star\Vert^2_{W_2},\label{Eq:object}\end{equation}
where $\xi$ is defined in Theorem~\ref{Thm:score}(ii), and
$E_k$ is precisely the operator $C$ defined in
Theorem~\ref{Thm:score}(ii). Following the proof of Theorem~\ref{Thm:score}(ii),
for $\lambda>0$, it is easy to show that the unique minimizer of the regularized
objective, $J(p_0\Vert p_f)+\frac{\lambda}{2}\Vert f\Vert^2_\eu{H}$ exists and
is
given by
\begin{equation}
 f_\lambda=-(E_k+\lambda I)^{-1}\xi=(E_k+\lambda
I)^{-1}S_kf_\star.\label{Eq:flambda}
\end{equation}
We would like to reiterate that
(\ref{Eq:object}) and (\ref{Eq:flambda}) also match with their counterparts in
Theorem~\ref{Thm:score} and therefore as in Theorem~\ref{Thm:score}(iv), an
estimator of $f_\star$ is given by $f_{\lambda,n}=-(\hat{E}_k+\lambda
I)^{-1}\hat{\xi}$. In other words, this is the same as in Theorem~\ref{Thm:score}(iv) since
$\hat{E}_k=\hat{C}$,
and can be solved by a simple linear system provided in
Theorem~\ref{Thm:representer}. Here $\hat{E}_k$ is
the empirical estimator of $E_k$. Now consider
\begin{eqnarray}
\sqrt{2\,J(p_0\Vert p_{f_{\lambda,n}})}=\Vert
I_k f_{\lambda,n}-f_\star\Vert_{W_2}
&{}\le{}& \Vert I_k(f_{\lambda,n}-f_\lambda)\Vert_{W_2}+\Vert
I_kf_\lambda-f_\star\Vert_{W_2}\nonumber\\
&{}={}&\Vert
\sqrt{E_k}(f_{\lambda,n}-f_\lambda)\Vert_\eu{H}+\Cal{B}(\lambda),
\label{Eq:main-1}
\end{eqnarray}
where $\Cal{B}(\lambda):=\Vert
I_kf_\lambda-f_\star\Vert_{W_2}$. The proof now proceeds using the following
decomposition, equivalent to
the one used in the proof of Theorem~\ref{Thm:rates}(i), i.e.,
\begin{eqnarray}
f_{\lambda,n}-f_\lambda&{}={}&-(\hat{E}_k+\lambda
I)^{-1}\hat{\xi}-f_\lambda\nonumber\\
&{}={}&-(\hat{E}_k+\lambda I)^{-1}(\hat{\xi}+\hat{E}_kf_\lambda+\lambda
f_\lambda)\nonumber\\
&{}\stackrel{(\dagger)}{=}{}&-(\hat{E}_k+\lambda I)^{-1}(\hat{\xi}+\hat{E}_kf_\lambda+S_kf_\star
-E_kf_\lambda-\hat{S}_kf_\star+\hat{S}_kf_\star),\nonumber
\end{eqnarray}
where we used (\ref{Eq:flambda}) in $(\dagger)$.  $\hat{S}_kf_\star$ is well-defined as it is the empirical version of 
the restriction of $S_k$ to 
$\Cal{W}^\sim_2(p_0)$. Since $S_kf_\star-E_kf_\lambda=S_k(f_\star-I_kf_\lambda)$ 
and $\hat{S}_kf_\star-\hat{E}_kf_\lambda=\hat{S}_k(f_\star-I_kf_\lambda)$, we have
$$f_{\lambda,n}-f_\lambda=-(\hat{E}_k+\lambda I)^{-1}(\hat{\xi}+\hat{S}_kf_\star)+(\hat{E}_k+\lambda I)^{-1}(\hat{S}_k-S_k)(f_\star-I_kf_\lambda)$$
and so
\begin{equation}
\Vert
\sqrt{E_k}(f_{\lambda,n}-f_\lambda)\Vert_\eu{H}\le  \Vert\sqrt{E_k}(\hat{E}_k+\lambda
I)^{-1}\Vert\left( \Vert
\hat{\xi}+\hat{S}_kf_\star\Vert_\eu{H}+\Vert (\hat{S}_k-S_k)(f_\star-I_kf_\lambda)\Vert_\eu{H}\right). \label{Eq:main-2}
\end{equation}
It follows from Proposition~\ref{pro:certain}(v) that \begin{equation}\Vert\sqrt{E_k}(\hat{E}_k+\lambda
I)^{-1}\Vert\lesssim\frac{1}{\sqrt{\lambda}}\label{Eq:bound-t1}\end{equation} for $n\ge \frac{c}{\lambda^2}$ where $c$ is a sufficiently large
constant that does not depend on $n$ and $\lambda$. 
Following the proof of
Proposition~\ref{pro:certain}(i), we have
$$\bb{E}\Vert\hat{\xi}+\hat{S}_kf_\star\Vert^2_\eu{H}=\frac{n-1}{n}\Vert \xi+S_kf_\star\Vert^2_\eu{H}+\frac{1}{n}
\int_\Omega \left\Vert \sum^d_{i=1}\partial_i k(x,\cdot)\partial_i f_\star+\xi_x\right\Vert^2_\eu{H}p_0(x)\,dx$$
wherein the first term is zero as $S_kf_\star+\xi=0$ and since $$\left\Vert \sum^d_{i=1}\partial_i k(x,\cdot)\partial_i f_\star+\xi_x\right\Vert^2_\eu{H}
\le 2\Vert \xi_x\Vert^2_\eu{H}+2\chi\Vert\nabla f_\star\Vert^2_2,$$ the integral in the second term is finite because 
of $\textbf{(D)}$ and $f^*\in W_2(\Omega,p_0)$. Therefore, an application of Chebyshev's inequality yields
\begin{equation}\Vert\hat{\xi}+\hat{S}_kf_\star\Vert_\eu{H}=O_{p_0}(n^{-1/2}).\label{Eq:bound-tt}\end{equation}
We now show that $\Vert (\hat{S}_k-S_k)(f_\star-I_kf_\lambda)\Vert_\eu{H}=O_{p_0}(\Cal{B}(\lambda)n^{-1/2})$. 
To this end, define $g:=f_\star-I_kf_\lambda$ and consider 
\begin{eqnarray}
\bb{E}_{p_0}\Vert \hat{S}_kg-S_kg\Vert^2_\eu{H}&{}={}&\frac{\int_\Omega \Vert \sum^d_{i=1}\partial_i k(x,\cdot)\partial_i g(x)\Vert^2_\eu{H}p_0(x)\,dx-\Vert S_k g\Vert^2_\eu{H}}{n}\le\frac{\chi}{n}\Vert g\Vert^2_{W_2},\nonumber
\end{eqnarray}
which therefore yields the claim through an application of Chebyshev's inequality. Using this along with (\ref{Eq:bound-t1}) and (\ref{Eq:bound-tt}) 
in (\ref{Eq:main-2}), and using the resulting bound in (\ref{Eq:main-1}) yields
\begin{equation}
\sqrt{2\,J(p_0\Vert p_{f_{\lambda,n}})}\le
O_{p_0}\left(\frac{1}{\sqrt { \lambda n}}+\frac{\Cal{B}(\lambda)}{\sqrt{\lambda
n}} \right)+\Cal { B } (\lambda).\label{Eq:blam-0}
\end{equation}
(i) We bound $\Cal{B}(\lambda)$ as follows. First note that $$\Cal{B}(\lambda)=\Vert
I_k(S_kI_k+\lambda I)^{-1}S_kf_\star-f_\star\Vert_{W_2}=\Vert (T_k+\lambda
I)^{-1}T_kf_\star-f_\star\Vert_{W_2}$$ and so for any $h\in\eu{H}$, we have
\begin{eqnarray}
\Cal{B}(\lambda)&{}={}&\Vert
(T_k+\lambda I)^{-1}T_kf_\star-f_\star\Vert_{W_2}\nonumber\\
&{}\le{}& \underbrace{\Vert
((T_k+\lambda
I)^{-1}T_k-I)(f_\star-I_k h)\Vert_{W_2}}_{(I)}+\underbrace{\Vert (T_k+\lambda
I)^{-1}T_kI_kh-I_k h\Vert_{W_2}}_{(II)}.\label{Eq:blam}
\end{eqnarray}
Since $T_k$ is a self-adjoint compact operator, there exists $(\alpha_l)_{l\in
\bb{N}}$ and ONB $(\phi_l)_{l\in\bb{N}}$ of $\overline{\Cal{R}(T_k)}$ so
that
$
T_k=\sum_l\alpha_l\langle \phi_l,\cdot\rangle_{W_2}\phi_l.$ 
Let $(\psi_j)_{j\in\bb{N}}$ be the orthonormal basis of $\Cal{N}(T_k)$. Then we
have
\begin{eqnarray}
(I)^2&{}={}&\sum_l\left(\frac{\alpha_l}{
\alpha_l+\lambda }-1\right)^2\langle
f_\star-I_kh,\phi_l\rangle^2_{W_2}+\sum_j\langle
f_\star-I_kh,\psi_j\rangle^2_{W_2}\nonumber\\
&{}\le{}& \sum_l \langle
f_\star-I_kh,\phi_l\rangle^2_{W_2}+\sum_j\langle
f_\star-I_kh,\psi_j\rangle^2_{W_2}=\Vert
f_\star-I_kh\Vert^2_{W_2}.\label{Eq:blam-1}
\end{eqnarray}
From $(T_k+\lambda I)^{-1}T_k=I_k(E_k+\lambda I)^{-1}S_k$ and $S_kI_kh=E_kh$,
we have
\begin{eqnarray}
(II)&{}={}&\Vert
I_k(E_k+\lambda I)^{-1}E_kh-I_kh\Vert_{W_2}\nonumber\\
&{}={}&\Vert \sqrt{E_k}(E_k+\lambda
I)^{-1}E_kh-\sqrt{E_k}h\Vert_{\eu{H}}\le\Vert
h\Vert_\eu{H}\sqrt{\lambda},\label{Eq:blam-2}
\end{eqnarray}
where the inequality follows from Proposition~\ref{pro:approxerror}(ii). Using
(\ref{Eq:blam-1}) and (\ref{Eq:blam-2}) in (\ref{Eq:blam}), we obtain
$\Cal{B}(\lambda)\le \Vert f_\star-I_kh\Vert_{W_2}+\Vert
h\Vert_\eu{H}\sqrt{\lambda}$,
using which in (\ref{Eq:blam-0}) yields
$$
\sqrt{2\,J(p_0\Vert p_{f_{\lambda,n}})}\le \Vert
f_\star-I_kh\Vert_{W_2}+O_{p_0}\left(\frac{1}{\sqrt{
\lambda n}}\right)+\Vert
h\Vert_\eu{H}\sqrt{\lambda}.
$$
Since the above inequality holds for any $h\in\eu{H}$, we therefore have
\begin{eqnarray}\sqrt{2\,J(p_0\Vert p_{f_{\lambda,n}})}&{}\le{}&
\inf_{h\in\eu{H}}\left(\Vert
f_\star-I_kh\Vert_{W_2}+\sqrt{\lambda}\Vert
h\Vert_\eu{H}\right)+O_{p_0}\left(\frac{1}{\sqrt{
\lambda n}}\right)\nonumber\\
&{}={}&
K(f_\star,\sqrt{\lambda},W_2(p_0),I_k(\eu{H}))+O_{p_0}\left(\frac{1}{\sqrt{
\lambda n}}\right)\label{Eq:kfunc-final}
\end{eqnarray}
where the $K$-functional is defined in (\ref{Eq:kfunc}). Note that
$I_k(\eu{H})\cong\eu{H}/\eu{H}\cap\mathbb{R}$ is continuously embedded in
$W(p_0)$. 
From (\ref{Eq:kfunc}),
it is clear that the $K$-functional as a function of $t$ is an infimum over a
family of affine linear and increasing functions and therefore is
concave, continuous and increasing w.r.t.~$t$. This means, in
(\ref{Eq:kfunc-final}), as
$\lambda\rightarrow 0$, 
$$
K(f_\star,\sqrt{\lambda},W_2(p_0),I_k(\eu{H}))\rightarrow
\inf_{h\in\eu{H}}\Vert
f_\star-I_kh\Vert_{W_2}=\sqrt{2
\inf_{p\in\Cal{P}}J(p_0\Vert p)}.$$ Since $J(p_0\Vert p_{f_{\lambda,n}})\ge
\inf_{p\in\Cal{P}}J(p_0\Vert p)$, we have that
$J(p_0\Vert p_{f_{\lambda,n}})\rightarrow
\inf_{p\in\Cal{P}}J(p_0\Vert p)$ as $\lambda\rightarrow 0$, $\lambda
n\rightarrow\infty$ and $n\rightarrow \infty$.\vspace{2mm}\\
(ii) Recall $\Cal{B}(\lambda)$ from (i). From
Proposition~\ref{pro:approxerror}(i) it follows that
$\Cal{B}(\lambda)\rightarrow 0$ as $\lambda\rightarrow 0$ if
$f_\star\in\overline{\Cal{R}(T_k)}$. Therefore, (\ref{Eq:blam-0}) reduces to
$ \sqrt{2\,J(p_0\Vert p_{f_{\lambda,n}})}\le
O_{p_0}\left(\frac{1}{\sqrt{\lambda n}}\right)+\Cal{B}(\lambda)$ and the
consistency result follows. If $f_\star\in \Cal{R}(T^\beta_k)$ for some
$\beta>0$, then the rates follow from Proposition~\ref{pro:approxerror} by
noting that $\Cal{B}(\lambda)\le \max\{1,\Vert
T_k\Vert^{\beta-1}\}\lambda^{\min\{1,\beta\}}\Vert
T^{-\beta}_kf_\star\Vert_{W_2}$ and choosing
$\lambda=n^{-\max\left\{\frac{1}{3},\frac{1}{2\beta+1}\right\}}$.\vspace{2mm}\\
(iii) This simply follows from an analysis similar to the one used in
the proof of Theorem~\ref{Thm:rates}(iii).\QEDA
\subsection{Proof of
Proposition~\ref{Thm:approx-fd}}\label{subsec:supp-thm-approx-fd}
For any $p\in \Cal{P}_{\text{FD}}$, define $f:=\log\frac{p}{q_0}$, which implies that
$[f]_\sim\in W_2(p)$. Since $I_k(\eu{H})$ is dense in $W_2(p)$, we have
for
any
$\epsilon>0$, there exists $g\in\eu{H}$ such that $\Vert
[f]_\sim-I_kg\Vert_{W_2}\le\sqrt{2\epsilon}$. 
For a given $g\in\eu{H}$, pick
$p_g\in\Cal{P}$. Therefore, 
\begin{equation}
J(p\Vert p_g)=\frac{1}{2}\int_\Omega
p(x)\left\Vert\nabla\log p-\nabla \log p_g\right\Vert^2_2\,dx=
\frac{1}{2}\Vert [f]_\sim-I_kg\Vert^2_{W_2}\le\epsilon \nonumber
\end{equation}
and the result follows.\QEDA
\begin{appendices}\label{appendix}
\numberwithin{equation}{section}
\section{Appendix: Technical Results}\label{Sec:suppresults}
In this appendix, we present some technical results that are used in
the proofs.
\subsection{Bounds on Various Distances Between $p_f$ and $p_g$}
In the following result, 
claims (iii) and (iv) are quoted from Lemma 3.1 of
\cite{Vaart-08}.
\begin{appxlem}\label{lem:distances}
Define $\Cal{P}_\infty:=\left\{p_f=e^{f-A(f)}q_0\,:\,f\in \ell^\infty(\Omega)\right\},$ where $q_0$ is a
probability density on $\Omega\subseteq\bb{R}^d$ and $\ell^\infty(\Omega)$ is the
space of bounded measurable functions on $\Omega$. Then for any
$p_f,p_g\in\Cal{P}_\infty$, we have
\begin{itemize}
 \item[(i)] $\Vert p_f-p_g\Vert_{L^r(\Omega)}\le 2e^{2\Vert
f-g\Vert_\infty}e^{2\min\{\Vert f\Vert_\infty, \Vert
g\Vert_\infty\}}\Vert f-g\Vert_\infty\Vert q_0\Vert_{L^r(\Omega)}$ for any
$1\le r\le\infty$;\vspace{1mm}
\item[(ii)] $\Vert p_f-p_g\Vert_{L^1(\Omega)}\le  2e^{\Vert
f-g\Vert_\infty}\Vert
f-g\Vert_\infty$;\vspace{1mm}
 \item[(iii)] $KL(p_f\Vert p_g)\le c\,\Vert f-g\Vert^2_\infty e^{\Vert
 f-g\Vert_\infty}\left(1+\Vert f-g\Vert_\infty\right)$ where $c$ is a universal constant;\vspace{1mm}
\item[(iv)] $h(p_f,p_g)\le e^{\Vert
f-g\Vert_\infty/2}\Vert
f-g\Vert_\infty.$
\end{itemize}
\end{appxlem}
\begin{proof}
\emph{(i)} Define $B(f):=\int e^f q_0\,dx$. 
Consider
\begin{eqnarray}
 \left\Vert
p_f-p_g\right\Vert_{L^r(\Omega)}
&{}={}&\left\Vert\frac{e^{f}q_0}{B(f)}-\frac{e^{g}q_0}{B(g)}
\right\Vert_{L^r(\Omega)}=\frac{\left\Vert
e^{f}q_0B(g)-e^{g}q_0B(f)\right\Vert_{L^r(\Omega)}}{B(f)B(g)}
\nonumber\\
&{}={}&\frac{\left\Vert e^{f}q_0
\left(B(g)-B(f)\right)+\left(e^{f}-e^{g}\right)q_0B(f)
\right\Vert_{L^r(\Omega)}}{B(f)B(g)}\nonumber\\
&{}\le{}&\frac{\left\Vert e^{f}q_0
\left(B(g)-B(f)\right)\right\Vert_{L^r(\Omega)}}{B(f)B(g)}
+\frac{\left\Vert
\left(e^{f}-e^{g}\right)q_0B(f)\right\Vert_{L^r(\Omega)}}{B(f)B(g)}
\nonumber\\
&{}\le{}&\frac{\left|B(g)-B(f)\right|\Vert
e^{f}q_0\Vert_{L^r(\Omega)}}{B(g)B(f)}
+\frac{\left\Vert
(e^{f}-e^{g})q_0\right\Vert_{L^r(\Omega)}}{B(g)}.
\label{Eq:list}
\end{eqnarray}
Observe that
$$|B(f)-B(g)|\le \int_\Omega |e^f-e^g|q_0\,dx= \int_\Omega e^g|e^{f-g}-1|q_0\,dx\le e^{\Vert
f-g\Vert_\infty}\Vert f-g\Vert_\infty B(g)$$
since $|e^{u-v}-1|\le |u-v|e^{|u-v|}$ for any $u,v\in\bb{R}$. Similarly,
$$\left\Vert
(e^{f}-e^{g})q_0\right\Vert_{L^r(\Omega)}\le e^{\Vert
f-g\Vert_\infty}\Vert f-g\Vert_\infty \Vert e^g q_0\Vert_{L^r(\Omega)}.$$ Using
these above, we obtain
\begin{equation}\left\Vert
p_f-p_g\right\Vert_{L^r(\Omega)}\le e^{\Vert
f-g\Vert_\infty}\Vert
f-g\Vert_\infty\left(\frac{\Vert
e^fq_0\Vert_{L^r(\Omega)}}{B(f)}+\frac{\Vert
e^gq_0\Vert_{L^r(\Omega)}}{B(g)}\right).\label{Eq:f}\end{equation}
Since $\Vert e^f q_0\Vert_{L^r(\Omega)}\le e^{\Vert f\Vert_\infty}\Vert
q_0\Vert_{L^r(\Omega)}$ and $B(f)\ge e^{-\Vert f\Vert_\infty}$, from
(\ref{Eq:f}) we obtain
\begin{eqnarray}\left\Vert
p_f-p_g\right\Vert_{L^r(\Omega)}&{}\le{}& e^{\Vert
f-g\Vert_\infty}\Vert
f-g\Vert_\infty\Vert q_0\Vert_{L^r(\Omega)}\left(e^{2\Vert
f\Vert_\infty}+e^{2\Vert g\Vert_\infty}\right).\nonumber\\
&{}\le {}&2e^{\Vert
f-g\Vert_\infty}\Vert
f-g\Vert_\infty\Vert q_0\Vert_{L^r(\Omega)}e^{2\max\{\Vert
f\Vert_\infty, \Vert g\Vert_\infty\}}\nonumber\\
&{}\le {}&2e^{2\Vert
f-g\Vert_\infty}\Vert
f-g\Vert_\infty\Vert q_0\Vert_{L^r(\Omega)}e^{2\min\{\Vert
f\Vert_\infty,\Vert g\Vert_\infty\}}\nonumber
\end{eqnarray}
where we used $\max\{a, b\}\le \min\{a, b\}+|a-b|$ for $a,b\ge 0$ in the
last line
above.\vspace{1mm}\\
\emph{(ii)} This simply follows from (\ref{Eq:f}) by using $r=1$.
\vspace{-5mm}
\end{proof}
\subsection{General Representer Theorem}\label{subsec:representer}
The following is the general representer theorem for abstract Hilbert spaces.
\begin{appxthm}[General representer theorem]\label{thm:representer}
Let $H$ be a real Hilbert space and let $(\phi_i)^m_{i=1}\in H^m$. Suppose $J:H\rightarrow\bb{R}$ be such that
$J(f)=V\left(\langle f,\phi_1\rangle_H,\ldots, \langle f,\phi_m\rangle_H\right),\,\,f\in H$
where $V:\bb{R}^n\rightarrow\bb{R}$ is a convex differentiable function. Define
$$f_\lambda=\arg\inf_{f\in H}J(f)+\frac{\lambda}{2}\Vert f\Vert^2_H,$$
where $\lambda>0$. Then there exists $(\alpha_i)^m_{i=1}\in\bb{R}^m$ such that
$f_\lambda=\sum^m_{i=1}\alpha_i\phi_i$
where $\bm{\alpha}:=(\alpha_1,\ldots,\alpha_m)$ satisfies the following (possibly nonlinear) equation
$$\lambda \bm{\alpha}+\nabla V\left(\bm{K}\bm{\alpha}\right)=0,$$
with $\bm{K}$ being a linear map on $\bb{R}^m$ and $(\bm{K})_{i,j}=\langle \phi_i,\phi_j\rangle_H,\,i\in[m],\,j\in[m]$.
\end{appxthm}
\begin{proof}
Define $A:H\rightarrow\bb{R}^m$, $f\mapsto (\langle f,\phi_i\rangle_H)^m_{i=1}$. Then $f_\lambda=\arg\inf_{f\in H}V(Af)+\frac{\lambda}{2}\Vert f\Vert^2_H$.
Therefore, Fermat's rule yields
\begin{eqnarray}
0=A^*\nabla V(Af_\lambda)+\lambda f_\lambda &{}\Leftrightarrow {}& f_\lambda =A^*\left(-\frac{1}{\lambda}\nabla V(Af_\lambda)\right)\nonumber\\
&{}\Leftrightarrow{}& (\exists\,\bm{\alpha}\in\bb{R}^m)\quad f_\lambda=A^*\bm{\alpha},\,\,\bm{\alpha}=-\frac{1}{\lambda}\nabla V(Af_\lambda)\nonumber\\
&{}\Leftrightarrow{}& (\exists\,\bm{\alpha}\in\bb{R}^m)\quad f_\lambda=A^*\bm{\alpha},\,\,\bm{\alpha}=-\frac{1}{\lambda}\nabla V(AA^*\bm{\alpha}),\nonumber
\end{eqnarray}
where $A^*:\bb{R}^m\rightarrow H$ is the adjoint of $A$ which can be obtained as follows. Note that
$$\langle Af,\bm{\alpha}\rangle=\sum^m_{i=1}\alpha_i\langle f,\phi_i\rangle_H=\left\langle f,\sum^m_{i=1}\alpha_i\phi_i\right\rangle_H\qquad (\forall\,f\in H)\,\,(\forall\,\bm{\alpha}\in\bb{R}^m)$$
and thus $A^*\bm{\alpha}=\sum^m_{i=1}\alpha_i\phi_i$. Therefore $AA^*\bm{\alpha}=\sum^m_{j=1}\alpha_j A\phi_j=\sum^m_{j=1}\alpha_j (\langle \phi_j,\phi_i\rangle_H)^m_{i=1}$,
and so for every $i\in[m]$, $(AA^*\bm{\alpha})_i=\sum^m_{j=1}\langle \phi_j,\phi_i\rangle_H \alpha_j$ and hence $AA^*=\bm{K}$.\vspace{-5mm}
\end{proof}
\subsection{Bounds on Approximation Errors, $\Cal{A}_0(\lambda)$ and
$\Cal{A}_{\frac{1}{2}}(\lambda)$}
The following result is quite well-known in the linear inverse problem theory 
\citep{Engl-96}.
\begin{appxpro}\label{pro:approxerror}
Let $C$ be a bounded, self-adjoint compact operator on a separable Hilbert
space $H$. For $\lambda>0$ and $f\in H$, define $f_\lambda:=(C+\lambda
I)^{-1}Cf$ and $\Cal{A}_\theta(\lambda):=\Vert C^{\theta}(f_\lambda-f)\Vert_H$
for $\theta\ge 0$. Then the following hold.
\begin{itemize} \item[(i)] For any $\theta>0$, $\Cal{A}_\theta(\lambda)\rightarrow 0$ as
$\lambda\rightarrow 0$ and if $f\in\overline{\Cal{R}(C)}$, then $\Cal{A}_0(\lambda)\rightarrow 0$
as $\lambda\rightarrow 0$.\vspace{1mm}
\item[(ii)] If $f\in\Cal{R}(C^\beta)$ for $\beta\ge 0$ and $\beta+\theta>0$,
then $$\Cal{A}_\theta(\lambda)\le \max\{1,\Vert C\Vert^{\beta+\theta-1}\}
\lambda^{\min\{1,\beta+\theta\}} \Vert
C^{-\beta}f\Vert_H.$$
\end{itemize}
\end{appxpro}
\begin{proof}
(i) Since $C$ is bounded, compact, and self-adjoint, the Hilbert-Schmidt theorem
\citep[Theorems VI.16, VI.17]{Reed-72} ensures that $C=\sum_l \alpha_l \phi_l\langle
\phi_l,\cdot\rangle_H,$
where $(\alpha_l)_{l\in\bb{N}}$ are the positive eigenvalues and
$(\phi_l)_{l\in\bb{N}}$
are the corresponding unit eigenvectors that form an ONB for $\Cal{R}(C)$. Let
$\theta=0$. Since $f\in\overline{\Cal{R}(C)}$, 
\begin{eqnarray}
 \Cal{A}^2_0(\lambda)&{}={}&\left\Vert
(C+\lambda
I)^{-1}Cf- f\right\Vert^2_H=\left\Vert \sum_i
\frac{\alpha_i}{\alpha_i+\lambda}\langle
f,\phi_i\rangle_H\phi_i-\sum_i \langle
f,\phi_i\rangle_H\phi_i\right\Vert^2_H\nonumber\\
&{}={}&\left\Vert \sum_i\frac{\lambda}{\alpha_i+\lambda}\langle
f,\phi_i\rangle_H\phi_i\right\Vert^2_H=\sum_i\left(\frac{
\lambda}{\alpha_i+\lambda}\right)^2\langle
f,\phi_i\rangle^2_H\rightarrow 0\,\,\text{as}\,\,\lambda\rightarrow
0\nonumber 
\end{eqnarray}
by the dominated convergence theorem. For any $\theta>0$, we have
$$\Cal{A}^2_\theta(\lambda)=\left\Vert
C^\theta(C+\lambda
I)^{-1}Cf- C^\theta f\right\Vert^2_H.$$ Let $f=f_R+f_N$ where $f_R\in
\overline{\Cal{R}(C^\theta)}$, $f_N\in
\overline{\Cal{R}(C^\theta)}^\perp$ if $0<\theta\le 1$ and $f_R\in
\overline{\Cal{R}(C)}$, $f_N\in\overline{\Cal{R}(C)}^\perp$ if $\theta\ge 1$.
Then
\begin{eqnarray}
 \Cal{A}^2_\theta(\lambda)&{}={}&\left\Vert
C^\theta(C+\lambda
I)^{-1}Cf- C^\theta f\right\Vert^2_H
=\left\Vert
C^\theta(C+\lambda
I)^{-1}Cf_R- C^\theta f_R\right\Vert^2_H\nonumber\\
&{}={}&\left\Vert \sum_i
\frac{\alpha^{1+\theta}_i}{\alpha_i+\lambda}\langle
f_R,\phi_i\rangle_H\phi_i-\sum_i \alpha^\theta_i\langle
f_R,\phi_i\rangle_H\phi_i\right\Vert^2_H\nonumber\\
&{}={}&\left\Vert \sum_i\frac{\lambda\alpha^\theta_i}{\alpha_i+\lambda}\langle
f_R,\phi_i\rangle_H\phi_i\right\Vert^2_H=\sum_i\left(\frac{
\lambda\alpha^\theta_i}{\alpha_i+\lambda}\right)^2\langle
f_R,\phi_i\rangle^2_H\rightarrow 0\nonumber 
\end{eqnarray}
as $\lambda\rightarrow 0$.\vspace{1mm}\\
(ii) If $f\in\Cal{R}(C^\beta)$, then there exists $g\in H$ such that $f=C^\beta
g$. This yields
\begin{eqnarray}
 \Cal{A}^2_\theta(\lambda)&{}={}&\left\Vert
C^\theta(C+\lambda
I)^{-1}Cf- C^\theta f\right\Vert^2_H=\left\Vert
C^\theta(C+\lambda														
I)^{-1}C^{\beta+1}g- C^{\theta+\beta} g\right\Vert^2_H\nonumber\\
&{}={}&\left\Vert
\sum_i\frac{\lambda\alpha^{\theta+\beta}_i}{\alpha_i+\lambda}\langle
g,\phi_i\rangle_H\phi_i\right\Vert^2_H=\sum_i\left(\frac{
\lambda\alpha^{\theta+\beta}_i}{\alpha_i+\lambda}\right)^2\langle
g,\phi_i\rangle^2_H.\label{Eq:twist}
\end{eqnarray}
Suppose $0<\beta+\theta< 1$. Then
$$\frac{\alpha^{\beta+\theta}_i\lambda}{\alpha_i+\lambda}=\left(\frac{\alpha_i}{
\alpha_i+\lambda}\right)^{\beta+\theta}\left(\frac{\lambda}{\alpha_i+\lambda}
\right)^{ 1-\theta-\beta}\lambda^{\beta+\theta}\le\lambda^{\beta+\theta}.$$
On the other hand, for $\beta+\theta\ge 1$, we have
$$\frac{\alpha^{\beta+\theta}_i\lambda}{\alpha_i+\lambda}=\left(\frac{\alpha_i}{
\alpha_i+
\lambda} \right)\alpha^{\beta+\theta-1}_i\lambda\le \Vert C\Vert^{\beta+\theta
-1}\lambda.$$
Using the above in (\ref{Eq:twist}) yields the result.\vspace{-5mm}
\end{proof}
\subsection{Bound on the Norm of Certain Operators and Functions}\label{subsec:certain}
The following result is used in many places throughout the paper. We would like to highlight that special cases of this result are known, e.g., see the proof of Theorem 4 in \cite{Caponnetto-07} where
concentration inequalites are obtained for the quantities in Proposition~\ref{pro:certain} using Bernstein's inequality. Here, we provide asymptotic statements using Chebyshev's inequality.
\begin{appxpro}\label{pro:certain}
Let $\Cal{X}$ be a topological space, $H$ be a separable Hilbert space and $\Cal{L}^+_2(H)$ be the space of positive, self-adjoint Hilbert-Schmidt operators
on $H$. Define $R:=\int_\Cal{X} r(x)\,d\bb{P}(x)$ and $\hat{R}:=\frac{1}{n}\sum^m_{a=1}r(X_a)$ where $\bb{P}\in M^1_+(\Cal{X})$, $(X_a)^m_{a=1}\stackrel{i.i.d.}{\sim}\bb{P}$
and $r$ is a $\Cal{L}^+_2(H)$-valued measurable function on $\Cal{X}$ satisfying $\int_\Cal{X} \Vert r(x)\Vert^2_{HS}\,d\bb{P}(x)<\infty$. Define $g_\lambda:=(R+\lambda I)^{-1}Rg$ for $g\in H$, $\lambda>0$,
and $\Cal{A}_0(\lambda):=\Vert g_\lambda-g\Vert_H$. Let $\alpha\ge 0$ and $\theta> 0$. Then the following hold:
\begin{itemize}
 \item[(i)] $\Vert (\hat{R}-R)(g_\lambda-g)\Vert_{H}=O_{\bb{P}}\left(\frac{\Cal{A}_0(\lambda)}{\sqrt{m}}\right)$.
 \item[(ii)] $\Vert R^{\alpha}(R+\lambda I)^{-\theta}\Vert\le\lambda^{\alpha-\theta}$.
 \item[(iii)] $\Vert \hat{R}^{\alpha}(\hat{R}+\lambda I)^{-\theta}\Vert\le\lambda^{\alpha-\theta}$.
 \item[(iv)] $\Vert (R+\lambda I)^{-\theta}(\hat{R}-R)\Vert=O_{\bb{P}}\left(\sqrt{\frac{1}{m\,\lambda^{2\theta}}}\right)$.
 \item[(v)] $\Vert R^\alpha (\hat{R}+\lambda I)^{-1}\Vert\lesssim \lambda^{\alpha-1}$ for $m\ge \frac{c}{\lambda^2}$ where is $c$ is a sufficiently large constant that depends on
 $\int \Vert r(x)\Vert^2_{HS}\,d\bb{P}(x)$ but not on $m$ and $\lambda$.
\end{itemize}
\end{appxpro}
\begin{proof}
(i) Note that for any $f\in H$, $$\bb{E}_\bb{P}\Vert (\hat{R}-R)f\Vert^2_{H}=\bb{E}_\bb{P}\Vert \hat{R}f\Vert^2_H+\Vert{R}f\Vert^2_H-2\bb{E}_\bb{P}\langle \hat{R}f,Rf\rangle_H,$$ 
where $\bb{E}_\bb{P}\langle \hat{R}f,Rf\rangle_H=\frac{1}{n}\sum^n_{a=1}\bb{E}_\bb{P}\langle r(X_a)f,Rf\rangle_H=\frac{1}{n}\sum^n_{a=1}\bb{E}_\bb{P}\langle r(X_a),f\otimes Rf\rangle_{HS}$.
Since $\int_\Cal{X} \Vert r(x)\Vert^2_{HS}\,d\bb{P}(x)<\infty$, $r(x)$ is $\bb{P}$-integrable in the Bochner sense (see \citealp*[Definition 1 and Theorem
2]{Diestel-77}), and therefore it follows from \citet[Theorem 6]{Diestel-77} that $\bb{E}_\bb{P}\langle r(X_a),f\otimes Rf\rangle_{HS}=\langle \int_\Cal{X} r(x)\,d\bb{P}(x),f\otimes Rf\rangle_{HS}=\Vert Rf\Vert^2_H$. Therefore,
$$\bb{E}_\bb{P}\Vert (\hat{R}-R)f\Vert^2_{H}=\bb{E}_\bb{P}\Vert \hat{R}f\Vert^2_H-\Vert{R}f\Vert^2_H,$$
where $$\bb{E}_\bb{P}\Vert \hat{R}f\Vert^2_H=\bb{E}_\bb{P}\left\Vert\frac{1}{m}\sum^m_{a=1}r(X_a)f\right\Vert^2_H=\frac{1}{m^2}\sum^m_{a,b=1}\bb{E}_\bb{P}\langle r(X_a)f,r(X_b)f\rangle_H.$$
Splitting the sum into two parts (one with $a=b$ and the other with $a\ne b$), it is easy to verify that $\bb{E}_\bb{P}\Vert \hat{R}f\Vert^2_H=\frac{1}{m}\int_\Cal{X} \Vert r(x)f\Vert^2_H\,d\bb{P}(x)+\frac{m-1}{m}\Vert Rf\Vert^2_H$, thereby
yielding \begin{eqnarray}\bb{E}_\bb{P}\Vert (\hat{R}-R)f\Vert^2_{H}=\frac{1}{m}\left(\int_\Cal{X} \Vert r(x)f\Vert^2_H\,d\bb{P}(x)-\Vert Rf\Vert^2_H\right)&{}\le{}&\frac{1}{m}\int_\Cal{X} \Vert r(x)f\Vert^2_H\,d\bb{P}(x)\nonumber\\
          &{}\le{}& \frac{\Vert f\Vert^2_H}{m}\int_\Cal{X} \Vert r(x)\Vert^2_H\,d\bb{P}(x).\nonumber
         \end{eqnarray}
Using $f=g_\lambda-g$, an application of Chebyshev's inequality yields the result.
\vspace{2mm}\\
(ii, iii) $\Vert R^{\alpha}(R+\lambda I)^{-\theta}\Vert=\sup_i\frac{\gamma^\alpha_i}{(\gamma_i+\lambda)^\theta}=\sup_i\left[\left(\frac{\gamma_i}{\gamma_i+\lambda}\right)^{\alpha}\frac{1}{(\gamma_i+\lambda)^{\theta-\alpha}}\right]
\le\sup_i\frac{1}{(\gamma_i+\lambda)^{\theta-\alpha}}\le \lambda^{\alpha-\theta}$, where $(\gamma_i)_{i\in\bb{N}}$ are the eigenvalues of $R$. (iii) follows by replacing $(\gamma_i)_{i\in\bb{N}}$
with the eigenvalues of $\hat{R}$.\vspace{2mm}\\
(iv) Since $\Vert (R+\lambda I)^{-\theta}(\hat{R}-R)\Vert\le \Vert (R+\lambda I)^{-\theta}(\hat{R}-R)\Vert_{HS}$, consider $\bb{E}_\bb{P}\Vert (R+\lambda I)^{-\theta}(\hat{R}-R)\Vert^2_{HS}$, which using the technique in the proof of (i),
can be shown to be bounded as 
\begin{equation}\label{eq:temp-aa}
\bb{E}_\bb{P}\Vert (R+\lambda I)^{-\theta}(\hat{R}-R)\Vert^2_{HS}\le \frac{1}{m}\int_\Cal{X} \Vert (R+\lambda I)^{-\theta}r(x)\Vert^2_{HS}\,d\bb{P}(x). 
\end{equation}
Note that 
\begin{eqnarray}\label{eq:temp-ab}
\Vert (R+\lambda I)^{-\theta}r(x)\Vert^2_{HS}&{}={}& \langle (R+\lambda I)^{-\theta}r(x),(R+\lambda I)^{-\theta}r(x)\rangle_{HS}\nonumber\\
&{}={}& \Vert (R+\lambda I)^{-2\theta}\Vert \text{Tr}\left(r(x) r(x)\right)=\Vert (R+\lambda I)^{-2\theta}\Vert \Vert r(x)\Vert^2_{HS}\nonumber\\
&{}\le{}&\lambda^{-2\theta}\Vert r(x)\Vert^2_{HS},
\end{eqnarray}
where the last inequality follows from (iii). Using (\ref{eq:temp-ab}) in (\ref{eq:temp-aa}), we obtain
\begin{eqnarray}\bb{E}_\bb{P}\Vert (R+\lambda I)^{-\theta}(\hat{R}-R)\Vert^2_{HS}&{}\le{}&\frac{1}{m\lambda^{2\theta}}\int_\Cal{X} \Vert r(x)\Vert^2_{HS}\,d\bb{P}(x).\nonumber	
\end{eqnarray}
The result therefore follows by an application of Chebyshev's inequality.\vspace{2mm}\\
(v) We use the idea in Step 2.1 of the proof of Theorem 4 in \cite{Caponnetto-07}, where $R^\alpha(\hat{R}+\lambda I)^{-1}$ is written equivalently as follows: Note that $\hat{R}+\lambda I=(\hat{R}-R)+(R+\lambda I)$, which implies
$$(\hat{R}+\lambda I)^{-1}=\left((\hat{R}-R)+(R+\lambda I)\right)^{-1}=(R+\lambda I)^{-1}\left(I-(R-\hat{R})(R+\lambda I)^{-1}\right)^{-1}$$ and so $R^\alpha(\hat{R}+\lambda I)^{-1}=R^\alpha(R+\lambda I)^{-1}\left(I-(R-\hat{R})(R+\lambda I)^{-1}\right)^{-1}$. Using the Von Neumann series representation, we have
$$R^\alpha(\hat{R}+\lambda I)^{-1}=R^\alpha(R+\lambda I)^{-1}\sum^\infty_{j=0}\left((R-\hat{R})(R+\lambda I)^{-1}\right)^{j}$$ so that
\begin{eqnarray}\Vert R^\alpha(\hat{R}+\lambda I)^{-1}\Vert&{}\le{}& \Vert R^\alpha(R+\lambda I)^{-1}\Vert \sum^\infty_{j=0}\Vert(R-\hat{R})(R+\lambda I)^{-1}\Vert^{j}_{HS}\nonumber\\
&{}\le{}& \lambda^{\alpha-1}\sum^\infty_{j=0}\Vert(R-\hat{R})(R+\lambda I)^{-1}\Vert^{j}_{HS}.\nonumber\end{eqnarray} From the proof of (iv), we have
that for any $\delta>0$, with probability at least $1-\delta$, 
$\Vert(R-\hat{R})(R+\lambda I)^{-1}\Vert_{HS}\le \sqrt{\frac{\int_\Cal{X} \Vert r(x)\Vert^2_{HS}\,d\bb{P}(x)}{m\lambda^2\delta}}.$ Suppose $m\ge \frac{\int_\Cal{X} \Vert r(x)\Vert^2_{HS}\,d\bb{P}(x)}{s^2\lambda^2\delta}$ where $s<1$. Then 
$\sum^\infty_{j=0}\Vert(R-\hat{R})(R+\lambda I)^{-1}\Vert^j_{HS}\le\sum^\infty_{j=0}s^j=\frac{1}{1-s}$. This means for $m\ge \frac{c}{\lambda^2}$ where $c$ is sufficiently large, we obtain
$\Vert R^\alpha(\hat{R}+\lambda I)^{-1}\Vert \lesssim\lambda^{\alpha-1}$.\vspace{-5mm}
\end{proof}
\subsection{Interpolation Space}\label{subsec:interpolation}
In this section, we briefly recall the definition of interpolation spaces of the real method. To this end,
let $E_0$ and $E_1$ be two arbitrary Banach spaces that are continuously
embedded in some topological (Hausdorff) vector space $\Cal{E}$. Then, for
$x\in E_0+E_1:=\{x_0+x_1:x_0\in E_0,\,x_1\in E_1\}$ and $t>0$, the
\emph{$K$-functional} of the real interpolation method (see \citealp*[Definition
1.1, p.~293]{Bennett-88}) is defined by
$$K(x,t,E_0,E_1):=\inf\{\Vert x_0\Vert_{E_0}+t\Vert x_1\Vert_{E_1}\,:\,x_0\in
E_0,\,x_1\in E_1,\,x=x_0+x_1\}.$$
Suppose $E$ and $F$ are two Banach spaces that satisfy $F\hookrightarrow E$
(i.e., $F\subset E$ and the inclusion operator $\text{id}:F\rightarrow E$ is
continuous), then the $K$-functional reduces to
\begin{equation}K(x,t,E,F)=\inf_{y\in F}\Vert x-y\Vert_{E}+t\Vert
y\Vert_F.\label{Eq:kfunc}\end{equation}
The $K$-functional can be used to define interpolation norms, for $0<\theta<1$,
$1\le s\le\infty$ and $x\in E_0+E_1$, as
$$\Vert
x\Vert_{\theta,s}:=\begin{cases}\left(\int
\left(t^{-\theta}K(x,t)\right)^st^{-1} \,
dt\right)^ { 1/s }, &\mbox {} 1\le s<\infty\\
\sup_{t>0}t^{-\theta}K(x,t),&\mbox{} s=\infty.\end{cases}$$
Moreover, the corresponding interpolation spaces \citep[Definition 1.7,
p.~299]{Bennett-88} are defined as
$$[E_0,E_1]_{\theta,s}:=\left\{x\in E_0+E_1\,:\,\Vert
x\Vert_{\theta,s}<\infty\right\}.$$
\section{Appendix: Miscellaneous Results}\label{Sec:supp-results}
In this appendix, we present the proofs of some claims that we made in Sections~\ref{Sec:Introduction}, \ref{Sec:Theory} and \ref{Sec:misspecified}.
\subsection{Relation between Fisher and Kullback-Leibler Divergences}\label{subsec:kl-J}
The following result provides a relationship between Fisher and Kullback-Leibler
divergences.
\begin{appxpro}\label{pro:supp-kl-J}
Let $p$ and $q$ be probability densities defined on $\bb{R}^d$. Define
$p_t:=p\ast
N(0,tI_d)$ and $q_t:=q\ast N(0,tI_d)$ where 
$N(0,tI_d)$ denotes a normal distribution on $\bb{R}^d$ with mean zero and
diagonal covariance with $t>0$. Suppose $p_t$ and $q_t$ satisfy
$$\partial_i p_t(x)\log p_t(x)=o\left(\Vert x\Vert^{\alpha}_2\right),\,\,\partial_i
p_t(x)\log q_t(x)=o\left(\Vert x\Vert^{\alpha}_2\right)\,\,\text{and}\,\,\partial_i
\log q_t(x)p_t(x)=o\left(\Vert x\Vert^{\alpha}_2\right)$$
as $\Vert x\Vert_2\rightarrow\infty$ for all $i\in[d]$ where $\alpha=1-d$.
Then 
\begin{equation}\label{Eq:id-inequality}
KL(p\Vert q)=\int^\infty_0 J(p_t\Vert q_t)\,dt,
\end{equation}
where $J$ is defined in (\ref{Eq:fisher}).
\end{appxpro}
\begin{proof}
Under the conditions mentioned on $p_t$ and $q_t$, it can be shown that
\begin{equation}
\frac{d}{dt}KL(p_t\Vert q_t)=-J(p_t\Vert q_t).
\label{Eq:identity} 
\end{equation}
See Theorem 1 in \cite{Lyu-09} for a proof. The above identity is a
simple generalization of de Bruijn's identity that relates the Fisher
information to the derivative of the Shannon entropy (see \citealp*[Theorem
16.6.2]{Cover-91}). Integrating w.r.t.~$t$ on both sides of (\ref{Eq:identity}),
we obtain $KL(p_t\Vert q_t)\Big{|}^{\infty}_{t=0}=-\int^\infty_0
J(p_t\Vert q_t)\,dt$
which yields the equality in (\ref{Eq:id-inequality}) as
$KL(p_t\Vert q_t)\rightarrow 0$ as
$t\rightarrow \infty$ and $KL(p_t\Vert q_t)\rightarrow KL(p\Vert q)$ as
$t\rightarrow 0$.\vspace{-5mm}
\end{proof}
\subsection{Estimation of $p_0$: Unbounded $k$} \label{subsec:unbounded-kernel}
To handle the case of unbounded $k$, in the
following, we assume that there exists a positive constant $M$ such that $\Vert
f_0\Vert_\eu{H}\le M$, so that an estimator of $f_0$ can be constructed as
\begin{equation}
    \breve{f}_{\lambda,n}=\arg\inf_{f\in\eu{H}}
\hat{J}_\lambda(f)\,\,\,\text{subject to
}\,\,\|f\|_\eu{H} \leq M,\label{Eq:unbounded-program}
\end{equation}
where $\hat{J}_\lambda$ is defined in Theorem~\ref{Thm:score}(iv). This modification
yields a
valid estimator $p_{\breve{f}_{\lambda,n}}$ as long as $k$ satisfies $\int_\Omega
e^{M\sqrt{k(x,x)}}q_0(x)\,dx<\infty,$ since this implies
$\breve{f}_{\lambda,n}\in\Cal{F}$. The construction of
$\breve{f}_{\lambda,n}$ requires the knowledge of $M$, however, which we assume is known
\emph{a
priori}. Using the representer theorem in RKHS, it can
be shown
(see Section~\ref{subsubsec:supp-qcqp}) that
$$\breve{f}_{\lambda,n}=\breve{\delta}\hat{\xi}+\sum^n_{b=1}
\sum^d_ { j=1 } \breve{\beta}_{(b-1)d+j}
\partial_j k(X_b,\cdot)$$
where $\breve{\delta}$ and $\breve{\bm{\beta}}$ are obtained by solving the
following quadratically constrained quadratic program (QCQP),
\begin{equation}
(\breve{\bm{\beta}},\breve{\delta})=:\breve{\Theta}=\arg\min_{\Theta\in\bb{R}^{nd+1}
} \frac { 1}{2}\Theta^T \bm{H}\Theta+\Theta^T\Delta\,\,\,\text{subject to
}\,\,\Theta^T\bm{K}\Theta\le M^2,\nonumber
\end{equation}
with $\Delta:=(\bm{h},\|\hxi\|^2_\eu{H})$, $\Theta:=
(\bm{\beta},\delta)$
and $\bm{K}$, $\bm{H}$ being defined in the proof of Theorem~\ref{Thm:representer} and the remark following it.
The following result investigates the consistency and convergence rates for $p_{\breve{f}_{\lambda,n}}$. 
\begin{appxthm}[Consistency and rates for
$p_{\breve{f}_{\lambda,n}}$]\label{Thm:density-2}
Let $M \ge \|f_0\|_\eu{H}$ be a fixed constant, and $\breve{f}_{n,\lambda}$ be a
clipped estimator given by (\ref{Eq:unbounded-program}). 
Suppose \emph{\textbf{(A)}}--\emph{\textbf{(D)}}
with $\varepsilon=2$ hold. Let $\emph{supp}(q_0)=\Omega$ and $\int_\Omega
e^{M\sqrt{k(x,x)}}q_0(x)\,dx<\infty$. Define
$\eta(x)=\sqrt{k(x,x)}e^{M\sqrt{k(x,x)}}$. Then, as $\lambda \sqrt{n}
\rightarrow \infty,\,\lambda\rightarrow
0\,\,\text{and}\,\,n\rightarrow \infty$,
\begin{itemize}
 \item[(i)] $\Vert p_{\breve{f}_{\lambda,n}}-p_0\Vert_{L^1(\Omega)}\rightarrow
0$, $KL(p_0\Vert p_{\breve{f}_{\lambda,n}})\rightarrow 0$ if $\eta \in
L^1(\Omega,q_0)$;
\item[(ii)] for $1<r\le\infty$, $\Vert
p_{\breve{f}_{\lambda,n}}-p_0\Vert_{L^r(\Omega)}\rightarrow 0$ if $\eta q_0\in
L^1(\Omega)\cap L^r(\Omega)$ and $e^{M\sqrt{k(\cdot,\cdot)}}q_0\in
L^r(\Omega)$;
\item[(iii)] $h(p_{\breve{f}_{\lambda,n}},p_0)\rightarrow 0$ if
$\sqrt{k(\cdot,\cdot)}\eta \in L^1(\Omega,q_0)$;
\item[(iv)] $J(p_0\Vert p_{\breve{f}_{\lambda,n}})\rightarrow 0$.
\end{itemize}
In addition, if $f_0\in\Cal{R}(C^\beta)$
for some $\beta>0$, then $\Vert p_{
\breve{f}_{\lambda,n}}-p_0\Vert_{L^r(\Omega)}=O_{p_0}(\theta_n),\,h(p_0,p_{
\breve{f}_
{\lambda,n}})=O_{p_0}(\theta_n),\,KL(p_0\Vert p_{\breve{f}_{\lambda,n}})=O_{p_0}
(\theta_n)\,\,\text{and}\,\,J(p_0\Vert p_{\breve{f}_{\lambda,n}})=O_{p_0}
(\theta^2_n)$
where $\theta_n:=n^{-\min\left\{\frac{1}{4},\frac {\beta}{2(\beta+1)}\right\}}$
with $\lambda=n^{-\max\left\{\frac{1}{4},\frac{1}{2(\beta+1)}\right\}}$
assuming the respective conditions in (i)-(iii) above hold.
\end{appxthm}
\begin{proof}
For any $x\in\Omega$, since $|f_0(x)|\le \Vert f_0\Vert_\eu{H}\sqrt{k(x,x)}\le
M\sqrt{k(x,x)}$ and 
$|\breve{f}_{\lambda,n}(x)|\le M\sqrt{k(x,x)}$, we have 
\begin{equation}\label{Eq:ttt}
    \bigl|e^{\breve{f}_{\lambda,n}(x)}-e^{f_0(x)}\bigr| 
    \leq e^{M\sqrt{k(x,x)}}\bigl| \breve{f}_{\lambda,n}(x)- f_0(x)\bigr|  
  \leq \eta(x)\bigl\| \breve{f}_{\lambda,n}-
f_0\bigr\|_\eu{H},
\end{equation}
where we used the fact that $|e^x-e^y|\le e^a|x-y|$ for $x,y\in [-a,a]$
and $\eta(x):=\sqrt{k(x,x)}e^{M\sqrt{k(x,x)}}$. In the following, we obtain
bounds for $\bigl\| p_{\breve{f}_{\lambda,n}}-p_{0}\bigr\|_{L^r(\Omega)}$ for
any $1\le r\le \infty$, $h(p_{\breve{f}_{\lambda,n}},p_{0})$ and
$KL(p_{0}\Vert p_{\breve{f}_{\lambda,n}})$ in terms of $\|\breve{f}_{\lambda,n}-
f_0\bigr\|_\eu{H}$. Define $B(f):=\int_\Omega e^f
q_0\,dx$. Since $k$ satisfies $\int_\Omega e^{M\sqrt{k(x,x)}}q_0(x)\,dx<\infty$, then
it is
clear that $\breve{f}_{\lambda,n}\in\Cal{F}$ as $B(\breve{f}_{\lambda,n})<\infty$ since $$\int_\Omega
e^{\breve{f}_{\lambda,n}(x)}q_0(x)\,dx\le \int_\Omega
e^{\Vert\breve{f}_{\lambda,n}\Vert_\eu{H}\sqrt{k(x,x)}}q_0(x)\,dx\le \int_\Omega
e^{M\sqrt{k(x,x)}}q_0(x)\,dx<\infty.$$ Similarly, it is easy to verify that 
$B(f_0)<\infty$.\vspace{1mm}\\
(i) Recalling (\ref{Eq:list}), we have
\begin{equation}\bigl\| p_{\breve{f}_{\lambda,n}}-p_{0}\bigr\|_{L^r(\Omega)} \le
\frac{|B(\breve{f}_{\lambda,n})-B(f_0)|\Vert
e^{f_0}q_0\Vert_{L^r(\Omega)}}{B(\breve{f}_{\lambda,n})B(f_0)}
+\frac{\Vert
(e^{f_0}-e^{\breve{f}_{\lambda,n}})q_0\Vert_{L^r(\Omega)}}{B(\breve{f}_{
\lambda,n})}.\nonumber\end{equation}
If $r=1$, we obtain
\begin{equation}\bigl\| p_{\breve{f}_{\lambda,n}}-p_{0}\bigr\|_{L^1(\Omega)} \le
\frac{|B(\breve{f}_{\lambda,n})-B(f_0)|}{B(\breve{f}_{\lambda,n})}
+\frac{\Vert
(e^{f_0}-e^{\breve{f}_{\lambda,n}})q_0\Vert_{L^1(\Omega)}}{B(\breve{f}_{
\lambda,n})}.\nonumber
\end{equation}
Using (\ref{Eq:ttt}), we bound $|B(\breve{f}_{\lambda,n}) - B(f_0) |$ as
\begin{equation*}\label{Eq:L1_convergence}
|B(\breve{f}_{\lambda,n}) - B(f_0) | 
 \leq \int_\Omega \bigl|e^{\breve{f}_{\lambda,n}(x)}-e^{f_0(x)}\bigr|q_0(x)\,dx 
 \leq \|\eta\|_{L^1(\Omega,q_0)} \bigl\| \breve{f}_{\lambda,n}-
f_0\bigr\|_\eu{H}. 
\end{equation*}
Also for
any $f\in\eu{H}$ with $\Vert f\Vert_\eu{H}\le M$, we have $B(f)\ge \int_\Omega
e^{-M\sqrt{k(x,x)}}q_0(x)\,dx=:\theta$, where $\theta>0$. Again
using (\ref{Eq:ttt}), we have
$$\Vert
(e^{f_0}-e^{\breve{f}_{\lambda,n}})q_0\Vert_{L^r(\Omega)}\le \Vert \eta
q_0\Vert_{L^r(\Omega)}\Vert \breve{f}_{\lambda,n}-f_0\Vert_\eu{H}$$
and $\Vert e^{f_0}q_0\Vert_{L^r(\Omega)}\le \Vert
e^{M\sqrt{k(x,x)}}q_0\Vert_{L^r(\Omega)}$.
Therefore,
\begin{eqnarray}\bigl\| p_{\breve{f}_{\lambda,n}}-p_{0}\bigr\|_{L^r(\Omega)}
&{}\le{}&
\frac{\|\eta\|_{L^1(\Omega,q_0)} \Vert
e^{M\sqrt{k(x,x)}}q_0\Vert_{L^r(\Omega)}\bigl\| \breve{f}_{\lambda,n}-
f_0\bigr\|_\eu{H}}{\theta^2}\nonumber\\
&{}{}&\qquad\qquad+\frac{\Vert \eta
q_0\Vert_{L^r(\Omega)}\Vert \breve{f}_{\lambda,n}-f_0\Vert_\eu{H}}{\theta}\nonumber
\end{eqnarray}
and for $r=1$, $$\bigl\| p_{\breve{f}_{\lambda,n}}-p_{0}\bigr\|_{L^1(\Omega)} \le
\frac{2\,\|\eta\|_{L^1(\Omega,q_0)} \bigl\| \breve{f}_{\lambda,n}-
f_0\bigr\|_\eu{H}}{\theta}.$$
(ii) Also \begin{eqnarray}KL(p_0\Vert p_{\breve{f}_{\lambda,n}})&{}={}&\int_\Omega
p_0\log\frac{p_0}{p_{\breve{f}_{\lambda,n}}}\,dx=\int_\Omega \log
 \left(e^{f_0-\breve{f}_{\lambda,n}}\frac{B(\breve{f}_{\lambda,n})}{
B(f_0) }\right)p_0(x)\,dx\nonumber\\
&{}={}&\int_\Omega
 \left(f_0-\breve{f}_{\lambda,n}+\log\frac{B(\breve{f}_{\lambda,n})}{
B(f_0)}\right)p_0(x)\,dx\nonumber\\
&{}\le{}&\frac{|B(\breve{f}_{\lambda,n})-B(f_0)|}{B(f_0)}+\Vert\breve{f}_{\lambda,n}-f_0\Vert_{L^1(\Omega,p_0)}
\le\frac{2\,\Vert \eta
q_0\Vert_{L^1(\Omega)}}{\theta}\Vert
\breve{f}_{\lambda,n}-f_0\Vert_\eu{H}.\nonumber
\end{eqnarray}
(iii) It is easy to verify that 
\begin{eqnarray}h(p_{\breve{f}_{\lambda,n}},p_0)&{}={}&\left\Vert
\frac{e^{\breve{f}_{\lambda,n}/2}}{\Vert e^{\breve{f}_{\lambda,n}/2}
\Vert_{L^2(\Omega,q_0)}}-\frac{e^{f_0/2}}{\left\Vert e^{f_0/2}
\right\Vert_{L^2(\Omega,q_0)}}\right\Vert_{ L^2(\Omega , q_0) }\nonumber\\
&{}\le{}&
\frac{2\Vert e^{\breve{f}_{\lambda,n}/2}-e^{f_0/2}
\Vert_{L^2(\Omega,q_0)}}{\left\Vert e^{f_0/2}
\right\Vert_{L^2(\Omega,q_0)}}\nonumber
\end{eqnarray}
where the above inequality is obtained by carrying out and simplifying the
decomposition as in (\ref{Eq:list}). Using (\ref{Eq:ttt}), we therefore have
$$h(p_{\breve{f}_{\lambda,n}},p_0)\le\sqrt{\frac{\int_\Omega
k(x,x)e^{M\sqrt{k(x,x)}}q_0\,dx}{\theta}}\Vert
\breve{f}_{\lambda,n}-f_0\Vert_\eu{H}.$$
(iv) As $f_0,\,\breve{f}_{\lambda,n}\in\Cal{F}$, by Theorem~\ref{Thm:score},
we obtain
$J(p_0\Vert p_{\breve{f}_{\lambda,n}})=\frac{1}{2}\Vert
\sqrt{C}(\breve{f}_{\lambda,n}-f_0)\Vert^2_\eu{H}\le
\frac{1}{2}\Vert\sqrt{C}\Vert^2\Vert
\breve{f}_{\lambda,n}-f_0\Vert^2_\eu{H}.$
\vspace{1mm}\\
Note that we have bounded the various distances between
$p_{\breve{f}_{\lambda,n}}$ and $p_0$ in terms of $\Vert
\breve{f}_{\lambda,n}-f_0\Vert_\eu{H}$. Since
$\breve{f}_{\lambda,n}=f_{\lambda,n}$ with probability converging to 1,
the assertions on consistency are proved by Theorem~\ref{Thm:rates}(i) in
combination with Lemma~\ref{lem:support}---as we did not explicitly assume
$f_0\in\overline{\Cal{R}(C)}$---and the rates follow from
Theorem~\ref{Thm:rates}(iii).\vspace{-5mm}
\end{proof}
\begin{rem}
The following observations can be made while comparing the scenarios of using bounded vs.~unbounded
kernels in the problem of estimating $p_0$ through Theorems~\ref{Thm:density}
and \ref{Thm:density-2}. First, the consistency results in $L^r$, Hellinger and
KL distances are the same but for additional integrability conditions on $k$
and $q_0$. The additional integrability
conditions are not too difficult to hold in practice as they involve $k$ and
$q_0$ which can be chosen appropriately. However, the unbounded situation in
Theorem~\ref{Thm:density-2} requires the knowledge of $M$ which is usually not
known. On the other hand, the consistency result in $J$ in
Theorem~\ref{Thm:density-2} is slightly weaker than in
Theorem~\ref{Thm:density}. This may be an artifact of our analysis as
we are not able to adapt the bounding technique used in the proof of
Theorem~\ref{Thm:density} to bound
$J(p_0\Vert p_{\breve{f}_{\lambda,n}})=\frac{1}{2}\Vert
\sqrt{C}(\breve{f}_{\lambda,n}-f_0)\Vert^2_\eu{H}$ as it critically depends on
the boundedness of $k$. Therefore, we used a trivial
bound of $J(p_0\Vert p_{\breve{f}_{\lambda,n}})=\frac{1}{2}\Vert
\sqrt{C}(\breve{f}_{\lambda,n}-f_0)\Vert^2_\eu{H}\le
\frac{1}{2}\Vert\sqrt{C}\Vert^2\Vert
\breve{f}_{\lambda,n}-f_0\Vert^2_\eu{H}$, which yields the result through
Theorem~\ref{Thm:rates}(i). Due to the same reason, we also obtain
a slower rate of convergence in $J$. Second, the rate of convergence in KL is
slower than in Theorem~\ref{Thm:density-2}, which again may be an artifact of
our analysis. The convergence rate for KL in Theorem~\ref{Thm:density} is based
on the application of Theorem~\ref{Thm:rates}(ii) in
Lemma~\ref{lem:distances}, where the bound on KL in
Lemma~\ref{lem:distances} 
critically
uses the boundedness to upper bound KL in terms of squared Hellinger distance. 
\end{rem}
\subsubsection{Derivation of $\breve{f}_{\lambda,n}$}\label{subsubsec:supp-qcqp}
Any $f\in\eu{H}$ can be decomposed as
$f=f_{\Vert}+f_{\perp}$ where $$f_\Vert\in\text{span}\left\{\hat{\xi},
(\partial_j
k(X_b,\cdot))_{b,j}\right\}=:\eu{H}_\Vert,$$ which is a closed subset of $\eu{H}$
and $f_\perp\in \eu{H}^\perp_\Vert:=\left\{g\in \eu{H}:\langle g,h
\rangle_\eu{H}=0,\,\forall\,h\in\eu{H}_\Vert\right\}$ so that
$\eu{H}=\eu{H}_\Vert\oplus \eu{H}^\perp_\Vert$. Since the objective function in
(\ref{Eq:unbounded-program}) matches with the one in
Theorem~\ref{Thm:representer}, using the above decomposition in
(\ref{Eq:unbounded-program}), it is easy to verify that $\hat{J}$ depends only
on $f_\Vert\in\eu{H}_\Vert$ so that (\ref{Eq:unbounded-program}) reduces to
\begin{equation}
(\breve{f}^\Vert_{\lambda,n},\breve{f}^\perp_{\lambda,n})=\arg\inf_{
\substack{f_\Vert\in
\eu{H}_\Vert, f_\perp\in\eu{H}^\perp\\ \Vert
f_\Vert\Vert^2_\eu{H}+\Vert
f_\perp\Vert^2_\eu{H} \leq M^2}}
\hat{J}_\lambda(f_\Vert)+\frac{\lambda}{2}\Vert
f_\Vert\Vert^2_\eu{H}+\frac{\lambda}{2}\Vert
f_\perp\Vert^2_\eu{H}\label{Eq:unbounded-program-1}
\end{equation}
and
$\breve{f}_{\lambda,n}=\breve{f}^\Vert_{\lambda,n}+\breve{f}^\perp_{\lambda,n}
$. Since $f_\Vert$ is of the form in (\ref{Eq:rep}), using it in
(\ref{Eq:unbounded-program-1}), it is easy to show that
$\hat{J}_\lambda(f_\Vert)+\frac{\lambda}{2}\Vert
f_\Vert\Vert^2_\eu{H}=\frac{1}{2}\Theta^T
\bm{H}\Theta+\Theta^T\Delta$. Similarly, it can be shown that $\Vert
f_\Vert\Vert^2_\eu{H}=\Theta^T
\bm{K}\Theta$. Since $f_\perp$
appears in (\ref{Eq:unbounded-program-1}) only through $\Vert
f_\perp\Vert^2_\eu{H}$, (\ref{Eq:unbounded-program-1}) reduces to
\begin{equation}
 (\Theta_\Vert,c_\perp)=\arg\inf_{\substack{\Theta\in \bb{R}^{nd+1}, c\ge
0\\\Theta^T\bm{K}\Theta+c\le M^2} }
\frac{1}{2}\Theta^T
\bm{H}\Theta+\Theta^T\Delta+\frac{\lambda}{2}c,\label{Eq:unbounded-program-2}
\end{equation}
where $\breve{f}^\Vert_{\lambda,n}$ is constructed as in
(\ref{Eq:rep}) using
$\Theta_\Vert$ and $\breve{f}^\perp_{\lambda,n}$ is such that
$\Vert\breve{f}^\perp_{\lambda.n}\Vert^2_\eu{H}=c_\perp$. The necessary and
sufficient conditions for the optimality of $(\Theta_\Vert,c_\perp)$ is given
by the following Karush-Kuhn-Tucker conditions,
\begin{eqnarray}
 (\bm{H}+2\tau
\bm{K})\Theta_\Vert+\Delta=0,\,\,\frac{\lambda}{2}+\eta-\tau=0\qquad\text{
(Stationarity) } \nonumber\\
\Theta^T_\Vert \bm{K}\Theta_\Vert+c_\perp\le M^2,\,\,c_\perp\ge
0\qquad\text{(Primal feasibility)}\nonumber\\
\eta\ge 0,\,\,\tau\ge 0\qquad\text{(Dual feasibility)}\nonumber\\
\tau c_\perp=0,\,\,\eta(\Theta^T_\Vert \bm{K}\Theta_\Vert+c_\perp-
M^2)=0\qquad\text{(Complementary slackness)}\nonumber
\end{eqnarray}
Combining the dual feasibility and stationary conditions, we have
$\eta=\tau-\frac{\lambda}{2}\ge 0$, i.e., $\tau\ge\frac{\lambda}{2}$. Using
this in the complementary slackness involving $\tau$ and $c_\perp$, it follows
that $c_\perp=0$. Since $\Vert \breve{f}^\perp_{\lambda,n}\Vert^2=c_\perp$, we
have $\breve{f}^{\perp}_{\lambda,n}=0$, i.e., $\breve{f}_{\lambda,n}$ is
completely determined by $\breve{f}^\Vert_{\lambda,n}$. Therefore
$\breve{f}^\Vert_{\lambda,n}$ is of the form in (\ref{Eq:rep})
and (\ref{Eq:unbounded-program-2}) reduces to a quadratically constrained
quadratic program.
\subsection{$\Cal{R}(C^\beta)$ and Interpolation Spaces}\label{sec:interpolate}
Proposition~\ref{Thm:interpolation} presents an interpretation for
$\Cal{R}(C^\beta)$ ($\beta>0$ and $\beta\notin\bb{N}$) as interpolation spaces
between $\Cal{R}(C^{\lceil\beta\rceil})$ and
$\Cal{R}(C^{\lfloor\beta\rfloor})$
where $\Cal{R}(C^0):=\eu{H}$. An inspection
of its proof shows that Proposition~\ref{Thm:interpolation} holds for any self-adjoint,
bounded, compact operator defined on a separable Hilbert space. 
\begin{appxpro}\label{Thm:interpolation}
Suppose \textbf{\emph{(B)}} and \textbf{\emph{(D)}} hold with $\varepsilon=1$.
Then for
all $\beta>0$ and $\beta\notin\bb{N}$
$$\Cal{R}(C^\beta)=\left[\Cal{R}(C^{\lfloor\beta\rfloor}),\Cal{R}(C^{\lceil
\beta\rceil}) \right]_{\beta-\lfloor\beta\rfloor,2}$$
where $\Cal{R}(C^0):=\eu{H}$, and the spaces
$\Cal{R}(C^\beta)$ and $\left[\Cal{R}(C^{\lfloor\beta\rfloor}),\Cal{R}(C^{\lceil
\beta\rceil}) \right]_{\beta-\lfloor\beta\rfloor,2}$ have equivalent norms.
\end{appxpro}
To prove Proposition~\ref{Thm:interpolation}, we need the following result 
which we quote from \citet[Lemma 6.3]{Steinwart-12} (also
see
\citealp*[Lemma 23.1]{Tartar-07}) that interpolates $L^2$-spaces whose underlying
measures are absolutely continuous with respect to a measure $\nu$.
\begin{appxlem}\label{lem:tartar}
Let $\nu$ be a measure on a measurable space $\Theta$ and
$w_0:\Theta\rightarrow [0,\infty)$ and $w_1:\Theta\rightarrow [0,\infty)$ be
measurable functions. For $0<\beta<1$, define $w_\beta:=w^{1-\beta}_0w^\beta_1$.
Then we have
$$[L^2(w_0\,d\nu),L^2(w_1\,d\nu)]_{\beta,2}=L^2(w_\beta\,d\nu)$$
and the norms on these two spaces are equivalent. Moreover, this result still
holds for weights $w_0:\Theta\rightarrow (0,\infty)$ and $w_1:\Theta\rightarrow
[0,\infty]$, if one uses the convention $0\cdot\infty:=0$ in the definition of
the weighted spaces.
\end{appxlem}

\noindent \emph{Proof of Proposition~\ref{Thm:interpolation}.}
The proof is based on the ideas used in the proof of Theorem 4.6 in
\cite{Steinwart-12}. Recall that by the Hilbert-Schmidt theorem, $C$ has the
following
representation,
$$C=\sum_{i\in I}\alpha_i\phi_i\langle \phi_i,\cdot\rangle_\eu{H}$$
where $(\alpha_i)_{i\in I}$ are the positive eigenvalues of
$C$, $(\phi_i)_{i\in I}$ are the corresponding unit eigenvectors that form an
ONB for $\Cal{R}(C)$ and $I$ is an index set which is either finite (if $\eu{H}$
is finite-dimensional) or $I=\bb{N}$ with $\lim_{i\rightarrow\infty}\alpha_i=0$
(if $\eu{H}$ is infinite dimensional). Let $(\psi_i)_{i\in J}$ be an ONB for
$\Cal{N}(C)$ where $J$ is some index set so that any $f\in\eu{H}$ can be written
as
$$f=\sum_{i\in I}\langle f,\phi_i\rangle_\eu{H}\phi_i+\sum_{i\in J}\langle
f,\psi_i\rangle_\eu{H}\psi_i=:\sum_{i\in I\cup J}a_i\theta_i$$
where $\theta_i:=\phi_i$ if $i\in I$ and $\theta_i:=\psi_i$ if $i\in J$ with
$a_i:=\langle f,\theta_i\rangle_\eu{H}$. Let $\beta> 0$. By definition,
$g\in\Cal{R}(C^\beta)$ is equivalent to $\exists h\in\eu{H}$ such that
$g=C^\beta h$, i.e.,
$$g=\sum_{i\in I}\alpha^\beta_i\langle
h,\phi_i\rangle_\eu{H}\phi_i=:\sum_{i\in I}b_i\alpha^\beta_i\phi_i$$
where $b_i:=\langle h,\phi_i\rangle_\eu{H}$. Clearly $\sum_{i\in I}
b^2_i=\sum_{i\in I}\langle h,\phi_i\rangle^2_\eu{H}\le \Vert
h\Vert^2_\eu{H}<\infty$, i.e., $(b_i)\in\ell_2(I)$. Therefore
$$\Cal{R}(C^\beta)=\left\{\sum_{i\in
I}b_i\alpha^\beta_i\phi_i\,:\,(b_i)\in\ell_2(I)\right\}=\left\{\sum_{i\in
I}c_i\phi_i\,:\,(c_i)\in\ell_2(I,\alpha^{-2\beta})\right\}$$
where $\alpha:=(\alpha_i)_{i\in I}$. Let us equip this space with the
bilinear form
$$\left\langle \sum_{i\in
I}c_i\phi_i,\sum_{i\in I}d_i\phi_i\right\rangle_{\Cal{R}(C^\beta)}:=\left\langle
(c_i),(d_i)\right\rangle_{\ell_2(I,\alpha^{-2\beta})}$$ so that it induces the
norm
$$\left\Vert \sum_{i\in
I}c_i\phi_i\right\Vert_{\Cal{R}(C^\beta)}:=\left\Vert
(c_i)\right\Vert_{\ell_2(I,\alpha^{-2\beta})}.$$ It is easy to verify that
$(\alpha^\beta_i\phi_i)_{i\in I}$ is an ONB of $\Cal{R}(C^\beta)$. Also since
$\Cal{R}(C^{\beta_1})\subset\Cal{R}(C^{\beta_2})$ for $0<\beta_2<\beta_1<\infty$
and $\text{id}:\Cal{R}(C^{\beta_1})\rightarrow \Cal{R}(C^{\beta_2})$ is
continuous, i.e., for any $g\in \Cal{R}(C^{\beta_1})$,
\begin{eqnarray}\Vert g\Vert_{\Cal{R}(C^{\beta_2})}=\Vert
(c_i)\Vert_{\ell_2(I,\alpha^{-2\beta_2})}=\sqrt{\sum_{i\in
I}\frac{c^2_i}{\alpha^{2\beta_2}_i}}&{}\le{}&\sup_{i\in
I}|\alpha_i|^{\beta_1-\beta_2}
\Vert(c_i)\Vert_{\ell_2(I,\alpha^{-2\beta_1})}\nonumber\\
&{}={}&\Vert
C\Vert^{\beta_1-\beta_2}\Vert g\Vert_{\Cal{R}(C^{\beta_1})}<\infty\nonumber\end{eqnarray}
and so $\Cal{R}(C^{\beta_1})\hookrightarrow\Cal{R}(C^{\beta_2})$. Similarly, we
can show that $\Cal{R}(C)\hookrightarrow \eu{H}$. In the following, we first
prove the result for $0<\beta<1$ and then for $\beta>1$. 

(a) \underline{$0<\beta<1$:} For any
$f\in\eu{H}$ and $g\in\Cal{R}(C)$, we have
$$\Vert f-g\Vert^2_\eu{H}=\left\Vert\sum_{i\in I\cup J}a_i\theta_i-\sum_{i\in
I}c_i\phi_i\right\Vert^2_\eu{H}=\left\Vert\sum_{i\in I\cup
J}(a_i-c_i)\theta_i\right\Vert^2_\eu{H}=\Vert(a_i-c_i)\Vert^2_{\ell_2(I\cup
J)}$$ where we define $c_i:=0$ for $i\in J$. For $t>0$, we find
\begin{eqnarray}K(f,t,\eu{H},\Cal{R}(C))&{}={}&\inf_{g\in\Cal{R}(C)}\Vert
f-g\Vert_\eu{H}+t\Vert g\Vert_{\Cal{R}(C)}\nonumber\\&{}={}&\inf_{(c_i)\in
\ell_2(I,\alpha^{-2})}\Vert (a_i-c_i)\Vert_{\ell_2(I\cup J)}+t\Vert
(c_i)\Vert_{\ell_2(I,\alpha^{-2})}\nonumber\\
&{}={}& K(a,t,\ell_2(I\cup J),\ell_2(I,\alpha^{-2})).\nonumber
\end{eqnarray}
From this we immediately obtain the equivalence
$$f\in [\eu{H},\Cal{R}(C)]_{\beta,2}\,\,\Longleftrightarrow\,\,(a_i)\in
[\ell_2(I\cup J),\ell_2(I,\alpha^{-2})]_{\beta,2}$$
where $0<\beta<1$. Applying the second part of Lemma~\ref{lem:tartar} to the
counting measure on $I\cup J$ yields
$$[\ell_2(I\cup
J),\ell_2(I,\alpha^{-2})]_{\beta,2}=\ell_2(I,\alpha^{-2\beta}).$$ Since
$\Cal{R}(C^\beta)$ and $\ell_2(I,\alpha^{-2\beta})$ are isometrically
isomorphic, we obtain $\Cal{R}(C^\beta)=[\eu{H},\Cal{R}(C)]_{\beta,2}$.

(b) \underline{$\beta>1$ and $\beta\notin\bb{N}$:} Define $\gamma:=\lfloor\beta\rfloor$. Let
$f\in\Cal{R}(C^\gamma)$ and $g\in\Cal{R}(C^{\gamma+1})$, i.e.,
$\exists\,(c_i)\in\ell_2(I,\alpha^{-2\gamma})$ and
$(d_i)\in\ell_2(I,\alpha^{-2\gamma-2})$ such that $f=\sum_{i\in I}c_i\phi_i$
and $g=\sum_{i\in I}d_i\phi_i$. Since 
$$\Vert f-g\Vert^2_{\Cal{R}(C^\gamma)}=\Vert
(c_i-d_i)\Vert^2_{\ell_2(I,\alpha^{-2\gamma})},$$ for $t>0$, we have
\begin{eqnarray}K(f,t,\Cal{R}(C^\gamma),\Cal{R}(C^{\gamma+1}))&{}={}&\inf_{
g\in\Cal { R } (C^{\gamma+1})} \Vert
f-g\Vert_{\Cal{R}(C^\gamma)}+t\Vert
g\Vert_{\Cal{R}(C^{\gamma+1})}\nonumber\\
&{}={}&\inf_{(d_i)\in
\ell_2(I,\alpha^{-2\gamma-2})}\Vert
(c_i-d_i)\Vert_{\ell_2(I,\alpha^{-2\gamma})}+t\Vert
(d_i)\Vert_{\ell_2(I,\alpha^{-2\gamma-2})}\nonumber\\
&{}={}&
K(c,t,\ell_2(I,\alpha^{-2\gamma}),\ell_2(I,\alpha^{-2\gamma-2})),\nonumber
\end{eqnarray}
from which we obtain the following equivalence
$$f\in
[\Cal{R}(C^\gamma),\Cal{R}(C^{\gamma+1})]_{\beta-\gamma,2}\,\,
\Longleftrightarrow\, \, (c_i)\in
[\ell_2(I,\alpha^{-2\gamma}),\ell_2(I,\alpha^{-2\gamma-2})]_{\beta-\gamma,2}
\stackrel{(\ast)}{=}\ell_2(I,\alpha^{-2\beta}),$$
where $(\ast)$ follows from Lemma~\ref{lem:tartar} and the result is obtained
by noting that $\ell_2(I,\alpha^{-2\beta})$ and $\Cal{R}(C^\beta)$ are
isometrically isomorphic.\QEDA
\subsection{Denseness of $I_k\eu{H}$ in $W_2(\bb{R}^d,p)$}\label{subsec:supp-dense}
In this section, we discuss the denseness of $I_k\eu{H}$ in $W_2(\bb{R}^d,p)$ 
for a given $p\in \Cal{P}_{\text{FD}}$, where $\Cal{P}_{\text{FD}}$ is defined
in
Theorem~\ref{Thm:approx-fd}, which is equivalent to the injectivity of $S_k$
(see \citealp*[Theorem 4.12]{Rudin-91}). To this end, in the following result we
show that under certain conditions on a bounded continuous translation invariant
kernel on $\bb{R}^d$, the restriction of $S_k$ to $\Cal{W}^\sim_2(\bb{R}^d,p)$ is
injective when $d=1$, while the result for any general $d>1$ is open. However,
even for $d=1$, this does not guarantee the injectivity of $S_k$ (which is
defined on $W_2(\bb{R}^d,p)$). Therefore, the question of characterizing the injectivity
of $S_k$ (or equivalently the denseness of $I_k\eu{H}$ in $W_2(\bb{R}^d,p)$) is open.
\begin{appxpro}\label{pro:dense}
Suppose $k(x,y)=\psi(x-y),\,x,y\in\bb{R}^d$ where $\psi\in C_b(\bb{R}^d)\cap
L^1(\bb{R}^d)$, $\int \Vert\omega\Vert_2\psi^\wedge(\omega)\,d\omega<\infty$
and $\text{\emph{supp}}(\psi^\wedge)=\bb{R}^d$. If $d=1$, then the restriction
of $S_k$ to $\Cal{W}^\sim_2(\bb{R}^d,p)$ is injective for 
any $p\in\Cal{P}_{\text{FD}}$.
\end{appxpro}
\begin{proof}
Fix any $p\in\Cal{P}_{\text{FD}}$. 
We need to show that for $[f]_\sim\in\Cal{W}^\sim_2(\bb{R}^d,p)$, $S_k[f]_\sim=0$
implies $[ f]_\sim=0$. 
From Proposition~\ref{pro:operators}, we have
\begin{eqnarray}
S_k[f]_\sim&{}={}&\int_{\bb{R}^d}\sum^d_{j=1}\partial_j k(x,\cdot)\partial_jf(x)\,p(x)\,dx= \int_{\bb{R}^d}\sum^d_{j=1}\partial_j \psi(x-\cdot)\partial_jf(x)\,p(x)\,dx\nonumber\\
&{}={}& \int_{\bb{R}^d}\sum^d_{j=1}\frac{1}{(2\pi)^{d/2}}\int_{\bb{R}^d}
i\omega_j\psi^\wedge(\omega)e^{i\langle
\omega,\cdot-x\rangle}\,d\omega\,\partial_jf(x)\,p(x)\,dx\nonumber\\
&{}={}&\int_{\bb{R}^d}\sum^d_{j=1}\phi_j(\omega)
\psi^\wedge(\omega)e^{i\langle
\omega,\cdot\rangle}\,d\omega\nonumber
\end{eqnarray}
where $$\phi_j(\omega):=\frac{1}{(2\pi)^{d/2}}\int_{\bb{R}^d} (i\omega_j)
e^{-i\langle\omega,x\rangle}\partial_jf(x)\,p(x)\,dx.$$ $S_kf=0$
implies $\sum^d_{j=1}\phi_j(\omega)
\psi^\wedge(\omega)=0$ for all $\omega\in\bb{R}^d$. Since
$\text{supp}(\psi^\wedge)=\bb{R}^d$, we have $\sum^d_{j=1}\phi_j(\omega)
=0$ a.e., i.e., for $\omega$-a.e., 
\begin{equation}0=\sum^d_{j=1}\int_{\bb{R}^d}(i\omega_j)\partial_jf(x)\,p(x)\,e^{
-i\langle\omega,x\rangle}\,dx
=\sum^d_{j=1}(i\omega_j)(p\partial_jf)^\wedge(\omega).\nonumber
\end{equation}
For $d=1$, this implies $(\partial_jf)p=0$ a.e.~and so $\Vert f\Vert_{W_2}=0$.
\end{proof}
Examples of kernels that
satisfy the conditions in Proposition~\ref{pro:dense} include the
Gaussian, Mat\'{e}rn (with $\beta>1$) and inverse multiquadrics on $\bb{R}$.

\end{appendices}

\section*{Acknowledgments}
Theorem~\ref{thm:representer}
and the proof are due to an anonymous reviewer. We thank Dougal Sutherland for a careful reading of the paper which helped to fix minor errors.
A part of the work was carried out while BKS was a Research Fellow in the Statistical
Laboratory, Department of Pure Mathematics and Mathematical Statistics,
University of Cambridge. BKS thanks
Richard Nickl for many valuable comments and insightful discussions. KF is supported in part by
MEXT Grant-in-Aid for Scientific Research on Innovative Areas 25120012.
\bibliography{Reference}

\end{document}